\documentclass[a4paper,reqno,11pt]{amsart}
\usepackage[left=1 in, right=1 in,top=1.2 in, bottom=1.2 in]{geometry}
\usepackage{amsfonts}% Mathematical package，include amssymb
\usepackage{amssymb} % Mathematical Greek notation
\usepackage{amsthm} % A environment of proof and theorem
\usepackage{amsmath} % Output of mathematical formula
\usepackage{mathrsfs} % for mathbf,mathbf,mathcal etc.
\usepackage{color}% provides both foreground (text, rules, etc.) and background color management
\usepackage{enumerate}%gives the enumerate environment an optional argument that determines the style in which the counter is printed.
\usepackage{hyperref}% used to handle cross-referencing commands in LaT X to produce hypertext links in the document
\usepackage[numbers,sort&compress]{natbib}%不但可以压缩参考文献标号，还可以进行排序，即无论正文里面的顺序怎样，显示出来都是先后顺序。在elsevier 模板中，natbib包已经默认引用了，无需重新引用，改一下natbib的选项即可，设置方法如下：\biboptions{numbers,sort&compress}
\usepackage{pdfsync} % Provide links between source and PDF
\usepackage{esint}%permits access to alternate integral symbols when you are using the Computer Modern fonts. In the original set, several integral symbols are missing, such as \oiint.
\usepackage{graphicx}%providing a key-value interface for optional arguments to the \includegraphics command.
\usepackage{float}%For insetting pictures
\usepackage{epstopdf}
\usepackage{caption}%For insetting pictures
\usepackage{subfigure}%For insetting pictures
\usepackage{url}% For the site of reference
\usepackage{stmaryrd} % for the interleave
\usepackage{dsfont}% \mathbb{C}
\usepackage{enumitem}
\usepackage{bm}% put word in {} can black it.
\usepackage{setspace}% distance between lines in the whole paper,
\usepackage{inputenc} %  for a content
\newtheorem{theorem}{Theorem}[section]
\newtheorem{definition}{Definition}[section]
\newtheorem{lemma}{Lemma}[section]
\newtheorem{remark}{Remark}[section]

\newtheorem{corollary}{Corollary}[section]
\numberwithin{equation}{section}
\renewcommand{\d}{\operatorname{d}}
\newcommand{\A}{\operatorname{A}}
\newcommand{\pl}{\parallel}
\renewcommand{\u}{\mathbf{u}}

\newcommand{\ls}{\leqslant}

\newcommand{\odiv}{\operatorname{div}}
\newcommand{\ra}{\rightarrow}

\newcommand{\eps}{\varepsilon}
\renewcommand{\u}{{\bf u}}

\renewcommand{\v}{{\bf{v}}}

\newcommand{\z}{{\bf{z}}}

\newcommand{\R}{{\mathbb R}}
\allowdisplaybreaks[4]

\date{\today}%%%%%%%%%%%%%%%%%%%%%%%%%%%换成默认的时间，否则显示的是中文时间。
\begin{document}
\title[Mild Solution to IINS]{$L^{p}$ Mild Solution to Stochastic Incompressible inhomogeneous Navier-Stokes Equations}

\author{Yachun Li$^{1}$, Ming Mei$^{2,3,4}$, Lizhen Zhang$^{*,5}$}
\dedicatory{ {\footnotesize\it $^1$School of Mathematical Sciences, CMA-Shanghai, MOE-LSC, Shanghai Jiao Tong University, 200240, China}\\
{\footnotesize\it $^2$School of Mathematics and Statistics, Jiangxi Normal University, Nanchang, 330022, China} \\
{\footnotesize\it $^3$Department of Mathematics, Champlain College Saint-Lambert,
	    Saint-Lambert, Quebec, J4P 3P2, Canada} \\
{\footnotesize\it $^4$Department of Mathematics and Statistics, McGill University,
	      Montreal, Quebec, H3A 2K6, Canada  }\\
{\footnotesize\it $^5$School of Mathematical Sciences, Shanghai Jiao Tong University,
	     Shanghai, 200240, China }\\
{\tt Emails: ycli@sjtu.edu.cn; ming.mei@mcgill.ca; Rainbowsun2023@outlook.com;}
}
%\thanks{Yachun Li is supported by [National Natural Science Foundation of China under grants 12371221, 12161141004, and 11831011.] The research of M. Mei was supported by [NSERC grant RGPIN 2022-03374 and NNSFC Grant W2431005.]}

\begin{abstract}
In this paper, we establish the global $L^{p}$ mild solution of inhomogeneous incompressible Navier-Stokes equations in the torus $\mathbb{T}^{N}$ with $N<p<6$, $1\ls N\ls 3$, driven by the Wiener Process. We introduce a new iteration scheme coupled the density $\rho$ and the velocity $\u$ to linearize the system, which defines a semigroup. Notably, unlike semigroups dependent merely on $x$, the generators of this semigroup depend on both time $t$ and space $x$. After demonstrating the properties of this time- and space-dependent semigroup, we prove the local existence and uniqueness of mild solution, employing the semigroup theory and Banach's fixed point theorem. Finally, we show the global existence of mild solutions by Zorn's lemma. Moreover, for the stochastic case, we need to use the operator splitting method to do some estimates separately.
\end{abstract}

%\date{\today}

% 35A01 Existence problems for PDEs: global existence, local existence, non-existence; 35Q30,  Navier-Stokes equations For fluid mechanics, 35R60, PDEs with randomness, stochastic partial differential equations;35B10 Periodic solutions to PDEs


\maketitle
%\begin{spacing}{0.94}
 {\small \tableofcontents}%%Smallize the size to make it contained in one page
 \setcounter{tocdepth}{2}

\section{Introduction}

The system of inhomogeneous incompressible Navier-Stokes equations describes the motion of the incompressible ideal fluids with the velocity $\u=(u_{1},u_{2},\cdots, u_{N})\in \R^{N}$ and the density $\rho$. However, physically, the motion of fluid may be influenced by some stochastic factors like wind, the action of the machine, industrial pollution, etc. To study the stochastic problem also helps us to discover which kind of noise might have minimal detrimental effects on the well-posedness of solutions. Sometimes, researchers specifically investigate types of noise that enhance or support the well-posedness of solutions.
While perturbed by random external force, the stochastic inhomogeneous incompressible Navier-Stokes equations (SIINS for short) is proposed as follows
\begin{equation}
\left\{\begin{array}{l}\label{sto inhomo incom NS}
\rho_{t}+\odiv \left(\rho\u\right) =0,\\
\rho \d \u+\left(\rho (\u\cdot \nabla \u)+\nabla P-\mu \triangle \u\right)\d t=\rho \Phi \d W,\\
\odiv \u=0,
\end{array}\right.
\end{equation}
 in $[0,T]\times \mathbb{T}^{N}$, where ``$\d$" is the differential notation,
 $P$ is the pressure, and the positive constant $\mu$ is the viscosity coefficient. %We also assume that density $\rho$ may change around the equilibrium of density $\bar{\rho}>0$.
  $\rho \Phi\d W$ is the external additive noise. $W$ is a cylindrical $\mathcal{F}_{t}$-adapted Wiener process in the stochastic basis $\left(\Omega,\mathcal{F},\mathbb{P}\right)$, where $\mathcal{F}=\left(\mathcal{F}_{t}\right)_{t\geqslant 0}$ is the right-continuous filtration. Let $\{e_{k}\}_{k=1}^{+\infty}$ be an orthonormal basis in an auxiliary separable Hilbert space $\mathcal{H}$, which is isometrically isomorphic to $l^{2}$, the space of square-summable sequences. $\mathcal{H}$ is independent of domain $\mathbb{T}^{N}$. $W$ is the cylindrical Wiener process in the form of
\begin{equation}
W=\sum\limits_{k=1}^{+\infty}e_{k}\beta_{k}, \quad \d W=\sum\limits_{k=1}^{+\infty}e_{k}\d \beta_{k},
\end{equation}
where $\beta_{k}$ is the standard real Brownian motion.
Let $\mathcal{V}$ be a Bochner space. In the random forcing term, $\Phi$ is a linear bounded operator from $ \mathcal{H}$ to $ \mathcal{H}$, to characterize the covariance of the noise in a specific volume,
\begin{equation}
\begin{split}
&\Phi: \mathcal{H}\ra\mathcal{V}= C\left([0, T]; L^{p}\left(\mathbb{T}^{N}\right)\right).
\end{split}
\end{equation}
 $\Phi$ acts on $e_{k}$ in $\mathcal{H}$ as
\begin{equation}
\left(\Phi(t,x),~e_{k}\right)_{\mathcal{H}} = \Phi_{k}\left(t,x\right), \quad \Phi \d W=\sum\limits_{k=1}^{+\infty} \Phi e_{k}\d \beta_{k}(t)=\sum\limits_{k=1}^{+\infty}\Phi_{k}\d \beta_{k}(t) e_{k},
\end{equation}
where $\Phi_{k}\left(t,x\right)$ is an $N$-dimensional $\mathcal{V}$-valued vector function. $\rho(t,x,\omega)$ and $\u(t,x,\omega)$ are actually stochastic processes of $t$, $x$ and $\omega$, where $\omega\in \Omega$ is the sample. We will write $\rho$ and $\u$ for short.
For system \eqref{sto inhomo incom NS}, the initial and boundary conditions are given by
\begin{align}\label{assum for boundary data}
\rho(0,x)=\rho_{0}(x),\quad \u(0,x)=\u_{0}(x), %\quad \u(t,x)\mid_{\partial \mathbb{T}^{N}}=0.
\end{align}
We assume that
\begin{align}\label{assum for initial data}
0<c_{0} \ls \rho_{0}(x)\ls C_{0}, \quad \odiv \u_{0}(x)=0, \quad \mathbb{P} ~{\rm a.s.}
\end{align}
where $c_{0}$ and $C_{0}$ are constants.

%Note that $\odiv \u=0$ is the incompressible condition when $\rho$ changes in the continuity equation. Thus, the original continuity equation is
%\begin{equation}
%\rho_{t}+ \operatorname{div}\left(\rho \u\right)=0.
%\end{equation}
%Multiplying it by $-\frac{1}{\rho^{2}}$, then the equation is transformed into
%\begin{equation}
%\frac{\bar{D}\tau}{\d t}=\tau\odiv \u,
%\end{equation}
% where $\tau=\frac{1}{\rho}$ is the specific volume and $\bar{D}=\frac{\partial }{\partial t}+ \u\cdot \nabla$ is the material derivative. Therefore, $\odiv \u=0$ represents $\frac{\overline{D}\tau}{\d t}=0$, which represents the changing rate of volume \cite{Li-Zhu2015}. Besides, we call the system homogeneous if the density is a constant.
 %%%%%
When $\Phi \d W=0$, \eqref{sto inhomo incom NS} reduces to the deterministic incompressible inhomogeneous Navier-Stokes equations (IINS).
 For $\rho$ being strictly positive, Kazhikov \cite{Kazhikhov1974Solvability} in 1974 obtained the global existence of Hopf solution to the initial and boundary value problems of the deterministic IINS, with $\nabla P\in L^{2}(U)$. Then in 1978, Ladyzhenskaya-Solonnikov \cite{Ladyzh-Solo1978} demonstrated the global existence of strong solutions $\u\in H^{1}\left(0, T; H^{2}\left(U\right)\right)$, $\nabla P \in L^{q}\left((0, T)\times U\right)$, $q>N$, $N=2,3$, $\rho\in C^{1}\left(\left(0, T\right)\times U\right)$ in a bounded domain $U$. They also proved the uniqueness of the initial and boundary value problem with large initial density data.
 These results have since been extended by various authors. J. Simon \cite{Simon1978,Simon1989,Simon1990} allowed $\rho_{0}$ to vanish, while maintaining a constant viscosity $\mu$. In 2003, Choe-Kim \cite{Choe-Kim2003} studied the global existence of strong solutions and the uniqueness with the high regularity of initial data for the initial and boundary value problem, as well as the uniqueness in \cite{Cho-Kim2004}. For other works concerning the global existence of strong solutions, we refer to \cite{Itoh-Tani1999,Germain2008,Danchin2006,Salvi2003}, among which Danchin \cite{Danchin2006} also gave the uniqueness. In 1990, Padula \cite{Padula1990OnTE} showed the existence and uniqueness of weak solutions in exterior domains with a specific form of radial $\rho_{0}$. In 1992, Fern\'andez--Cara--Guill\'en \cite{Fernandez-Cara-Guillen1992} got the existence of weak solutions in unbounded domains.

 For the IINS with degenerate viscosities, Diperna-Lions \cite{DiPerna1988-1989,DiPerna-Lions-Globalsolutions1989} proved the global existence of weak solutions to the initial and boundary value problem for the Dirichlet case, the periodic case, and the whole space case.
For the strong solutions with vacuum, the global well-posedness is uniquely solved \cite{Zhang-Zhao-Zhang2014,Zhang2015,Zhang-Shi-Cao2022}, provided that $\rho_{0}$ is smooth enough and $\u_{0}$ satisfies the compatibility condition
\begin{align}\label{compatibility condition}
-\triangle \u_0+\nabla P_0=\sqrt{\rho_0} g, \text { for some } g \in L^2(U) \text { and } P_0 \in H^1(U).
\end{align}
In 2018, based on the local existence and uniqueness obtained by Li \cite{JingLi2017JDElocal}, without the compatibility condition \eqref{compatibility condition},
 Danchin-Mucha \cite{Danchin-Mucha2018} proved the global existence and uniqueness of strong solutions in a $C^{2}$ bounded subset or a torus. Their analysis was based on an initial velocity in $H^{1}$ and assumed that the density is nonnegative and bounded. For the large time behavior, one can refer to \cite{Lyu-Shi-Zhong2018,He-Li-Lv2021,Liu2021}. For the studies in Besov space, one can further refer to \cite{Danchin2003,Abidi2007,Danchin-Mucha2009,Farwig-Qian-Zhang2020,ZhifeiZhangCPDE,Danchin-Wang2023}.
% obtained the global existence of weak solutions $(\rho, \u)$ under the boundedness of initial density in 2020, by setting $\u$ in a Besov space, and they showed that the uniqueness of the solution could be obtained by improving the regularity of initial velocity.

For the mild solutions of homogeneous incompressible Navier-Stokes equations, early in 1962, Kato-Fujita \cite{Kato-Fujita1962,Kato-Fujita1964} studied the global existence of mild solutions in 3-D and 2-D space in Hilbert space. In \cite{Kato-Fujita1962}, for the system
\begin{align}\label{Kato equation}
\frac{\d \u}{\d t} = \triangle \u + \mathcal{P} f(t)- \mathcal{P}( \u\cdot \nabla )\u ,
\end{align}
where $f(t)$ is the external force and $\mathcal{P}$ is the Leray projector, they defined the mild solution as
\begin{align}
\u(t)=e^{t\triangle }\u_{0}+\int_0^t e^{(t-s) \triangle} \mathcal{P} f(s) \d s+\int_0^t e^{(t-s) \triangle} \mathcal{P}( \u\cdot \nabla )\u(s) \d s.
\end{align}
If $\u_0 \in\left(H^1\left(\mathbb{R}^3\right)\right)^3$ and $f \in L^2\left(0, T; \left(L^2\left(\mathbb{R}^3\right)^3\right)\right)$, then there exists a $T_0 \in$ $(0, T)$ and a mild solution $\u$ of \eqref{Kato equation}, in form of
\begin{align}
\u(t)=e^{t \triangle }\u_{0}+\int_0^t e^{(t-s) \triangle} \mathcal{P} f(s) \d s+\int_0^t \triangle^{\frac{1}{4}} e^{(t-s) \triangle} \triangle^{-\frac{1}{4}}\mathcal{P}( \u\cdot \nabla )\u(s)\d s,
\end{align}
on $\left(0, T_0\right) \times \mathbb{R}^3$, such that $\u \in C\left(\left[0, T_0\right],\left(H^1(\mathbb{R}^3)\right)^3\right) \cap L^2\left(0, T_0;\left(H^2(\mathbb{R}^3)\right)^3\right)$. Then, in 1964, for $\u_{0}\in \dot{H}^{\frac{1}{2}}(U)$, they \cite{Kato-Fujita1964} proved the global existence and uniqueness of solutions in $C\left(\left(T_{1}, T_{2}\right]; \dot{H}^{\frac{1}{2}}(U)\right)$, $T_{1}>0$, with the smallness of $\|\u_{0}\|_{\dot{H}^{\frac{1}{2}}}$.
 Swann \cite{Swann1971TAMS} and Kato \cite{Kato1972JFA} individually showed that the Cauchy problem for the Navier-Stokes equation in $\mathds{R}^3$ admits a unique classical local solution $\u\in C\left([0, T];H^{3}\left(\mathds{R}^3\right)\right)$ for $\u_{0}\in H^{3}\left(\mathds{R}^3\right)$. Scholars aim at proving the uniqueness in weaker space for the deterministic INS.
In 1984, Kato \cite{Kato1984} developed the global existence and uniqueness of solutions in
$$\u \in C\left([0, +\infty]; L^{3}(\mathds{R}^{3})\right) \cap L^{5}\left((0, +\infty)\times \mathds{R}^3\right)$$
with the smallness of $\|\u_{0}\|_{L^{3}}$.
In 1992, Chemin \cite{Chemin1992} established the existence of solutions in $C([0, +\infty); H^{1/2}) \cap L^2([0, +\infty); H^{3/2}) \cap C((0, +\infty); H^s) \cap L^\infty((0, +\infty); H^{s+1})$, under the assumption that the external forces is in $C([0, +\infty); H^{1/2}) \cap L^2([0, +\infty); H^{3/2}) \cap C((0, +\infty); H^s) \cap L^\infty((0, +\infty); H^{s+1})$, where $s\geqslant 0$.
In 1994, Cannone-Meyer-Planchon \cite{Cannone1993-1994} showed that there exists a global unique solution, provided that $\|\u_0\|_{\dot{E}_{1, m}}$ is small, $\forall ~ m\geqslant 1$. The homogeneous space $\dot{E}_{q, m}$ is defined by
\begin{align}
&\dot{E}_{q, m}=\left\{f \left| \left|\partial^\alpha f(x)\right| \leq C|x|^{-q-|\alpha|}, \quad|\alpha| \leq m \right.\right\},\\
&\|f\|_{\dot{E}_{q, m}} = \sup_{|x| = 1} |D^{\alpha} f(x)|.
\end{align}
 They extended the well-posedness of solutions to that in Chemin-Lerner space $L^{\infty}\left(0, T ; \dot{B}^{0}_{3, \infty}\right)$ as well.
In 2001, Lions-Masmoudi \cite{Lions-Masmoudi2001} proved the uniqueness of solutions in the space
$$C\left(\left[0, T\right); L^N\left(\mathbb{R}^N\right)\right) \cap\left\{\u | t^{1 / 4} \u(t) \in C\left(\left[0, T\right) ; L^{2 N}\left(\mathbb{R}^N\right)\right)\right\}.$$
Koch-Tataru \cite{Koch-Tataru2001} developed the global well-posedness of solutions in $BMO^{-1}$ space
$$\sup _{x, 0<R<\sqrt{T}}\left(|B(x, R)|^{-1} \int_{B(x, R)} \int_0^{R^2}|u(t, y)|^2 \d t \d y\right)^{\frac{1}{2}}<\infty ,$$
 under the smallness of initial velocity.
After that, Lei-Lin \cite{Lei-Lin2011} gave the global solutions for the initial velocity in $\mathcal{X}^{-1}=\left\{f \in \mathcal{D}^{\prime}\left(\mathds{R}^3\right): \int_{\mathbb{R}^3}|\xi|^{-1}|\hat{f}| \d \xi<\infty\right\}$. Then in 2015, Bae \cite{Bae2015} developed an alternative proof of existence of these solutions in $\mathcal{X}^{-1}$, which allowed for the estimate of the radius of analyticity at positive times. In 2021, Ambrose-Filho-Lopes \cite{Ambrose-Filho-Lopes} proved the existence of the $\mathcal{X}^{-1}$ solution in the spatially periodic setting with an improved bound of the radius of analyticity.
Very recently, Miller \cite{Miller2021} gave the conditions on the regularity when strong solutions are close to mild solutions.

Regarding the mild solutions to deterministic IINS, Danchin-Mucha \cite{Danchin-Mucha2012} introduced the Lagrangian approach and proved the global well-posedness including the uniqueness, under the following assumptions
\begin{align}\label{assumptions in Danchin-Mucha2012}
\left\|\rho_0-1\right\|_{\mathcal{M}\left(\dot{B}_{p, 1}^{\frac{N}{p}-1}\right)}+\left\|u_0\right\|_{\dot{B}_{p, 1}^{\frac{N}{p}-1}} \leqslant \eps,\quad p\ls N,
\end{align}
for some small $\eps$. Here $\mathcal{M}\left(\dot{B}_{p, 1}^{\frac{N}{p}-1}\right)$ denotes the multiplier space of $\dot{B}_{p, 1}^{\frac{N}{p}-1}$. Paicu-Zhang-Zhang \cite{ZhifeiZhangCPDE} and Chen-Zhang-Zhao \cite{Chen-Zhang-Zhao2016} relaxed the smallness condition \eqref{assumptions in Danchin-Mucha2012}. They \cite{ZhifeiZhangCPDE,Chen-Zhang-Zhao2016} established the global unique solvability and decay estimates by the weighted energy method and dual method, under the assumption:
\begin{itemize}
  \item $N=2: 0<c \leqslant \rho_0 \leqslant C, ~\u_0 \in H^s\left(\mathbb{R}^2\right)$, $s>0$; \cite{ZhifeiZhangCPDE}
  \item $N=3: 0<c \leqslant \rho_0 \leqslant C, ~\u_0 \in H^1\left(\mathbb{R}^3\right)$ and $\left\|\u_0\right\|_{L^2}\left\|\nabla \u_0\right\|_{L^2} \leqslant \varepsilon$ with $\varepsilon$ sufficiently small. \cite{ZhifeiZhangCPDE}
  \item $N=3$: $\left(\rho_0, \u_0\right) \in L^{\infty}\left(\mathbb{R}^3\right) \times H^s\left(\mathbb{R}^3\right)$ with $s>\frac{1}{2}$ satisfies
$$
0<c_0 \leqslant \rho_0 \leq C_0<+\infty, \quad\left\|\u_0\right\|_{\dot{H}^{\frac{1}{2}}} \leqslant \varepsilon
$$
for some small $\varepsilon>0$ depending only on $c_0, C_0$. \cite{Chen-Zhang-Zhao2016}
\end{itemize}
P. Zhang \cite{Zhang2020ADVFK} and Danchin-Wang \cite{Danchin-Wang2023} obtained the global well-posedness and uniqueness in critical Besov space. Abidi-Gui-Zhang \cite{Abidi-Gui-Zhang2024} removed the assumption that the initial density is close enough to a positive constant in \cite{Danchin-Wang2023} yet with additional regularities on the initial density. 
Very recently, Hao-Shao-Wei-Zhang \cite{Hao-Shao-Wei-Zhang2024} established the global existence and uniqueness of Fujita-Kato solutions for the 3-D case with initial vacuum. Additionally, they proved the global well-posedness of solutions for the 2-D case with initial vacuum. In \cite{Hao-Shao-Wei-Zhang2024}, the weighted energy estimates and Lagrangian approach are employed. Moreover, the initial layer is considered carefully in the above deterministic works, due to the setting of the weighted solutions space.

For homogeneous incompressible Navier-Stokes equations driven by Gaussian noise, in 1995, Benssousan \cite{Bensoussan1995} derived the model of the stochastic Navier-Stokes equations from the motion of the physical particle. He also established the existence of weak solutions in an $N$-D bounded domain. For the global existence in 1-D and 2-D periodic domain, we refer to Tornatore \cite{Tornatore2000} and Yashima \cite{Yashima2001}, respectively. Du-Zhang \cite{Du-Zhangting2020} established local existence and uniqueness of the strong solution when the initial data take values in the critical Besov space, for system forced by a multiplicative white noise, as well as the global existence for small initial data. Recently, Kukavica-Xu \cite{Kukavica2024} proved the almost global existence result for small data in $L^{3}$ space, under natural smallness conditions on the noise. 

For SIINS, by Galerkin approximation and energy estimates, Yashima \cite{Yashima1992} did the earliest study of the global martingale solutions when the system is influenced by addictive Gaussian noise, with no vacuum in the initial density. In 2006, Cutland-Enright \cite{Cutland-Enright2006} gave the existence of strong solutions ($N\ls 3$)  driven by multiplicative noise in a bounded domain, with no vacuum in the initial density.
In 2010, Sango \cite{Sango2010} obtained the existence of weak solutions for multiplicative noise, where the coefficient function of $\rho$ and $\u$ is unnecessarily Lipschitz continuous. They showed that their approach works for the case that the initial density is with vacuum. In the above stochastic results, the energy method is employed instead of semigroup. In 2019, Chen-Wang-Wang \cite{Chen-Wang-Wang.2019} obtained the global martingale solutions to SIINS with L\'evy noise.

For the mild solutions of the homogeneous incompressible Navier-Stokes equations driven by fractional Brownian motion $W^{\mathscr{H}}$ with Hurst parameter $0<\mathscr{H}<1$, the works of mild solutions \cite{Tindal-Tudor-Viens,Fang-Sundar-Viens,Duncan-Maslowski2002,Duncan-Maslowski2006,Ferrario-Olivera}, are based on the semigroup theory of Laplacian operator.

We will consider the global existence and uniqueness of the mild solution for SIINS system \eqref{sto inhomo incom NS} by semigroup method, instead of the weighted energy estimate. We first give the definition of the mild solution to \eqref{sto inhomo incom NS} here.
\begin{definition}
Let $\left(\Omega,\mathcal{F},\mathbb{P}\right)$ be a fixed complete probability space with the right continuous filtration $\mathcal{F}=\left(\mathcal{F}_{s}\right)_{s\geqslant 0}$. Let $W(t)$ be the fixed cylindrical Wiener process. A pair of stochastic processes $\left(\rho,\u\right)$ is called a mild solution to system \eqref{sto inhomo incom NS}-\eqref{assum for boundary data}-\eqref{assum for initial data}, if:
\begin{enumerate}
\item $\left(\rho,\u\right)$ is adapted in $\mathcal{F}=\left(\mathcal{F}_{s}\right)_{s\geqslant 0}$;
\item  $\mathbb{P}[\{(\rho(t),\u(0))=(\rho_{0},\u_{0})\}]=1$;
\item The mass equation
\begin{align}
 \rho(t)=\rho_{0}-\int_0^t \odiv \left(\rho\u\right) \d s,
\end{align}
 holds $\mathbb{P}$ a.s., for any $t\in [0,T]$;
\item The momentum equation
\begin{align}\label{Solution in semigroup convolution}
\u(t)=&S(t)\u_{0}-\int_0^t S(t-s)\u\cdot \nabla \u(s) \d s-\int_0^t S(t-s) \nabla P(s)\d s\\
&+\int_0^t S(t-s) \Phi \d W(s),\notag
\end{align}
holds $\mathbb{P}$ a.s., for any $t\in [0,T]$, where $S(t)$ is a semigroup specified later. \eqref{Solution in semigroup convolution} is understood as the limits of an iteration scheme.
\end{enumerate}
\end{definition}
We will use the property of semigroup to get the result, in which a new iteration scheme and some technical estimates are carried out.
 Our result on the mild solution to stochastically forced systems \eqref{sto inhomo incom NS} goes as follows.
\begin{theorem}\label{main theorem}
Let $1\ls N\ls 3$, $N< p \ls 6$. We assume that \eqref{assum for initial data} holds. Given $\rho_{0}\in W^{2,p}\left(\mathbb{T}^{N}\right) \cap H^{3}\left(\mathbb{T}^{N}\right)$, and $\u_{0}\in W^{2,p}\left(\mathbb{T}^{N}\right) \cap H^{3}\left(\mathbb{T}^{N}\right)$ $\mathbb{P}$ {\rm a.s.} in the probability space $\left(\Omega, \mathcal{F}, \mathbb{P}\right) $. If, for any fixed $T>0$,
\begin{align}\label{C_Phi intro}
C_{\Phi}=\max\limits_{t\in[0,T]}\left\{\sum\limits_{k=1}^{\infty} \left\|\Phi_{k}\right\|_{\infty}^{2}, ~\sum\limits_{k=1}^{\infty} \left\|\nabla \Phi_{k}\right\|_{\infty}^{2}, ~\sum\limits_{k=1}^{\infty} \left\|\triangle \Phi_{k}\right\|_{\infty}^{2}, ~\sum\limits_{k=1}^{\infty} \left\|\nabla \triangle \Phi_{k}\right\|_{\infty}^{2}\right\}\ls C\ls \infty,
\end{align}
 then there exists a unique global mild solution to \eqref{sto inhomo incom NS}: \\
  $\rho\in C\left([0, T];  L^{p}\left(\mathbb{T}^{N}\right)\right)$, $\u\in C\left([0, T];  L^{p}\left(\mathbb{T}^{N}\right)\right)$, $\mathbb{P}$ {\rm a.s.} in  $\left(\Omega,\mathcal{F},\mathbb{P} \right) $. \\
   Besides, there exists a global energy solution to \eqref{sto inhomo incom NS}: \\
   $\rho\in C\left([0, T];  C^{1,\ell}\left(\mathbb{T}^{N}\right)\right)$, $\u\in C\left([0,T]; H^{3}\left(\mathbb{T}^{N}\right)\right)\cap L^{2}\left(0,T; H^{4}\left(\mathbb{T}^{N}\right)\right)$ $\mathbb{P}$ {\rm a.s.} in $\left(\Omega, \mathcal{F},\mathbb{P}\right) $.
\end{theorem}

%
%\begin{remark}
%A requirement of $N\leqslant 3$ is specified in the uniform estimate of $\u$, as detailed in Section \ref{priori estimate of v}.
%\end{remark}

We begin by deforming the equations.
We denote $a=\frac{\rho-\bar{\rho}}{\bar{\rho}}$.  Due to \eqref{assum for initial data}, the transported density $\rho$ is bounded away uniformly from zero and has an upper bound. Without loss of generality, we heuristically assume that $|a|\ls \frac{1}{2}$, i.e., $\frac{1}{2} \ls 1+a \ls \frac{3}{2} < 2$. System \eqref{sto inhomo incom NS} becomes
\begin{equation}\label{new form inhomo incom NS}
\left\{\aligned
& a_{t}+\u \cdot \triangledown a =0,\\
& \left(1+a\right)\d \u+ \left(\left(1+a\right)\u\cdot\nabla \u-\frac{\mu}{\bar{\rho}}\triangle \u +\frac{\nabla P}{\bar{\rho}}\right)\d t= \Phi \d W (t),\\
& \odiv \u=0,
\endaligned
\right.
\end{equation}
subjected to the following initial and boundary value conditions
\begin{equation}\label{initial and boundary conditions}
 a\mid_{t=0}=a_{0},  \quad \u\mid_{t=0}=\u_{0}(x). %\quad \u\mid_{\partial \mathbb{T}^{N}}=0,
\end{equation}
We assume that the pressure $P$ is the function of $\rho$ and $\u$; and we denote
\begin{align}
\frac{\nabla P}{\bar{\rho}\left(1+a\right)}=\frac{\nabla P}{\rho}\triangleq \nabla Q\left(a,\u\right).
\end{align}
Then \eqref{new form inhomo incom NS} is reduced to
\begin{equation}\label{new one form inhomo incom NS}
\left\{\aligned
& a_{t}+\mathbf{u} \cdot \triangledown a =0,\\
& \d \u +\left(\u\cdot\nabla \u-\frac{\mu}{\bar{\rho}\left(1+a\right)}\triangle \u + \nabla Q\left(a,\u\right)\right)\d t = \Phi \d W (t),\\
& \odiv \u=0.
\endaligned
\right.
\end{equation}
% The idea is to divide the problem into two parts: the initial boundary problem to the evolution equation with Brownian motion,
% \begin{equation}\label{pl1}
%\left\{\aligned
%&\d \z(t)- \frac{\mu}{\bar{\rho}\left(1+a\right)}\triangle \z(t)\d t= \Phi \d W(t),\\
%& \odiv \z=0,\\
%& \z\mid_{t=0}=0, \quad a\mid_{t=0}=a_{0}, \quad\z\mid_{\partial \mathbb{T}^{N}}=0,
%\endaligned
%\right.
%\end{equation}
%and the initial boundary problem to the deterministic nonlinear evolution equation
%\begin{equation}\label{pn2}
%\left\{\aligned
%&\partial_{t}\v(t) - \frac{\mu}{\bar{\rho}\left(1+a\right)}\triangle\v(t) = -\left(\v(t)+\z(t)\right)\cdot\nabla\left(\v(t)+\z(t)\right)- \nabla Q\left(a,\u\right), \\
%& \odiv \v=0, \\
%&\v\mid_{t=0}=\u_0, \quad a\mid_{t=0}=a_{0},\quad \v\mid_{\partial \mathbb{T}^{N}}=0.
%\endaligned
%\right.
%\end{equation}
%Therefore, the solution $\u$ of \eqref{new form sto inhomo incom NS} is expressed as $\u=\v+\z$.

For the stochastic forced problem, the $r$-th moment estimate of the convective term is no longer uniformly bounded within a ball in view of Banach's fixed point theorem. This challenge is caused not only by the stochastic integral, but also by the varying density and the nonlinear term. To solve this, we avoid estimating the nonlinear term in the stochastic moment sense, i.e., we estimate the nonlinear term $\mathbb{P}$ a.s.
The idea is to divide the problem into two parts: the initial and boundary problem to an evolution equations with Wiener process,
 \begin{equation}\label{pl1}
\left\{\aligned
&\d \z(t)- \frac{\mu}{\bar{\rho}\left(1+a\right)}\triangle \z(t)\d t= \Phi \d W (t),\\
%& \odiv \z=0,\\
& \z\mid_{t=0}=0, \quad a\mid_{t=0}=a_{0}, %\quad\z\mid_{\partial \mathbb{T}^{N}}=0,
\endaligned
\right.
\end{equation}
and the initial and boundary value problem to the deterministic nonlinear evolution equations
\begin{equation}\label{pn2}
\left\{\aligned
&\v_{t} - \frac{\mu}{\bar{\rho}\left(1+a\right)}\triangle\v = -\left(\v+\z\right)\cdot\nabla\left(\v+\z\right)- \nabla Q\left(a,\v
+\z\right), \\
%& \odiv \v=0, \\
&\v\mid_{t=0}=\u_0, \quad a\mid_{t=0}=a_{0}. %\quad \v\mid_{\partial \mathbb{T}^{N}}=0.
\endaligned
\right.
\end{equation}
Therefore, the solution $\u$ of \eqref{sto inhomo incom NS} is expressed as $\u=\v+\z$. Then, with the $\mathbb{P}$ a.s. boundedness of $\|\z\|_{H^{3}}$, we obtain that the norm of $\v$ is bounded $\mathbb{P}$ a.s. Next, we utilise Banach's fixed point theorem to prove the local existence. Finay, we derive the \emph{a priori} estimates to get the global existence.

The difficulties and strategies for the well-posedness of solution to SIINS are further analysed as follows.
\begin{enumerate}
  \item {\bf A new linearizing scheme.} The usual method to deal with deterministic inhomogeneous incompressible equations is linearizing the momentum equation by substituting $\rho $ with $\bar{\rho} $ in the momentum equation, where $\bar{\rho}$ is a constant. Under this substitution, the semigroup generator $\frac{ \mu \triangle}{\bar{\rho}}$ is a Laplacian operator up to a constant. Consequently, it suffices to deal the remaining term $\left(\rho-\bar{\rho}\right) \u_{t}$.
However, this method does not work for the semigroup method while it works for energy solution. We would like to propose a new linearizing scheme to gain the mild solution of the inhomogeneous system \eqref{new one form inhomo incom NS} by the semigroup method. Since $\rho$ is not a constant, we treat the velocity as known in the following iteration of mass conservation equation:
\begin{equation}
\partial_{t} a_{(n)} + \u_{(n-1)}\cdot\nabla a_{(n)} = 0,
\end{equation}
with initial condition
\begin{equation}
a_{(n)}\mid_{t=0}=a_{0},\quad a_{(0)}=0, \quad \u_{(0)}=\u_{0}.
\end{equation}
Meanwhile, we need to linearize the momentum equation as follows
\begin{equation}\label{deterministic linearized system}
\left\{\aligned
& \partial_{t}\u_{(n)}(t) - \frac{\mu}{\bar{\rho}\left(1+a_{(n)}\right)}\triangle\u_{(n)}(t)\\
& =-\left(\u_{(n-1)}(t)\right)\cdot\nabla\left(\u_{(n-1)}(t)\right)- \nabla Q\left(a_{(n)},\u_{(n-1)}\right)+\Phi \d W (t), \\
& \odiv \u_{(n)}=0, \quad \u_{(n)}\mid_{t=0}=\u_{0}, %\quad \u_{(n)}\mid_{\partial \mathbb{T}^{N}}=0.\\
\endaligned
\right.
\end{equation}

\item {\bf More delicate estimates for the inhomogeneous case.} While employing Banach's fixed point theorem, for the new iteration scheme, we should derive the contraction of $a_{(n)}$, and $\u_{(n)}$ instead of the only contraction of $\u_{(n)}$ in the homogeneous case. Besides, higher regularity of $a$ and $\u$ are necessary in local existence. Thus, we assume the high regularity of initial data to obtain the a prior estimate of solutions. Furthermore, we do not project the pressures $\nabla P$ to be zero, which is different with the homogeneous case. Actually, to estimate $\nabla P$ suffices to the estimate the terms $$\quad -\nabla \triangle^{-1}\left( \nabla \cdot \left( \u_{(n-1)} \cdot \nabla \u_{(n-1)} \right)\d s +\frac{\mu\nabla \rho_{(n)}}{\rho_{(n)}}\triangle \v_{(n-1)}\right) ,$$
      cf. subsection \ref{Onto mapping in a ball}. %If we still apply Leray projection to the momentum equation and project the pressure to be $0$, we need deal with the extra nonlinear term
%       $$\nabla \triangle^{-1}\left( \nabla \cdot \left( \u_{(n-1)} \cdot \nabla \u_{(n-1)} \right)\right).$$

  \item {\bf Semigroup depending on time.} Compared with the usual scheme, the semigroup generator $\A_{(n)}=\frac{\mu}{\bar{\rho}\left(1+a_{(n)}\right)}\triangle$ concerns with the solution itself. So it is worth notion that $\A_{(n)}$ actually depends on time $t$, introducing a new challenge for mild solutions. This unfixed operator is proved as a generator of $C_{0}$-semigroup by checking Lumer-Phyllips's theorem. In this paper, we provide the existence of the evolving semigroup $ S_{(n)}(t)$, present its exact expression on the periodic domain, and demonstrate its time decay properties as well as its commutation with the gradient operator.

  \item {\bf Troubles aroused by Brownian motion.} As we addressed, the nonlinear term and the inhomogeneous density arise troubles for the estimates of moment. To tackle this, we deal with the stochastic term and the nonlinear term separately. Namely, we require the equation of $\v_{(n)}$ to be ``deterministic" when we split $\u_{(n)}=\v_{(n)}+\z_{(n)}$. In this approach, we estimate only the $r$-th expectation of $\z_{(n)}$. Consequently, the noise is supposed to be additive in the process of splitting. After this separation, we can establish the $\mathbb{P}$ a.s. boundedness of $\|\z_{(n)}\|_{H^{3}}$ by employing It\^o's formula and Burkholder-Davis-Gundy inequality. Then, based on the $\mathbb{P}$ a.s. boundedness of $\|\z_{(n)}\|_{H^{3}}$, we obtain the  $\mathbb{P}$ a.s. boundedness of the norm of $\v_{(n)}$ in momentum equation $\ref{deterministic linearized system})_{2}$. Furthermore, we do not use the temporal weighted method as in the deterministic case \cite{ZhifeiZhangCPDE,Zhang-Zhao-Zhang2014}, since we do not have the temporal derivative ~$\left(\u_{(n)}\right)_{t}$. The definition of ~$\left(\u_{(n)}\right)_{t}$ in the sense of weak derivatives involves the framework of Malliavin calculus, which will not be discussed in this paper.
\end{enumerate}

We describe the sketch of proof. For the iteration system \eqref{deterministic linearized system}, we prove that $\A_{(n)}$ is dissipative. By Lumer-Phyllips' theorem, it follows that $\A_{(n)}$ is the infinitesimal generator of a semigroup $S_{(n)}$.
% $S_{(n)}\Phi$ is a Hilbert-Schmidt operator,
%\begin{equation}
%\begin{split}
% S_{(n)}\Phi:[0,T]\times \mathbb{T}^{N} \ra \mathcal{V},\\
% S_{(n)}\Phi(t,x)=S(t)\Phi(t,x)\in \mathcal{H}.
%\end{split}
%\end{equation}
%where $\Phi\in L\left(\mathcal{H};\mathcal{V}\right)$, $S_{(n)}(t)\in L\left(\mathcal{V};\mathcal{V}\right)$, $S_{(n)}(t)\Phi\in L\left(\mathcal{H};\mathcal{V}\right)$, and it takes values in $\mathcal{H}$.
 Moreover, by the expression of a semigroup in a complex plane, we use Cauchy's integral formula to give a precise expression of the semigroup depending $t$ and $x$. We verify that the time decay property of $\left\|\A_{(n)}S_{(n)}\right\| \ls\frac{C}{t} \left\|S_{(n)}\right\|$ and its commutation with the gradient operator still holds. With these properties, we further use Banach's fixed point theorem to prove the local existence of a mild solution.
The map $\mathcal{L}$ is defined as an operator in the space
 \begin{align}
\mathcal{B}=\left\{\left.f\right|f\in C\left([0, T]; L^{p}\left(\mathbb{T}^{N}\right)\right)\right\},
\end{align}
where $f$ is a time-dependent function in $L^{p}(\mathbb{T}^{N})$ space.
The operator acts as
 \begin{align}
\mathcal{L}: \quad\mathcal{B}\ra \mathcal{B}, \quad\mathcal{L}f_{(n-1)}=f_{(n)},
\end{align}
essentially defining an iteration relationship between consecutive functions $f_{(n-1)}$ and $f_{(n)}$.
In this paper, we set
\begin{align}
\mathcal{L}\u_{(n-1)}=\u_{(n)}, \quad \mathcal{L}a_{(n-1)}=a_{(n)}.
\end{align}
In order to show $\mathcal{L}$ is onto, we need to show the estimate of $a_{(n)}$ and $\u_{(n)}$ uniformly in $n$, which need some good regularity of $a_{(n)}$. $a_{(n)}\in C\left([0, T];  W^{2,p}\left(\mathbb{T}^{N}\right)\cap C^{1,\ell}\left(\mathbb{T}^{N}\right)\right)$ is obtained by the hyperbolic property of mass conservation equation, provided $a_{0}\in   W^{2,p}\left(\mathbb{T}^{N}\right)\cap C^{1,\ell}\left(\mathbb{T}^{N}\right)$ and $\u_{(n-1)} \in C\left([0, T];  W^{2,p}\left(\mathbb{T}^{N}\right)\cap C^{1,\ell}\left(\mathbb{T}^{N}\right)\right)$. In the periodic domain, $W^{2,p}\left(\mathbb{T}^{N}\right)$ is compactly embedded in $L^{p}\left(\mathbb{T}^{N}\right)$. Then the contraction is operated in space $C\left([0, T]; L^{p}\left(\mathbb{T}^{N}\right)\right)$. Sequentially, we do the {\it a priori} energy estimate such that the regularity of $a_{(n)}$ and $\u_{(n)}$ meet each other's requirements mutually. Then it follows the global existence. Moreover, the incompressible condition is necessary in the zeroth-order energy estimates of $\u_{(n)}$. For the stochastic system, it is worth noting that we separately present the estimates for the onto mapping as well as the contractions of \(\z_{(n)}\) and \(\v_{(n)}\). However, the {\it a priori} estimates can be manipulated directly in terms of $\u_{(n)}$.

%We remark that the hyperbolic-parabolic structure of the system does not change because we assume that the density is far from the vacuum. By reviewing the proof, one can find that the requirement of density is just the existence of uniform positive lower bound. Moreover, this result does not need smallness of initial velocity as in \cite{ZhifeiZhangCPDE} and \cite{Chen-Zhang-Zhao2016}, but the high regularity of initial data is assumed.

The paper is organized as follows. In \S 2, we give the property of semigroup and its generators, especially for the operator depending on time $t$ and space $x$. In \S 3, we prove the local existence of mild solution $\u\in C\left([0, T];  L^{p}(\mathbb{T}^{N})\right)$, $\rho \in C\left([0, T];  L^{p}(\mathbb{T}^{N})\right)$. In \S 4, based on the energy estimate in \S 3, we show the global existence of mild solution by Zorn's lemma and the a priori estimates. \S 5 is the Appendix, in which we list some preliminaries we used in this paper.
 \smallskip
  \smallskip

\section{Some semigroup theories}

Here and hereafter, we use $\left\|\cdot\right\|$, $\left|\cdot\right|_{p}$, $\left|\cdot\right|_{k, p}$, and $\left\|\cdot\right\|_{s}$ to denote the $L^{2}$-norm, $L^{p}$-norm, $W^{k, p}$-norm  and $H^{s}$-norm respectively, ~$1<p\ls \infty$.

From Weyl's law (see \cite{strauss2007partial}), for a smooth bounded set $ U \subset \mathds{R}^{N}$, we let $\left\{\lambda_{j}\right\}_{j=1}^{\infty}$ be the eigenvalue of Laplacian operator $-\triangle$. Then, under the zero boundary condition, we have
\begin{equation}
\lim\limits_{j\ra\infty}\frac{\lambda_{j}^{\frac{N}{2}}}{j}= \frac{(2\pi)^{N}}{(\left|U\right|\alpha(N))},
\end{equation}
where $\left|U\right|$ is the volume of $U$, $\alpha(N)$ is the sphere area of the unit ball in $\mathds{R}^{N}$;
further, by denoting $\lambda_{1}= \min \left\{\left(\int_{\mathbb{T}^{N}}\left|\nabla\u\right|^{2}\d x\right)^{\frac{1}{2}}\mid \u\in H_{0}^{1}(U), ~ \left\| \u\right\|=1\right\}$, recalling $\left\| \cdot \right\|$ being the $L^{2}$-norm, we have
\begin{equation}
0<\lambda_{1}\ls\lambda_{2} \ls \lambda_{3}\ls \cdots.
\end{equation}
For torus $\mathbb{T}^{N}=[0,1]^{N}$, the eigenvectors are $\{\tilde{e}_{\mathbf{k}}\}_{|\mathbf{k}|=1}^{+\infty}=\left\{\cos2\pi \mathbf{k}\cdot \mathbf{x}, ~ \sin 2\pi \mathbf{k}\cdot \mathbf{x}\right\}$. Actually, we consider the following equation of eigenvalue
\begin{equation}
-\triangle u = \lambda u, \quad x \in \mathbb{T}^{N}.
\end{equation}
Assume that the solution is in the form of Fourier series:
\begin{equation}
u_{\mathbf{k}}(\mathbf{x}) = e^{i (\mathbf{k} \cdot \mathbf{x})}.
\end{equation}
Substituting the above form into the eigenvalue equation, we have
\begin{equation}
-\triangle u  = -\sum_{i=1}^{n} k_i^2 u.
\end{equation}
Thus, The eigenvalues are
\begin{equation}
\lambda_k = \sum_{i=1}^{n} k_i^2.
\end{equation}
Since  ~$u(x)$ satisfies the periodic condition
\begin{align}
u(x_1, \cdots, x_i, \cdots,  x_n) = u(x_1, \cdots, x_i+1, \cdots,  x_n) , \quad i=1, 2, \cdots, n, \\
\frac{\partial }{\partial x_{i}}u(x_1, \cdots, x_i, \cdots,  x_n) = \frac{\partial }{\partial x_{i}} u(x_1, \cdots, x_i+1, \cdots,  x_n), \quad i=1, 2, \cdots, n,
\end{align}
 $k_{n}$ is 
\begin{equation}
k_n = 2\pi n, \quad n \in \mathbb{Z}.
\end{equation}
In summary, The eigenvalues are
\begin{equation}
\lambda_{\mathbf{k}}=\sum_{i=1}^{n}\left(2\pi k_i\right)^2;
\end{equation}
and the eigenfunctions are 
\begin{equation}
u_{\mathbf{k}}(\mathbf{x}) = e^{i \sum\limits_{i=1}^{n} k_i x_i}.
\end{equation}

Notice that for general smooth domain, the eigenvectors are unknown.

  We list some theories on semigroup and its properties in this section.

\begin{lemma}\label{Lumer-Phyllips' theorem}
(Lumer-Phyllips' theorem) Let the linear operator $\A$ be densely defined in the Banach space $Y$, i.e., the domain of $\A$, $D(\A)$ is dense in $Y$. If both $\A$ and the conjugate operator $\A^{*}$ are dissipative, then $\A$ is the infinitesimal generator of a $C_{0}$-semigroup of contractions on $Y$.
\end{lemma}
\begin{lemma}\label{dissipative operator is semigroup generator}
A linear operator $\A$ is dissipative if and only if
$\left\|\left(\lambda I-\A\right) g\right\|\geqslant \lambda\left\| g\right\|$, for all $ g\in D(\A)$, and $\lambda>0$.
\end{lemma}
We begin by verifying the dissipative property, as stated in the following lemma.
\begin{lemma} \label{dissipative generator}
Let $\A(t,x)=\mu(t,x)\triangle $, $0\ls  \underline{\mu} \ls \mu(t,x)\ls \overline{\mu}$, $\underline{\mu}$ and $\overline{\mu}$ are constants, then $\A(t,x)$ is dissipative for any fixed $t$.
\end{lemma}
\begin{proof}
 For any $ g \in Y $, where $\{\tilde{e}_{\mathbf{k}}\}_{|\mathbf{k}|=1}^{+\infty}$ is dense in $Y$, it holds that
 \begin{align}\label{lambda I -A}
& \pl\left(\lambda I-\A\left(t,x\right) \right) g \pl^{2}\notag \\
=& \left( \left(\lambda I-\A\left(t,x\right) \right) g,~\left(\lambda I-\A\left(t,x\right) \right) g \right) \\
=&\left( \lambda^{2}  g,~ g\right)+2\left( \lambda  g ,~-\A\left(t,x\right)  g\right) + \left( \A\left(t,x\right) g,~\A\left(t,x\right)  g\right).\notag
\end{align}
For $ g =\sum_{|\mathbf{i}|=1}^{+\infty}k_{\mathbf{i}}\tilde{e}_{\mathbf{i}}$,  $\sum_{|\mathbf{i}|=1}^{+\infty}|k_{\mathbf{i}}|^{2} $ is bounded, where $k_{\mathbf{i}}$ is merely dependent of $t$. There holds
\begin{align}
&\left( \lambda  g  ,~ -\A\left(t,x\right) g \right)  \notag\\
=&\left( -\mu(t,x) \lambda \triangle \left(\sum_{|\mathbf{i}|=1}^{+\infty}k_{\mathbf{i}}\tilde{e}_{\mathbf{i}}\right),~ \sum_{i=1}^{+\infty} k_{\mathbf{i}}\tilde{e}_{\mathbf{i}} \right)  \notag\\
=&\left( -\mu(t,x)\lambda \sum_{|\mathbf{i}|=1}^{+\infty} (-\lambda_{\mathbf{i}}k_{\mathbf{i}}\tilde{e}_{\mathbf{i}}),~ \sum_{i=1}^{+\infty} k_{\mathbf{i}}\tilde{e}_{\mathbf{i}}\right) \notag\\
=& \left( \mu(t,x) \lambda \sum_{|\mathbf{i}|=1}^{+\infty} (\lambda_{\mathbf{i}}k_{\mathbf{i}}\tilde{e}_{\mathbf{i}}),~ \sum_{i=1}^{+\infty} k_{\mathbf{i}}\tilde{e}_{\mathbf{i}}\right)\\
\geqslant & \underline{\mu}\lambda \left( \sum_{|\mathbf{i}|=1}^{+\infty} (\lambda_{\mathbf{i}}k_{\mathbf{i}}\tilde{e}_{\mathbf{i}}),~ \sum_{|\mathbf{i}|=1}^{+\infty} k_{\mathbf{i}}\tilde{e}_{\mathbf{i}}\right) \notag\\
\geqslant & \underline{\mu}\lambda  \sum_{|\mathbf{i}|=1}^{+\infty}\lambda_{\mathbf{i}} |k_{\mathbf{i}}|^{2} \geqslant \underline{\mu}\lambda \lambda_{1} \sum_{|\mathbf{i}|=1}^{+\infty} |k_{\mathbf{i}}|^{2} \geqslant  0, \notag
%=& \overline{\mu} \int_{\mathbb{T}^{N}} \left|\nabla u_{\mathbf{i}}\right|^{2}  I_{\{\triangle u_{\mathbf{i}} u_{\mathbf{i}}\geqslant 0\} }\d x +  \underline{\mu} \int_{\mathbb{T}^{N}} \left|\nabla u_{\mathbf{i}}\right|^{2} I_{\{\triangle u_{\mathbf{i}} u_{\mathbf{i}} \ls 0 \}}\d x \notag\\
\end{align}
For the last term in \eqref{lambda I -A}, by the definition, we have
\begin{align}
\left( \A\left(t,x\right)  g ,~ \A\left(t,x\right)  g \right) \geqslant 0.
\end{align}
Therefore, we have
\begin{align}
\pl\left(\lambda I-\A\left(t,x\right)\right)g \pl^{2} \geqslant  \left( \lambda^{2} g, g \right)=\lambda^{2} \pl g \pl^{2}.
\end{align}\hfill $\square$
\end{proof}
 %\begin{equation}\label{linearized system}
%\left\{\aligned
%& \partial_{t} a_{(n)}+\u_{(n-1)}\cdot\nabla a_{(n)}= 0, \\
%& \mathbf{P}_{q}\left(\d \z_{(n)}(t)- \frac{\mu\triangle}{\left(1+a_{(n)}\right)}\z_{(n)}(t)\d t\right)=\mathbf{P}_{q}\Phi \d W (t),\\
%&\mathbf{P}_{q} \left(\partial_{t}\v_{(n)}(t) - \frac{ \mu \triangle}{\left(1+a_{(n)}\right)}\v_{(n)}(t)\right)=-\mathbf{P}_{q}(\u_{(n-1)}(t))\cdot\nabla(\u_{(n-1)}(t)), \\
%& a_{(n)}\mid_{t=0}=a_{0},\quad a_{(0)}=a_{0},\\
%& \odiv \v_{(n)}=0, \quad \v_{(n)}\mid_{t=0}=\u_{0}, \quad\v_{(0)}=\u_{0},\quad \v_{(n)}\mid_{\partial U}=0,\\
%& \odiv \z_{(n)}=0, \quad \z_{(n)}\mid_{t=0}=0, \quad \z_{(0)}=0,  \quad \z_{(n)}\mid_{\partial U}=0.
%\endaligned
%\right.
%\end{equation}
%Let $\lambda_{j}$ be the eigenvalue of Laplacian operator with periodic condition, and we let $\left\{\tilde{e}_{j}\right\}_{|j|=1}^{\infty}$ be the family of trigonometric eigenfunctions.
We consider the conjugate operator $\A^{*}=\mu(t,x)\triangle$, i.e., ~$\left(\A f, ~g\right)=\left( f, ~\A^{*} g\right)$, for ~$ f,~ g  \in D(\A)$. 
Since ~$f=\sum_{|\mathbf{i}|=1}^{+\infty}l_{\mathbf{i}}e_{\textbf{i}}$, ~$g =\sum_{|\mathbf{i}|=1}^{+\infty}k_{\mathbf{i}}e_{\textbf{i}}$, we have
\begin{align}
\left( \A\left(t,x\right)  f,~g  \right)
=&\left(\mu\left(t,x\right)\triangle \sum_{|\mathbf{i}|=1}^{+\infty}l_{\mathbf{i}}e_{\textbf{i}},~\sum_{|\mathbf{i}|=1}^{+\infty}k_{\mathbf{i}}e_{\textbf{i}}  \right)
=\left(\mu\left(t,x\right) \sum_{|\mathbf{i}|=1}^{+\infty}\lambda_{i}l_{\mathbf{i}}e_{\textbf{i}},~\sum_{|\mathbf{i}|=1}^{+\infty}k_{\mathbf{i}}e_{\textbf{i}} \right)\\
=&\left( \sum_{|\mathbf{i}|=1}^{+\infty}k_{\mathbf{i}}l_{\mathbf{i}}e_{\textbf{i}}, ~ \mu\left(t,x\right)\sum_{|\mathbf{i}|=1}^{+\infty}\lambda_{i} k_{\mathbf{i}}e_{\textbf{i}} \right)=\left( f, ~\A^{*}(t,x) g\right),\notag
\end{align}
Thus, ~$\A^{*}(t,x)=\mu(t,x)\triangle $, is dissipative as well.

 We will prove the following lemma, which gives the explicit expression for semigroup.
\begin{lemma}\label{lemma expression of semigroup}
Let the $C_{0}-$semigroup $T(t,x)$ be generated by $\A(t,x)$, satisfying $\left\|T(t,x)\right\|\ls M e^{\omega t}$, and $\omega\ls 0$. Then, for eigenfunctions $\tilde{e}_{\mathbf{j}}\in D\left(\A\right)$, there holds
\begin{align}\label{expression of semigroup}
T(t)\tilde{e}_{\mathbf{j}}= \frac{1}{2\pi i}\int_{l} e^{\lambda t} \left(\lambda I-\A\right)^{-1}\tilde{e}_{\mathbf{j}}\d \lambda = e^{-\lambda_{\mathbf{j}}\mu(t,x) t} \tilde{e}_{\mathbf{j}},
\end{align}
where $l$ is a line from $\jmath-i\infty$ to $\jmath+i\infty$, $\jmath> 0$, for any fixed $t$.
\end{lemma}
\begin{proof}
By Corollary 1.7.5 in \cite{Pazy}, it holds that
\begin{equation}
T(t)f=\frac{1}{2\pi i}\int_{\jmath-i\infty}^{\jmath+i\infty} e^{\lambda t} \left(\lambda I-\A\right)^{-1}f\d \lambda.
\end{equation}
Let $\jmath> \max \{0,\omega\}=0$. For $0<\jmath\ls \mu(t,x) \lambda_{\mathbf{j}}$, i.e., $0<\frac{\jmath}{\mu(t,x)} \ls \lambda_{\mathbf{j}}$, denoting $\Lambda=\frac{\lambda}{\mu(t,x)}$, we have
\begin{align}
T(t)\tilde{e}_{\mathbf{j}} = & \frac{1}{2\pi i}\int_{\jmath-i\infty}^{\jmath+i\infty} e^{\lambda t} \left(\lambda I-\A\right)^{-1}\tilde{e}_{\mathbf{j}}\d \lambda \notag\\
=& \frac{1}{2\pi i}\int_{\jmath-i\infty}^{\jmath+i\infty} e^{\lambda t} \left(\lambda I-\mu(t,x) \triangle\right)^{-1}\tilde{e}_{\mathbf{j}} \d \lambda \\
=&  \frac{1}{\mu(t,x)} \frac{1}{2\pi i} \int_{\jmath-i\infty}^{\jmath+i\infty} e^{\lambda t} \left(\frac{\lambda }{\mu(t,x)} I- \triangle\right)^{-1}\tilde{e}_{\mathbf{j}} \d \lambda \notag \\
=& \frac{1}{2\pi i}\int_{\frac{\jmath}{\mu(t,x)}-i\infty}^{\frac{\jmath}{\mu(t,x)}+i\infty} e^{\Lambda\mu(t,x) t} \left(\Lambda I- \triangle\right)^{-1}\tilde{e}_{\mathbf{j}} \d \Lambda \notag.
\end{align}
We calculate $\left(\Lambda I- \triangle\right)^{-1}\tilde{e}_{\mathbf{j}}$ as follows: \\
since
\begin{align}
\left(\Lambda I-\triangle \right)\tilde{e}_{\mathbf{j}} = \left(\Lambda+\lambda_{\mathbf{j}} \right)\tilde{e}_{\mathbf{j}},\\
\tilde{e}_{\mathbf{j}} = \frac{1}{\left(\Lambda+\lambda_{\mathbf{j}} \right)}\left(\Lambda I-\triangle \right)\tilde{e}_{\mathbf{j}},
\end{align}
 we have
\begin{align}
\left(\Lambda I-\triangle \right)^{-1}\tilde{e}_{\mathbf{j}} = \left(\Lambda I-\triangle \right)^{-1} \frac{1}{\left(\Lambda+\lambda_{\mathbf{j}} \right)}\left(\Lambda I-\triangle \right)\tilde{e}_{\mathbf{j}}= \frac{1}{\left(\Lambda+\lambda_{\mathbf{j}} \right)}\tilde{e}_{\mathbf{j}}.
\end{align}
Thus, it holds that
\begin{align}
T(t)\tilde{e}_{\mathbf{j}}=\frac{1}{2\pi i}\int_{\frac{\jmath}{\mu(t,x)}-i\infty}^{\frac{\jmath}{\mu(t,x)}+i\infty} e^{\Lambda\mu(t,x) t} \frac{1}{\left(\Lambda+\lambda_{\mathbf{j}} \right)}\tilde{e}_{\mathbf{j}} \d \Lambda.
\end{align}
We denote the line as $l:$ extending from $\jmath-i\infty$ to $\jmath+i\infty$, and the half circle as $\mathscr{C}: R e^{i\theta}$, where $\theta$ ranges from $\frac{\pi}{2}$ to $ \frac{3\pi}{2}$. Then, we estimate
\begin{align}
& \left|\frac{1}{2\pi i}\int_{\mathscr{C}} e^{\Lambda\mu(t,x)t} \left(\Lambda I- \triangle\right)^{-1}\tilde{e}_{\mathbf{j}} \d \Lambda \right|\notag \\
\ls & \frac{1}{2\pi}\left|\int_{\frac{\pi}{2}}^{\frac{3\pi}{2}} e^{Re^{i\theta}\mu(t,x) t}  \frac{1}{\left(R e^{i\theta}+\lambda_{\mathbf{j}} \right)}\tilde{e}_{\mathbf{j}} \d \theta \right|\notag \\
= & \frac{1}{2\pi}\left|\int_{\frac{\pi}{2}}^{\frac{3\pi}{2}} e^{R\left(\cos \theta+i \sin \theta\right)\mu(t,x)t}  \frac{1}{\left(R e^{i\theta}+\lambda_{\mathbf{j}} \right)}\tilde{e}_{\mathbf{j}} \d \theta \right|\notag \\
= & \frac{1}{2\pi}\int_{\frac{\pi}{2}}^{\frac{3\pi}{2}} e^{R\cos \theta\mu(t,x) t} \left| \frac{1}{\left(R e^{i\theta}+\lambda_{\mathbf{j}} \right)}\tilde{e}_{\mathbf{j}} \right| \d \theta \notag \\
\ls & \frac{1}{2\pi}\int_{\frac{\pi}{2}}^{\frac{3\pi}{2}}  \left| \frac{1}{\left(R e^{i\theta}+\lambda_{\mathbf{j}} \right)}\tilde{e}_{\mathbf{j}} \right| \d \theta \notag \\
\ra &0, \text{ as }R \ra \infty, \notag
\end{align}
due to $R\cos\theta\mu(t,x)t\ls 0$. One can also get the above formula from Jordan's lemma. By Cauchy's integral formula, we have
\begin{align}
T(t)\tilde{e}_{\mathbf{j}}= \frac{1}{2\pi i}\int_{l+\mathscr{C}} e^{\lambda t} \left(\lambda I-\A\right)^{-1}\tilde{e}_{\mathbf{j}}\d \lambda = e^{-\lambda_{\mathbf{j}}\mu(t,x) t} \tilde{e}_{\mathbf{j}}.
\end{align}\hfill $\square$
\end{proof}
The following lemma shows the important property of the resolvent $\left(\lambda I-\A\right)^{-1}$. This property is the precondition of the existence of semigroup for the temporal evolving systems.
\begin{lemma}\label{Theorem 1.5.2 in Pizy}(Theorem 1.5.2 in \cite{Pazy}) The linear operator $\A$ is the infinitesimal generator of a $C_0-$semigroup $T(t)$, satisfying $\|T(t)\| \ls  M $, with $M$ being a constant $M \geqslant 1$, if and only if
\begin{itemize}
  \item[1.] $\A$ is closed and $D(\A)$ is dense in $Y$;
  \item[2.] The resolvent set $\Theta(\A)$ of $\A$ contains $\mathbb{R}^{+}$. Moreover, the resolvent of $\A(t)$, $R(\lambda: \A)=\left(\lambda I-\A\right)^{-1}$,  is a bounded linear operator for $\lambda\in \Theta(\A)$, satisfying
\begin{align}
\left\|R(\lambda: \A)^n\right\| \ls M / \lambda^n \quad \text { for } \quad \lambda>0, \quad n=1,2, \ldots.
\end{align}
\end{itemize}
\end{lemma}

From the textbook \cite{Pazy}, it is clear that the existence and the uniqueness of semigroup $U(t,s)$ are subjected to the following assumptions of generators:
\begin{itemize}
  \item[$\left(P_1\right)$] the domain $D(\A(t))$ of $\A(t)$, $0 \ls  t \ls  T$ is dense in $Y$ and independent of $t$;
  \item[$\left(P_2\right)$] for $t \in[0, T]$, the resolvent of $\A(t)$, $R(\lambda: \A(t))=\left(\lambda I-\A(t)\right)^{-1}$, exists for all $\lambda$ with $\operatorname{Re} \lambda < 0$ and there is a constant $M$ such that
$$
\|R(\lambda: \A(t))\| \ls  \frac{M}{|\lambda|} \quad \text { for } \quad \operatorname{Re} \lambda < 0,\quad t \in[0, T];
$$
  \item[$\left(P_3\right)$] there exist constants $L$ and $0<\alpha \ls  1$ such that
$$
\left\|(\A(t)-\A(s)) \A(\tau)^{-1}\right\| \ls  L|t-s|^\alpha \quad \text { for } \quad s, ~t, ~\tau \in[0, T];
$$
\end{itemize}
With $\left(P_1\right)-\left(P_3\right)$ above, one has the existence of a specific evolving system. We list the theorem 5.6.1 and differentiability in \cite{Pazy} here.
\begin{lemma}\label{Semigroup property}\cite{Pazy}
Under the assumptions $\left(P_1\right)-\left(P_3\right)$, for the evolution system
\begin{align}
\frac{\partial}{\partial t} u+\A(t)u & =0, \quad \text { for } \quad 0 \ls  s<t \ls  T,
\end{align}
with $u(s)=v\in Y$, there exists a unique solution $U(t, s)v$, where $U(t, s)$ evolves from time $s$ to $t$, $0 \ls  s \ls  t \ls  T$. The solution satisfies the following properties:
\begin{itemize}
  \item[$\left(E_1\right)^{+}$] $\|U(t, s)\| \ls  C, \quad$ for $0 \ls  s \ls  t \ls  T$;
  \item[$\left(E_2\right)^{+}$] for $0 \ls  s<t \ls  T, \quad U(t, s): Y \rightarrow D\left(\A(t)\right)$ and $t \rightarrow U(t, s)$ is strongly differentiable in $Y$. The derivative $\frac{\partial}{\partial t}  U(t, s) $ is strongly continuous on $0 \ls  s<t \ls  T$. Moreover, it holds:
\begin{align}
\frac{\partial}{\partial t} U(t, s)+\A(t) U(t, s) & =0, \quad \text { for } \quad 0 \ls  s<t \ls  T,
\end{align}
\begin{align}\label{time decay of laplacian}
\left\|\frac{\partial}{\partial t} U(t, s)\right\| & =\|\A(t) U(t, s)\| \ls  \frac{C}{t-s},
\end{align}
and
\begin{align}
\left\|\A(t) U(t, s) \A(s)^{-1}\right\| \ls  C, \quad \text { for } \quad 0 \ls  s \ls  t \ls  T ;
\end{align}
  \item[$\left(E_3\right)^{+}$] for every $v \in D\left(A(t)\right)$ and $\left.\left.t \in\right] 0, T\right]$, $U(t, s) v$ is differentiable with respect to $s$ on $0 \ls  s \ls  t \ls  T$ and satisfies
\begin{align}
\frac{\partial}{\partial s} U(t, s) v=U(t, s) \A(s) v.
\end{align}
%\begin{align}
%\left\|S_t(s)\right\| \ls  C,  \quad & \text { for } \quad s \geq 0,
%\end{align}
%\begin{align}
%\left\|\A(t) S_t(s)\right\| \ls  \frac{C}{s}, \quad& \text { for } \quad s>0,
%\end{align}
\end{itemize}
%where $S_t(s)$ is the semigroup generated by $\A(t)$.
\end{lemma}
Again, from Theorem $5.6.8$ in \cite{Pazy}, if ($P_{1}$)-($P_{3}$) is satisfied, then it follows that for the initial data $v$,
\begin{align}\label{analytic property of A}
U(t, s) v=U(t, r) U(r, s) v, \quad \text { for } \quad 0 \ls  s \ls  t \ls  T.
\end{align}
Lemma $5.6.6$ in \cite{Pazy} implies that
\begin{align}\label{continuity at initial time}
\lim _{\varepsilon \rightarrow 0} U(\varepsilon,0) v=v, \quad \text { uniformly in } \quad 0 \ls  \varepsilon \ls  T.
\end{align}
The operator $\left(0I-\A\right)$ is invertible, i.e., $0$ is in the resolvent of $\A$. Then, there exists a constant $\delta> 0$ such that
$\left(- \A + \delta I\right)$ is still an infinitesimal generator of an analytic semigroup. We denote $S(t)=U(t,0)$ as the semigroup for the evolving system.
Therefore, we have
\begin{align}
\|S(t)\| & \ls  M e^{-\delta t},\notag \\
\|\A S(t)\| & \ls  M_{1} t^{-1} e^{-\delta t}.
%\left\|\A^{m} T(t)\right\| & \ls  M_{m} t^{-m} e^{-\delta t}.\notag
\end{align}
As in \cite{Pazy}, for $0<\imath<1$, we define
\begin{equation}
\A^{-\imath}=\frac{1}{2 \pi i} \int_{\mathscr{C}} z^{-\imath}(\A-z I)^{-1} \d z.
\end{equation}
where the path $\mathscr{C}$ runs in the resolvent set of $A$ from $\infty e^{-i \vartheta}$ to $\infty e^{i \vartheta}$, with $\omega<\vartheta<\pi$, avoiding the negative real axis and the origin. $z^{-\imath}$ is taken to be positive for real positive values of $z$.
In the complex plane, we define
\begin{equation}
 \Sigma^{+}=\{\lambda: 0<\kappa<|\arg z| \ls  \pi\},
\end{equation}
for some positive constant $\kappa$. Let $V$ be a neighborhood of zero, then $V$ and $\Sigma^{+}$ are the subsets of the resolvent set of $\A$.
Since, for the eigenvalues $\lambda_{k}>0$, $\lambda_{k}I-\A$ is not invertible, we obtain that for $\kappa<\frac{\pi}{2}$, $\Sigma^{+}$ is the subset of the resolvent set of $\A$. Hence, the path of integration is deformed into the upper and lower sides of the negative real axis. Therefore, it holds that
\begin{equation}
\A^{\imath}=\frac{\sin \pi \imath}{\pi} \int_{0}^{\infty} \jmath^{-\imath}(\jmath I+\A)^{-1} \d \jmath, \quad 0<\imath<1.
\end{equation}
\begin{equation}
\A^{\imath}=\left(\A^{-\imath}\right)^{-1},
\end{equation}
Thus, there holds
\begin{equation}
\A^{\imath} x=\A\left(\A^{\imath-1} x\right)=\frac{\sin \pi \imath}{\pi} \int_{0}^{\infty} \jmath^{\imath-1} \A\left(\jmath I+\A\right)^{-1} x \d \jmath.
\end{equation}

 Now we briefly derive $\nabla S(t,x)=S(t,x)\nabla $ when $S(t,x)$ is differentiable with respect to $x$, i.e., when $a\in C([0,T]; C^{1}(\mathbb{T}^{N}))$. We use a split on $t$ and $x$, and denote $t^{n}_{k}=\left(\frac{k}{n}\right)T$, $k=0,~1, ~2, ~\cdots, ~n$, $i=1,~\cdots,~N$, $x=\left(x^{1},~ x^{2}, ~\cdots, ~x^{N}\right)$, $x^{m,i}_{k}=\left(\frac{k}{m}\right)$, $k=0, ~1, ~2, ~\cdots, ~m$. If $t_{k}^n\ls t <t_{k+1}^n$, ~$\frac{k}{m}\ls x^{i}< \frac{k+1}{m}$, then
 \begin{align}
 S_{n,m}(t,x)=S\left(t^{n}_{k}, ~x^{m,1}_{k}, ~x^{m,2}_{k}, ~\cdots, ~x^{m,N}_{k}\right), \\
S_{n,m}\left(T, ~1, ~1, ~\cdots, ~1\right)=S\left(T, ~1, ~1, ~\cdots, ~1\right).
 \end{align}
Since $S_{n,m}(t,x)$ converges to $S(t,x)$ uniformly, we deduce that:
\begin{align}
&\nabla S(t,x)=\nabla \lim_{n,m\ra +\infty} S_{n,m}(t,x)\\
= &\lim_{n,m\ra +\infty} \nabla S_{n,m}(t,x)= \lim_{n,m\ra +\infty} S_{n,m}(t,x) \nabla =S(t,x)\nabla , \notag
\end{align}
provided that $S(t,x)$ is differentiable.

 By reviewing the proof of theorem 2.6.13 part (a) and part (c) in \cite{Pazy},
 the fractional derivative of semigroup holds
\begin{align} \label{time decay of fractional derivative of semigroup}
\left\|\A_{(n)}^{\imath} S(t)\right\| \ls  M_{\imath} t^{-\imath} e^{-\delta t}, ~0<\imath<1.
\end{align}

With the above properties, by reviewing the proof of Proposition 3.1. (i) in Wessler's paper \cite{Weissler1980}, the following statement still holds.
\begin{lemma}\label{time decay of semigroup}
Let $1<p<q<\infty$. Then, for any $t>0$, $S(t)=e^{t\A}$ is a bounded map from $L^p$ to $ L^q $. For each $T>0$, there is a constant $C$ depending on $p$ and $q$, %but growing only as a power of $T$,
such that
\begin{equation}
\left\|e^{t \A} f\right\|_q \ls C t^{-\frac{N}{2 l}}\left\|f\right\|_p,\\
\end{equation}
for all $f \in X$ and $t \in(0, T]$, where $\frac{1}{l}=\frac{1}{p}-\frac{1}{q}$.
\end{lemma}
% For $S(t)\in L\left(\mathcal{V};\mathcal{V}\right)$, $p=q$, it holds $\frac{1}{l}=0$. Hence, we have $\left\|e^{t \A} f\right\|_{\mathcal{V}} \ls  C \left\|f\right\|_{\mathcal{V}}$.
 \smallskip
 \smallskip

\section{Local mild $L^p$-solutions}

While considering the effects of stochastic forces, we separate the stochastic system \eqref{sto inhomo incom NS}
into \eqref{pl1} and \eqref{pn2}.
% \begin{equation}\label
%\left\{\aligned
%&\d \z(t)- \frac{\mu}{\bar{\rho}\left(1+a\right)}\triangle \z(t)\d t= \Phi \d W(t),\\
%& \z\mid_{t=0}=0, \quad a\mid_{t=0}=a_{0}, \quad\z\mid_{\partial \mathbb{T}^{N}}=0,
%\endaligned
%\right.
%\end{equation}
%and the initial boundary problem to the deterministic nonlinear evolution equation
%\begin{equation}\label{pn2}
%\left\{\aligned
%&\partial_{t}\v(t) - \frac{\mu}{\bar{\rho}\left(1+a\right)}\triangle\v(t) = -\left(\v(t)+\z(t)\right)\cdot\nabla\left(\v(t)+\z(t)\right)- \nabla Q\left(a,\u\right), \\
%&\v\mid_{t=0}=\u_0, \quad a\mid_{t=0}=a_{0},\quad \v\mid_{\partial \mathbb{T}^{N}}=0.
%\endaligned
%\right.
%\end{equation}
\eqref{pn2} is not a exact deterministic equation, since $a$ depends on $\u$, so $a$, $\u$, and $\z$ are all the stochastic process. Hence we try to get the $\mathbb{P}$ a.s. estimates of $\z$. %That is why we separate the system into two parts.
We need to linearize the momentum equation as follows
\begin{equation}\label{linearized system}
\left\{\aligned
&\partial_{t} a_{(n)} + \left(\v_{(n-1)}+\z_{(n-1)}\right)\cdot\nabla a_{(n)} = 0,\\
& \d \z_{(n)}(t) - \frac{\mu}{\bar{\rho}\left(1+a_{(n)}\right)}\triangle\z_{(n)}(t)\d t = \Phi \d W(t),\\
& \partial_{t}\v_{(n)}(t) - \frac{\mu}{\bar{\rho}\left(1+a_{(n)}\right)}\triangle\v_{(n)}(t)\\
& =-\left(\v_{(n-1)}(t)+\z_{(n-1)}(t)\right)\cdot\nabla\left(\v_{(n-1)}(t)+\z_{(n-1)}(t)\right)- \nabla Q\left(a_{(n)},\u_{(n-1)}\right), \\
& \odiv \left(\v_{(n)}+\z_{(n)}\right)=0,\\
& \v_{(n)}\mid_{t=0}=\u_{0}, \quad \v_{(0)}=\u_{0}, \\ %\quad \v_{(n)}\mid_{\partial \mathbb{T}^{N}}=0,\\
&  \z_{(n)}\mid_{t=0}=0, \quad \z_{(0)}=0.%,  \quad \z_{(n)}\mid_{\partial \mathbb{T}^{N}}=0.
\endaligned
\right.
\end{equation}
We denote the operator $\A_{(n)}(t, x)=\frac{\mu  \triangle}{\bar{\rho}\left(1+a_{(n)(t,x)}\right)}$, writing $\A_{(n)}(t)$ for short. $a_{(n)}$ obtained first from $\eqref{linearized system}_{1}$, then becomes known in $\eqref{linearized system}_{2}$ for the $n$-th iteration. Applying the theory in \S 7.6 of \cite{Pazy}, for a smooth uniformly-elliptic operator, there exists a unique solution $\u_{(n)}$, with $\u_{(n-1)}$ given. By Weyl's law, it holds $-\triangle \tilde{e}_{\mathbf{j}}=\lambda_{j}\tilde{e}_{\mathbf{j}}$, so there holds $-\frac{\mu  \triangle}{\bar{\rho}\left(1+a_{(n)}\right)}\tilde{e}_{\mathbf{j}}=\frac{\mu  \lambda_{j}}{\bar{\rho}\left(1+a_{(n)}\right)}\tilde{e}_{\mathbf{j}}$. That is, $-\frac{\lambda_{j}}{\bar{\rho}\left(1+a_{(n)}\right)}$ is the eigenvalue of $\A_{(n)}(t)$. The following lemma states that $\A_{(n)}(t)$ is an infinitesimal generator of the semigroup $S_{(n)}(t)$ for $a_{(n)}\in C\left([0, T]; W^{2,p}\left(\mathbb{T}^{N}\right)\right)$.
\begin{lemma} \label{infinitesimal generator}
Let $a_{(n)}\in C\left([0, T]; C^{1,\ell}\left(\mathbb{T}^{N}\right)\right)$.
Then, $\A_{(n)}(t)$ is an infinitesimal generator of a semigroup $S_{(n)}(t)$, $t\in [0,T]$. Meanwhile, it holds that
\begin{equation}
S_{(n)}(t)\tilde{e}_{\mathbf{j}}=e^{-(t-s)\frac{\mu\lambda_{j}}{\bar{\rho}\left(1+a_{(n)}\right)}} \tilde{e}_{\mathbf{j}}.
\end{equation}
\end{lemma}
\begin{proof}
We first observe that $D(A_{(n)}(t))=\operatorname{span}\left\{\tilde{e}_{\mathbf{j}}\right\}$  %\in C\left([0, T];  $H^{5}\left(\mathbb{T}^{N}\right)\right)$, is independent on $t$. Regarding the dimension $1\ls N\ls 3$, %this space exhibits different embeddings. For $N<4$,
% $H^{5}\left(\mathbb{T}^{N}\right)\hookrightarrow C^{3}\left(\mathbb{T}^{N}\right)$, $C^{3}\left(\mathbb{T}^{N}\right)$ is continuously embedded in $ H^{3}\left(\mathbb{T}^{N}\right)\cap W^{2,p}\left(\mathbb{T}^{N}\right) $; %for $N>4$, $p<\frac{2N}{N-4}$, $H_{0}^{2}\left(\mathbb{T}^{N}\right)$ is contained in $L^{p}\left(\mathbb{T}^{N}\right)$.
%It is worth noting that $C^{\infty}\left(\mathbb{T}^{N}\right)$ is densely in $W^{2,p}\left(\mathbb{T}^{N}\right)$ and $H^{3}\left(\mathbb{T}^{N}\right)$. Consequently, $C\left([0, T];  H^{5}\left(\mathbb{T}^{N}\right)\right)$
 is densely defined in the solutions space $ C\left([0, T];  H^{3}\left(\mathbb{T}^{N}\right)\cap W^{2,p}\left(\mathbb{T}^{N}\right)\right)$. By selecting $\mu(t,x)= \frac{\mu}{\bar{\rho}\left(1+a_{(n)}\right)}$ in Lemma \ref{dissipative operator is semigroup generator}, with $ \mu(t,x)\ls \frac{2\mu}{\bar{\rho}}$, then the operator $\A_{(n)}(t)$ is dissipative. Notably, $\frac{1}{\bar{\rho}\left(1+a_{(n)}\right)}$ is independent of $\u$, i.e., $\A_{(n)}(t)$ is linear. By Lumer-Phyllips' theorem \ref{Lumer-Phyllips' theorem}, we ascertain that $\A_{(n)}(t)$ is an infinitesimal generator for the semigroup $S_{(n)}(t)$. By Lemma \ref{lemma expression of semigroup}, the expression of semigroup \eqref{expression of semigroup} holds. This proves the theorem.
\hfill $\square$
\end{proof}
 %\begin{equation}\label{linearized system}
%\left\{\aligned
%& \partial_{t} a_{(n)}+\u_{(n-1)}\cdot\nabla a_{(n)}= 0, \\
%& \mathbf{P}_{q}\left(\d \z_{(n)}(t)- \frac{\mu\triangle}{\left(1+a_{(n)}\right)}\z_{(n)}(t)\d t\right)=\mathbf{P}_{q}\Phi \d W (t),\\
%&\mathbf{P}_{q} \left(\partial_{t}\v_{(n)}(t) - \frac{ \mu \triangle}{\left(1+a_{(n)}\right)}\v_{(n)}(t)\right)=-\mathbf{P}_{q}(\u_{(n-1)}(t))\cdot\nabla(\u_{(n-1)}(t)), \\
%& a_{(n)}\mid_{t=0}=a_{0},\quad a_{(0)}=a_{0},\\
%& \odiv \v_{(n)}=0, \quad \v_{(n)}\mid_{t=0}=\u_{0}, \quad\v_{(0)}=\u_{0},\quad \v_{(n)}\mid_{\partial U}=0,\\
%& \odiv \z_{(n)}=0, \quad \z_{(n)}\mid_{t=0}=0, \quad \z_{(0)}=0,  \quad \z_{(n)}\mid_{\partial U}=0.
%\endaligned
%\right.
%\end{equation}
$\left(P_1\right)$ is satisfied from the above lemma.
For each $n$, $\A_{(n)}(t)$ is an analytic infinitesimal generator of the semigroup. Thus, by Lemma \ref{Theorem 1.5.2 in Pizy}, the resolvent set $\Theta (\A_{(n)}(t))$ of $\A_{(n)}(t)$ contains $\mathbb{R}^{+}$ and there holds
\begin{align}
\left\|R(\lambda: \A_{(n)}(t))^n\right\| \ls  M / \lambda^n \quad \text { for } \quad \lambda>0, \quad n=1,2, \cdots.
\end{align}
Hence, $\A_{(n)}(t)$ satisfies the condition $\left(P_2\right)$.

Next we verify the condition $(P_{3})$. For $\mathbf{g}\in H^{3}\left(\mathbb{T}^{N}\right)$, noting that
\begin{align}
\frac{\mu}{\bar{\rho}\left(1+a_{(n)}\right)}\triangle\mathbf{g}=\mathbf{f},
\end{align}
due to $a_{(n)}\in L^{\infty}\left(0, T; C^{1,\ell}\left(\mathbb{T}^{N}\right)\right)$, we have
\begin{align}
&\left\|\left(\frac{\mu}{\bar{\rho}\left(1+a_{(n)}(t,x)\right)}\triangle-\frac{\mu}{\bar{\rho}\left(1+a_{(n)}(s,x)\right)}\triangle\right)
   \left(\frac{\mu}{\bar{\rho}\left(1+a_{(n)}(\tau,x)\right)}\triangle\right)^{-1}\mathbf{f}\right\|\notag\\
=&\left\|\left(\frac{\mu}{\bar{\rho}\left(1+a_{(n)}(t,x)\right)}-\frac{\mu}{\bar{\rho}\left(1+a_{(n)}(s,x)\right)}\right)\triangle\mathbf{g}\right\|\notag\\
=&\frac{\mu}{\bar{\rho}}\left\|\frac{a_{(n)}(s,x)-a_{(n)}(t,x)}{\left(1+a_{(n)}(t,x)\right)\left(1+a_{(n)}(s,x)\right)}\frac{\mu\triangle\mathbf{g}}{\bar{\rho}\left(1+a_{(n)}(\tau,x)\right)}
  \frac{\bar{\rho}\left(1+a_{(n)}(\tau,x)\right)}{\mu}\right\|\\
=&\frac{\mu}{\bar{\rho}}\left\|\frac{a_{(n)}(s,x)-a_{(n)}(t,x)}{\left(1+a_{(n)}(t,x)\right)\left(1+a_{(n)}(s,x)\right)} \frac{\bar{\rho}\left(1+a_{(n)}(\tau,x)\right)}{\mu}\mathbf{f}\right\|\notag\\
\ls &\left\|\frac{a_{(n)}(s,x)-a_{(n)}(t,x)}{\left(1+a_{(n)}(t,x)\right)\left(1+a_{(n)}(s,x)\right)} \left(1+a_{(n)}(\tau,x)\right)\mathbf{f}\right\|\notag\\
\ls & 6\left\|a_{(n)}(s,x)-a_{(n)}(t,x)\right\|\left\|\mathbf{f}\right\|,\notag
\end{align}
which implies that $\A_{(n)}$ satisfies $\left(P_3\right)$ according to ~\eqref{onto mapping of a deterministic}.

Since $\A_{(n)}$ satisfies $\left(P_1\right)-\left(P_3\right)$, then there exists a semigroup to the evolving system satisfing \eqref{time decay of laplacian}, \eqref{analytic property of A}, \eqref{continuity at initial time}. According to Lemma \ref{Semigroup property}, time decay inequalities hold, i.e., there hold \eqref{time decay of laplacian} and \eqref{time decay of fractional derivative of semigroup}.

%Regarding the initial condition \eqref{assum for initial data}, $K_{0}(T)$ is bounded by some constant $C(T)$.
%
%The mild solutions of \eqref{equation for bar w -w} are
%\begin{align}
%\left(\u_{(n)}\right)=&\int_{0}^{t} S(t-s) B\left(\u_{(n-1)},\u_{(n-1)}\right)\d s \\
%&- \int_{0}^{t} S(t-s) \nabla Q\left(a_{(n)},\u_{(n-1)}\right) \d s . \notag
%\end{align}
By Lemma \ref{time decay of semigroup}, there holds
\begin{equation}\label{time decay of qnorm of semigroup}
\left|S(t) f\right|_{q} \ls C t^{-\frac{N}{2 \eth}}\left|f\right|_{p},\\
\end{equation}
 where $\frac{1}{\eth}=\frac{1}{p}-\frac{1}{q}$. Thus, we have
 \begin{equation}
 \left|S(t-s)\u_{0}\right|_{p} \ls M_{1}\left|\u_{0}\right|_{p}.% \ls  M_{1}K_{0}(t).\\
\end{equation}
It holds the time decay property of semigroup
 \begin{equation}
\left|S(t) f\right|_{p} \ls  M_{2} t^{-\frac{N}{2 p}}\left|f\right|_{\frac{p}{2}}.\\
\end{equation}

\subsection{The onto mapping estimates of $\z$}
 For the linear problem \eqref{pl1}, by Lemma \ref{infinitesimal generator}, the mild solutions to \eqref{pl1} are the stochastic convolution
\begin{align}\label{formula of zn}
  \z_{(n)}(t)&= S_{(n)}(t)\z_{(n)}\left(0\right)+\int_0^t S_{(n)}(t-s)\Phi\left(s,x\right)\d W (s)\\
  =& \int_0^t S_{(n)}(t-s)\Phi\left(s,x\right)\d W (s),\notag
\end{align}
where $S_{(n)}(t)$ is the semigroup generated by Stokes operator $\A_{(n)}(t)$.
\begin{lemma} \label{arbitrariness of r}
Let $r$ be an arbitrary large integer. If there exists a constant $\tilde{C}$, such that the following inequality
\begin{align}
\mathbb{E}\left[\left|f(x)\right|^{r}\right]\ls \tilde{C}^{r},
\end{align}
holds, then, $\left|f(x)\right|\ls 2\tilde{C}$ holds $\mathbb{P}$ {\rm a.s.}
\end{lemma}
\begin{proof}
By Chebyshev's inequality, it holds that
\begin{align}
\mathbb{P}\left[\left\{\omega\in \Omega| \left|f(x)\right|>2\tilde{C}\right\}\right]
\ls \frac{\mathbb{E}\left[\left|f(x)\right|^{r}\right]}{\left(2\tilde{C}\right)^{r}}
\ls \left(\frac{\tilde{C}}{2\tilde{C}}\right)^{r}.
\end{align}
Passing the limit $r\ra \infty$, there holds $\mathbb{P}\left[\left\{\omega\in \Omega| \left|f(x)\right|>2\tilde{C}\right\}\right]\ra 0$, i.e., $\left|f(x)\right|\ls 2\tilde{C}$ holds $\mathbb{P}$ {\rm a.s.}
\end{proof}

% We use the Einstein summation convention to introduce the following theorem about $\z$.
 \begin{lemma}\label{Lp for z}
Let $a_{(n)}\in C\left([0, T]; C^{1,\ell}\left(\mathbb{T}^{N}\right)\right)$.
If
%\begin{align}\label{C_Phi_lambda}
%C_{\Phi,\lambda}=\max\left\{\aligned \sum\limits_{k}  \lambda_{j}^{-1}\sup\limits_{s\in[0,t]}\left|\left(\Phi_{k},\tilde{e}_{\mathbf{j}}\right)\right|^{2},\quad & \sum\limits_{k}  \lambda_{j}^{2N-1}\sup\limits_{s\in[0,t]}\left|\left(\Phi_{k},\tilde{e}_{\mathbf{j}}\right)\right|^{2}, \\
%\sum\limits_{k} \lambda_{j}^{N+1}\sup\limits_{s\in[0,t]}\left|\left(\Phi_{k},\tilde{e}_{\mathbf{j}}\right)\right|^{2}\quad & \sum\limits_{k} \lambda_{2j}^{N+1}\sup\limits_{s\in[0,t]}\left|\left(\Phi_{k},\tilde{e}_{2j}\right)\right|^{2}
%\endaligned
%\right\}.
%\end{align}
\begin{align}\label{C_Phi}
C_{\Phi}=\max\limits_{t\in[0,T]}\left\{\sum\limits_{k=1}^{\infty} \left|\Phi_{k}\right|_{\infty}^{2}, ~\sum\limits_{k=1}^{\infty} \left|\nabla \Phi_{k}\right|_{\infty}^{2}, ~\sum\limits_{k=1}^{\infty} \left|\triangle \Phi_{k}\right|_{\infty}^{2}, ~\sum\limits_{k=1}^{\infty} \left|\nabla \triangle \Phi_{k}\right|_{\infty}^{2}\right\}
\end{align}
is bounded,
then for any $T$,
\begin{align}\label{H3 norm of_z}
\left\|\z_{(n)}(t)\right\|_{C\left([0, T]; H^{3}\left(\mathbb{T}^{N}\right)\right)}<C,\quad t\in [0, T],
\end{align}
holds $\mathbb{P}$ {\rm a.s.} in $\left(\Omega, \mathbb{P},\mathcal{F}\right)$, where $C$ only depends on $C_{\Phi}$ and $T$.
\end{lemma}
\begin{proof}
 We use the energy method to give the onto mapping estimates of $\z_{(n)}$. Recall that
 \begin{align}
\d \z_{(n)}(t)- \frac{\mu}{\bar{\rho}\left(1+a_{(n)}\right)}\triangle \z_{(n)}(t) \d t= \Phi \d W(t).
 \end{align}
 By It\^o's formula (see Appendix), there holds
\begin{align}
\d\frac{\left|\z_{(n)}\right|^{2}}{2} = & \z_{(n)} \cdot \d \z_{(n)} + \langle \d \z_{(n)}, \d \z_{(n)}\rangle\notag\\
= &\frac{\mu}{\bar{\rho}\left(1+a_{(n)}\right)}\triangle \z_{(n)}\cdot \z_{(n)} \d t+ \Phi \cdot \z_{(n)} \d W(t)+ \left|\Phi\right|^{2}\d t,
\end{align}
where $\langle \cdot, \cdot \rangle$ is the cross quadratic variation (see Appendix).
For the convenience of energy estimate, we multiply the above formula by velocity. Then, we have
\begin{align}\label{Ito foumula times rho}
 \rho_{(n)} \d \frac{|\z_{(n)}|^{2}}{2}=&\mu\triangle \z_{(n)}\cdot \z_{(n)}\d t 
+ \rho_{(n)} \Psi\cdot\z_{(n)} \d W+ \frac{1}{2} \rho_{(n)} \left|\Psi\right|^{2}\d t. 
\end{align}
We integrate ~\eqref{Ito foumula times rho} with respect to ~$x$. Since ~$\rho_{(n)}$ has a lower bound of ~$c_{0}>0$, as derived from the transport equation, there holds
\begin{align}
\int_{\mathbb{T}^{N}} \rho_{(n)} \d \frac{|\z_{(n)} |^{2}}{2} \d x \geqslant c_{0}\int_{\mathbb{T}^{N}}  \d \frac{|\z_{(n)}|^{2}}{2} \d x,
\end{align}
i.e.,
\begin{align}
\int_{0}^{t}\int_{\mathbb{T}^{N}}  \d \frac{|\z_{(n)}|^{2}}{2} \d x \d s \ls  \frac{1}{c_{0}}\int_{0}^{t}\int_{\mathbb{T}^{N}}\rho_{(n)}\d \frac{|\z_{(n)}|^{2}}{2} \d x \d s.
\end{align}
 Hence, $\int_{0}^{t}\int_{\mathbb{T}^{N}}  \d \frac{|\z_{(n)}|^{2}}{2} \d x \d s$ is bounded by the integral of the right hand side terms in \eqref{Ito foumula times rho}. We estimate the right hand side of \eqref{Ito foumula times rho} term by term.
By manipulating integration by parts, we have
\begin{align}
-\mu \int_{0}^{t}\int_{\mathbb{T}^{N}}\triangle \z_{(n)}\cdot \z_{(n)}  \d x \d s=\mu\int_{0}^{t}\int_{\mathbb{T}^{N}} \left|\nabla \z_{(n)}\right|^{2} \d x \d s.
\end{align}
By stochastic ~Fubini theorem and ~B\"urkh\"older-Davis-Gundy's inequality, for any $r\geqslant 2$, we have
\begin{align}
&\mathbb{E}\left[\left|\int_{0}^{t} \int_{\mathbb{T}^{N}} \rho_{(n)}\Psi \cdot \z_{(n)}  \d W\d x \right|^{r}\right]
 = \mathbb{E}\left[\left|\int_{0}^{t} \int_{\mathbb{T}^{N}} \rho_{(n)}\sum\limits_{k=1}^{\infty}\Psi_{k} \cdot \z_{(n)} \d x \d W \right|^{r}\right]\notag \\
 \ls & \mathbb{E}\left[\left|\int_{0}^{t} \left|\int_{\mathbb{T}^{N}} \rho_{(n)}\sum\limits_{k=1}^{\infty}\Psi_{k} \cdot \z_{(n)} \d x \right|^{2}\d s \right|^{\frac{r}{2}}\right] \notag\\
  \ls & \mathbb{E}\left[\left|C\int_{0}^{t} \left|\int_{\mathbb{T}^{N}}\rho_{(n)} |\z_{(n)}|\d x \right|^{2} \d s \right|^{\frac{r}{2}}\right]\\
\ls & \mathbb{E}\left[\left|C \int_{0}^{t} \left(\left|\z_{(n)}\right|_{2}^{2}+  \left|\rho_{(n)}\right|_{2}^{2}\right) \d s \right|^{\frac{r}{2}}\right] \notag\\
 \ls & C_{r}\mathbb{E}\left[  \left| t\int_{\mathbb{T}^{N}} \rho_{0}^{2}\d x\right|^{r}  \right]+ \mathbb{E}\left[  \int_{0}^{t} \left| C \int_{\mathbb{T}^{N}} |\z_{(n)}|^{2} \d x\right|^{r} \d s \right], \notag
\end{align}
where $C_{r}$ is the $r$-th power of some constant dependent of $\sum\limits_{k=1}^{\infty} \Psi_{k}^{2}$. 
It is straightforward to estimate that
\begin{align}
\mathbb{E}\left[\left|\int_{0}^{t} \int_{\mathbb{T}^{N}} \frac{1}{2}\rho\left|\Psi\right|^{2}\d x\d s \right|^{r}\right] \ls C_{r}\mathbb{E}\left[ \left| t\int_{\mathbb{T}^{N}} \rho_{0}^{2}\d x\right|^{r} \right].
\end{align} 
In summary, the following inequality holds
\begin{align}
&\mathbb{E}\left[\left|\int_{\mathbb{T}^{N}}\frac{1}{2}\left|\z_{(n)}\right|^{2}\d x \right|^{r}\right]+\mathbb{E}\left[\left|\int_{0}^{t}\mu \int_{\mathbb{T}^{N}}\left|\nabla \z_{(n)}\right|^{2} \d x \right|^{r}\right]\\
\ls &\mathbb{E}\left[\left|\int_{\mathbb{T}^{N}}\frac{1}{2}\left|\z_{0}\right|^{2}\d x \right|^{r}\right]+C_{r}\mathbb{E}\left[ \left| t\int_{\mathbb{T}^{N}} \rho_{0}^{2}\d x\right|^{r} \right]+ C\int_{0}^{t}\mathbb{E}\left[\left| \int_{\mathbb{T}^{N}} |\z_{(n)}|^{2} \d x\right|^{r}\right] \d s, \notag
\end{align}
where ~$C_{r}$is the $r$-th power of some constant. 
By Gr\"onwall's inequality, the following energy estimates
\begin{align}
\mathbb{E}\left[\left|\int_{\mathbb{T}^{N}}\frac{1}{2}\left|\z_{(n)}\right|^{2}\d x \right|^{r}\right] \ls  &\mathbb{E}\left[\left|\int_{\mathbb{T}^{N}}\frac{1}{2}\left|\z_{0}\right|^{2}\d x \right|^{r}\right]+ C_{r}t^{r},\\
\mathbb{E}\left[\left|\int_{0}^{t}\int_{\mathbb{T}^{N}}\left|\nabla \z_{(n)}\right|^{2} \d x \right|^{r}\right]  
 \ls &\mathbb{E}\left[\left|\int_{\mathbb{T}^{N}}\frac{1}{2}\left|\z_{0}\right|^{2}\d x \right|^{r}\right]+ C_{r}t^{r},
\end{align}
 hold, where ~$C_{r}$ is the $r$-th power of some constant dependent of $\sum\limits_{k=1}^{\infty} \Psi_{k}^{2}$ and ~$\rho_{0}$. 

%By stochastic Fubini's theorem and B\"urkh\"older-Davis-Gundy's inequality, for any $r\geqslant 2$, there holds
%\begin{align}
%&\mathbb{E}\left[\left|\int_{0}^{t} \int_{\mathbb{T}^{N}} \Phi \cdot \z_{(n)}  \d W\d x \right|^{r}\right]
% = \mathbb{E}\left[\left|\int_{0}^{t} \int_{\mathbb{T}^{N}} \sum\limits_{k=1}^{\infty}\Phi_{k} \cdot \z_{(n)} \d x \d W \right|^{r}\right]\notag \\
% \ls & \mathbb{E}\left[\left|\int_{0}^{t} \left|\int_{\mathbb{T}^{N}} \sum\limits_{k=1}^{\infty}\Phi_{k} \cdot \z_{(n)} \d x \right|^{2}\d s \right|^{\frac{r}{2}}\right]
%  \ls \mathbb{E}\left[\left|C\int_{0}^{t} \left|\int_{\mathbb{T}^{N}}   |\z_{(n)}|\d x \right|^{2} \d s \right|^{\frac{r}{2}}\right] \\
%\ls & \mathbb{E}\left[\left|\sup_{s\in[0,t]}\left\|\z_{(n)}\right\|^{2}  C  \int_{0}^{t} \d s \right|^{\frac{r}{2}}\right]
% \ls  \frac{1}{4} \mathbb{E}\left[\sup_{s\in[0,t]}\left\|\z_{(n)}\right\|^{2r}\right] + Ct^{r}, \notag
%\end{align}
%where $C$ depends on $\sum\limits_{k=1}^{\infty} \left|\Phi_{k}\right|_{\infty}^{2}$.
%It is straightforward to estimate that
%\begin{align}
%\mathbb{E}\left[\left|\int_{0}^{t} \int_{\mathbb{T}^{N}} \frac{1}{2}\left|\Phi \right|^{2}\d x\d s \right|^{r}\right]
%\ls C \mathbb{E}\left[\left|  t \sum\limits_{k=1}^{\infty}\|\Phi_{k}\|^{2} \right|^{r}\right] \ls Ct^{r}.
%\end{align}
In summary, it holds that
\begin{align}
&\mathbb{E}\left[\left|\int_{\mathbb{T}^{N}}\frac{1}{4}\left|\z_{(n)}\right|^{2}\d x +\xi\int_{0}^{t} \int_{\mathbb{T}^{N}}\left|\nabla \z_{(n)}\right|^{2} \d x \d s\right|^{r}\right]
\ls \mathbb{E}\left[\left|\int_{\mathbb{T}^{N}}\frac{1}{4}\left|\z_{0}\right|^{2}\d x \right|^{r}\right]+ C_{1}t^{r},
%\ls &  C_{1} K_{0}(t)^{2}+ C_{1} t K_{0}(t)+\frac{1}{2}K_{0}(t)=K_{0}(t)\left(\frac{1}{2}+ C_{1} K_{0}(t)+C_{1} t K_{0}(t)\right). \notag
\end{align}
where $C_{1}$ is the $r$-th power of some positive constant dependent of $\sum\limits_{k=1}^{\infty} \Psi_{k}^{2}$ and ~$\rho_{0}$. 

We take the first-order derivative with respect to $x$. Hence, we have
\begin{align}
\d \nabla \z_{(n)}(t)- \nabla\left(\frac{\mu}{\bar{\rho}\left(1+a_{(n)}\right)}\triangle \z_{(n)}(t)\right) \d t= \nabla \Phi \d W(t),
 \end{align}
 By It\^o's formula, there holds
\begin{align}
\d\frac{\left|\nabla \z_{(n)}\right|^{2}}{2} = & \nabla\z_{(n)} : \d \nabla \z_{(n)} + \langle \d \nabla\z_{(n)}, \d \nabla\z_{(n)}\rangle\notag\\
= &\nabla\left(\frac{\mu}{\bar{\rho}\left(1+a_{(n)}\right)}\triangle \z_{(n)}\right): \nabla\z_{(n)} \d t+ \nabla \Phi : \z_{(n)} \d W(t)+ \left|\nabla\Phi\right|^{2}\d t,
\end{align}
where $``:"$ means summing the products of the corresponding element.
Due to $\left(1+a\right)\bar{\rho} \geqslant \frac{1}{2}\bar{\rho} $, by the integration by parts, we obtain
\begin{align}
 & -\int_{\mathbb{T}^{N}}\nabla\left(\frac{\mu}{\bar{\rho}\left(1+a_{(n)}\right)}\triangle \z_{(n)}\right)\cdot \nabla\z_{(n)} \d x \d t\notag\\
= &  \int_{\mathbb{T}^{N}}\frac{\mu}{\bar{\rho}\left(1+a_{(n)}\right)}\left|\triangle \z_{(n)}\right|^{2} \d x\d t
\geqslant \frac{\mu}{2\bar{\rho}}\int_{\mathbb{T}^{N}}\left|\triangle \z_{(n)}\right|^{2} \d x\d t. \notag
\end{align}
By stochastic Fubini's theorem and B\"urkh\"older-Davis-Gundy's inequality, for any $r\geqslant 2$, there holds
\begin{align}
&\mathbb{E}\left[\left|\int_{0}^{t} \int_{\mathbb{T}^{N}}\nabla \Phi \cdot \z_{(n)}  \d W\d x \right|^{r}\right]
% = \mathbb{E}\left[\left|\int_{0}^{t} \int_{\mathbb{T}^{N}} \Phi_{k} \cdot \z_{(n)} \d x \d W \right|^{r}\right]\notag \\
% \ls & \mathbb{E}\left[\left|\int_{0}^{t} \left|\int_{\mathbb{T}^{N}} \Phi_{k} \cdot \z_{(n)} \d x \right|^{2}\d s \right|^{\frac{r}{2}}\right]
%  \ls \mathbb{E}\left[\left|\int_{0}^{t} \left|\int_{\mathbb{T}^{N}} C  |\z_{(n)}|\d x \right|^{2} \d s \right|^{\frac{r}{2}}\right] \\
%\ls & \mathbb{E}\left[\left|\sup_{s\in[0,t]}\left|\z_{(n)}\right|_{2}^{2}  C  \int_{0}^{t} \d s \right|^{\frac{r}{2}}\right] \notag \\
 \ls  \frac{1}{4} \mathbb{E}\left[\sup_{s\in[0,t]}\left\|\z_{(n)}\right\|^{2r}\right] + Ct^{r},
\end{align}
where $C$ depends on $\sum\limits_{k=1}^{\infty}|\nabla \Phi_{k}|_{\infty}^{2}$ and ~$\rho_{0}$. 
It is easy to see that
\begin{align}
\mathbb{E}\left[\left|\int_{0}^{t} \int_{\mathbb{T}^{N}} \frac{1}{2}\left|\nabla\Phi \right|^{2}\d x\d s \right|^{r}\right]
\ls C \mathbb{E}\left[\left|  t \sum\limits_{k=1}^{\infty}\|\nabla\Phi_{k}\|_{2}^{2} \right|^{r}\right] \ls Ct^{r}.
\end{align}
In summary, we derive
\begin{align}
&\mathbb{E}\left[\left|\int_{\mathbb{T}^{N}}\frac{1}{4}\left|\nabla\z_{(n)}\right|^{2}\d x +\xi\int_{0}^{t} \int_{\mathbb{T}^{N}}\left|\triangle \z_{(n)}\right|^{2} \d x \d s\right|^{r}\right]
\ls \mathbb{E}\left[\left|\int_{\mathbb{T}^{N}}\frac{1}{4}\left|\nabla\z_{0}\right|^{2}\d x \right|^{r}\right]+ C_{2}t^{r},
%\ls &  C_{2} K_{0}(t)^{2}+ C_{2} t K_{0}(t)+\frac{1}{2}K_{0}(t)=K_{0}(t)\left(\frac{1}{2}+ C_{2} K_{0}(t)+C_{2} t K_{0}(t)\right). \notag
\end{align}
where $C_{2}$ is the $r$-th power of some positive constant, and $C_{2}$ depends on $\sum\limits_{k=1}^{\infty} \left|\nabla \Phi_{k}\right|_{\infty}^{2}$.
Similarly, by It\^o's formula, Burkh\"oler-Davis-Gundy's inequality, and integration by parts, we have the second-order estimates
\begin{align}
&\mathbb{E}\left[\left|\int_{\mathbb{T}^{N}}\frac{1}{4}\left|\triangle \z_{(n)}\right|^{2}\d x +\xi\int_{0}^{t} \int_{\mathbb{T}^{N}}\left|\nabla \triangle \z_{(n)}\right|^{2} \d x \d s\right|^{r}\right]\\
\ls &\mathbb{E}\left[\left|\int_{\mathbb{T}^{N}}\frac{1}{4}\left|\triangle \z_{0}\right|^{2}\d x \right|^{r}\right]+ C_{3}t^{r},\notag
\end{align}
 and the third-order estimates
%\ls &  C_{2} K_{0}(t)^{2}+ C_{2} t K_{0}(t)+\frac{1}{2}K_{0}(t)=K_{0}(t)\left(\frac{1}{2}+ C_{2} K_{0}(t)+C_{2} t K_{0}(t)\right). \notag
\begin{align}
&\mathbb{E}\left[\left|\int_{\mathbb{T}^{N}}\frac{1}{4}\left|\nabla\triangle \z_{(n)}\right|^{2}\d x +\xi\int_{0}^{t} \int_{\mathbb{T}^{N}}\left|\triangle^{2} \z_{(n)}\right|^{2} \d x \d s\right|^{r}\right]\\
\ls &\mathbb{E}\left[\left|\int_{\mathbb{T}^{N}}\frac{1}{4}\left|\nabla\triangle  \z_{0}\right|^{2}\d x \right|^{r}\right]+ C_{4}t^{r},\notag
\end{align}
where $C_{3}$ is the $r$-th power of some positive constant, dependent of $\sum\limits_{k=1}^{\infty} \left|\triangle \Phi_{k}\right|_{\infty}^{2}$ and ~$\rho_{0}$; $C_{4}$ is the $r$-th power of some positive constant, dependent of $\sum\limits_{k=1}^{\infty} \left|\nabla \triangle \Phi_{k}\right|_{\infty}^{2}$ and ~$\rho_{0}$.
By Lemma \ref{arbitrariness of r}, there holds
\begin{align}\label{C5 formula}
&\left\| \z_{(n)}(t)\right\|_{3}^{2} + 4\xi\int_{0}^{t}\left\| \z_{(n)}\right\|_{4}^{2}\d s \ls  C_{5}t,
\end{align}
noticing that we set $\z_{0}=0$. Hence, for any $n$, any $t$,
 we have
 \begin{align}
 \left\| \z_{(n)}(t)\right\|_{3}\ls C_{5}t^{\frac{1}{2}}.
  \end{align}
 $\mathbb{P}$ a.s., where $C_{5}$ depends on $C_{\Phi}$ and ~$\rho_{0}$.
 
Considering the estimates in $[t_{1}, ~t_{2}]$, we have
 \begin{align}
&\left\| \z_{(n)}(t_{2})\right\|_{3}^{2}-\left\| \z_{(n)}(t_{1})\right\|_{3}^{2} + 4\xi \int_{t_{1}}^{t_{2}}\left\| \z_{(n)}\right\|_{4}^{2}\d s \ls C(t_{2}-t_{1}),\\
&\left\| \z_{(n)}(t_{2})\right\|_{3}-\left\| \z_{(n)}(t_{1})\right\|_{3}\ls C(t_{2}-t_{1}),
\end{align}
which gives the time continuity, with $C$ depends on $T$, $C_{\Phi}$ and ~$\rho_{0}$.

\hfill $\square$
\end{proof}

\subsection{Onto mapping estimates of $a$}

Considering the transport equation, along the characteristic line, $a_{(n)}$ is a constant, i.e.,
\begin{equation}
a_{(n)}\left(t,x\right)=a_{0}\left(x-\int_0^t\u_{(n-1)}\d s\right).
\end{equation}

We calculate $\nabla a_{(n)}\left(t,x\right)$ by chain rule,
\begin{align}
&\left|\nabla a_{(n)}\left(t,x\right)\right|=\left|\nabla a_{0}\left(x-\int_0^t\u_{(n-1)}\d s\right)\right|
=\left|\nabla a_{0}\left(x\right)\left(1-\int_0^t\nabla\u_{(n-1)}(x)\d s\right) \right|.
\end{align}
%\begin{align}
%&\left|\nabla a_{(n)}\left(t,x_{1}\right)-\nabla a_{(n)}\left(t,x_{2}\right)\right|\notag\\
%=&\left|\nabla a_{0}\left(x_{1}-\int_0^t\u_{(n-1)}\d s\right)-\nabla a_{0}\left(x_{2}-\int_0^t\u_{(n-1)}\d s\right)\right|\\
%=&\left|\nabla a_{0}\left(x_{1}\right)\left(1-\int_0^t\nabla\u_{(n-1)}(x_{1})\d s\right) -\nabla a_{0}\left(x_{2}\right)\left(1-\int_0^t\nabla\u_{(n-1)}(x_{2})\d s\right)\right|.\notag
%\end{align}
%with $a_{(0)}(x)\in H^{3}\left(\mathbb{T}^{N}\right)$, and $\nabla a_{0}(x) \in C^{0, \ell}\left(\mathbb{T}^{N}\right)$, where $\ell$ is a positive constant.% If $\u_{(n-1)}(s,x) \in H^{3}\left(\mathbb{T}^{N}\right)$ compactly embedded in $C^{1,\ell}\left(\mathbb{T}^{N}\right)$, then by interpolation, we have $ \nabla a_{(n)}(t,x)\in C^{0, \ell}\left(\mathbb{T}^{N}\right)$ for $t$ in a finite time.
Moreover, we calculate
\begin{align}\label{onto mapping of a deterministic}
&\left|\nabla^{2} a_{(n)}\left(t,x\right)\right|_{p}=\left|\nabla^{2} a_{0}\left(x-\int_0^t\u_{(n-1)}\d s\right)\left(t,x\right)\right|_{p}\notag\\
=&\left|\nabla \left(\nabla a_{0}(x)\left(1-\int_0^t \nabla \u_{(n-1)}\d s\right)\right)\right|_{p}\notag\\
\ls &\left|\nabla^{2} a_{0}(x)\left(1-\int_0^t \nabla \u_{(n-1)}\d s\right)\right|_{p}
+ \left|\nabla a_{0}\left(x\right) \left(\int_0^t\nabla^{2}\u_{(n-1)}(x) \d s \right) \right|_{p}\\
\ls &\left|\nabla^{2} a_{0}(x)\right|_{p}\left(1+ t \sup_{s\in[0,t]}\left|\nabla \u_{(n-1)}\right|_{\infty}\right)
+ \left|\nabla a_{0}\left(x\right)\right|_{\infty}\left( t \sup_{s\in[0,t]}\left| \nabla^{2}\u_{(n-1)}\left(x\right)\right|_{p} \right)\notag\\
\ls &\left|\nabla^{2} a_{0}(x)\right|_{p}\left(1+ t C_{6}\sup_{s\in[0,t]}\left|\nabla^{2} \u_{(n-1)}\right|_{p}\right), \notag
\end{align}
by Sobolev's inequality with $p\geqslant N$, where $C_{6}$ depends on $N$.
With $a_{(0)}(x)\in W^{2,p}\left(\mathbb{T}^{N}\right)$, $\u_{(n-1)}(s,x) \in W^{2,p}\left(\mathbb{T}^{N}\right)$, by H\"older's inequality, we obtain that $a_{(n)}(x)\in W^{2,p}\left(\mathbb{T}^{N}\right)$ for any fixed $t$. By Sobolev's inequality with $p\geqslant N$, it holds $a_{(n)}(x)\in C^{1,\ell}\left(\mathbb{T}^{N}\right)$, for any $t$ fixed, $\ell>0$.

%Moreover, we let
%\begin{align}\label{def of K0}
% K_{0}=\max \left\{2 \left|a_{0}\right|_{2,p}, 2 \left|\u_{0}\right|_{2,p}, 1 \right\},
%\end{align}
%and we have the following lemma.
%\begin{lemma}
%If there exists a constant $M$ such that
%\begin{align}
%\sup\limits_{s\in[0,t]}\left|\u_{(n-1)}\right|_{2,p}\ls 2 M K_{0},
%\end{align}
%then it holds
%\begin{align}
%\sup\limits_{s\in[0,t]}\left|a_{(n-1)}\right|_{2,p}\ls K_{0},
% \end{align}
%with $t\ls T_{0}$ and
%\begin{align}
%2 T_{0} C_{6} M K_{0} \ls 1.
%\end{align}
%\end{lemma}

Let
\begin{align}\label{def of K0 stochastic}
K_{0}=\max \left\{ 2 \left|a_{0}\right|_{2,p},2 \left|\u_{0}\right|_{2,p}, C_{5}, 1 \right\},
\end{align}
where $C_{5}$ is from \eqref{C5 formula}.
Sequentially, we have the following lemma.
\begin{lemma}
If there exists a constant $M$ such that
\begin{align}
\sup\limits_{s\in[0,t]}\left|\v_{(n-1)}\right|_{2,p} \ls 2 M K_{0},
\end{align}
then it holds that
\begin{align}
\sup\limits_{s\in[0,t]}\left|a_{(n)}\right|_{2,p}\ls K_{0},
 \end{align}
with $t \ls T_{0}$ and
\begin{align}
2 T_{0} C_{6} M K_{0} + T_{0}^{\frac{1}{2}}K_{0} \ls 1.
\end{align}
\end{lemma}
From \eqref{onto mapping of a deterministic}, by Sobolev's embedding with $N<p\ls 6$, we know that
\begin{align}
\left|a_{(n)}\left(t,x\right)\right|_{2,p} \ls \left| a_{0}(x)\right|_{2,p} \left(1+ t C_{6}\left(\left|\v_{(n-1)}(s,x)\right|_{2,p} + \left\|\z_{(n-1)}\left(t,x\right)\right\|_{3}\right)\right).
\end{align}
Then, by the definition of $K_{0}$ \eqref{def of K0 stochastic} and \eqref{C5 formula}, we have
\begin{align}\label{onto map of a sto}
\left|a_{(n)}\left(t,x\right)\right|_{2,p} \ls \left| a_{0}(x)\right|_{2,p} \left(1+ t C_{6}2MK_{0} + K_{0}t^{\frac{1}{2}}\right).
\end{align}
The continuity of $\left|a_{(n)}\left(t,x\right)\right|_{2,p} $ in time also holds from the above estimate.

\subsection{The onto mapping estimates of $\v$}\label{Onto mapping in a ball}

  For the linearized system \eqref{linearized system}, by Lemma \ref{infinitesimal generator}, $\v_{(n)}$ is written by the convolution
\begin{align}\label{mild solutions u n}
  \v_{(n)}(t) =& S_{(n)}(t)\v_{(n)}\left(0\right)+\int_0^t S_{(n)}(t-s)B\left(\u_{(n-1)}, \u_{(n-1)}\right)\left(s,x\right)\d s \notag\\
  &+\int_0^t S_{(n)}(t-s)\nabla Q\left(a_{(n)},\u_{(n-1)}\right)\left(s,x\right)\d s\notag\\
  =&  S(t-s)\u_{0}+\int_0^t S_{(n)}(t-s)B\left(\u_{(n-1)}, \u_{(n-1)}\right)\left(s,x\right)\d s \\
   &+\int_0^t S_{(n)}(t-s)\nabla Q\left(a_{(n)},\u_{(n-1)}\right))\left(s,x\right)\d s, \notag
\end{align}
where $S_{(n)}(t)$ is the semigroup generated by Stokes operator $\A_{(n)}(t)$ and $B\left(\mathbf{a}, \mathbf{b}\right)\triangleq -\left(\mathbf{a} \cdot \nabla\right)\mathbf{b}$.

In the light of Banach's fixed point theorem, we need to prove the uniform bound of the mapping in a ball, i.e., the mapping from the solution at step $n-1$ to the solution at step $n$ is onto. Due to the nonlinear term, if we want estimate the $\left|\cdot\right|_{p} $-norm or $\left|\cdot\right|_{2,p} $-norm, $p\neq2$, then the semigroup method is more appropriate than the direct energy estimates.

 Next, we will prove the following lemma on ~$\v_{(n)}$.
  \begin{lemma}
  There exists a constant $M$ only depending on $p$ such that, if $\sup\limits_{s\in[0,t]} \left|\v_{(n-1)}\right|_{2,p}\ls 2MK_{0}$ in $C\left([0,T_{0}];W^{2,p}\left(\mathbb{T}^{N}\right)\right)$, then it holds $\sup\limits_{s\in[0,t]}  \left|\v_{(n)}\right|_{2,p} \ls 2M K_{0}$ in $C\left([0, T_{0}]; W^{2,p}\left(\mathbb{T}^{N}\right)\right)$.
\end{lemma}
\begin{proof}
The mild solutions of \eqref{equation for bar w -w} are
\begin{align}
\left(\v_{(n)}\right)=&\int_{0}^{t} S(t-s) B\left(\v_{(n-1)}+\z_{(n-1)},\v_{(n-1)}+\z_{(n-1)}\right)\d s \\
&- \int_{0}^{t} S(t-s) \nabla Q\left(a_{(n)},\u_{(n-1)}\right) \d s + S(t-s)\v_{0}. \notag
\end{align}
By \eqref{time decay of qnorm of semigroup}, we have
 \begin{equation}
 \left|S(t-s)\v_{0}\right|_{p} \ls M_{1}\left|\v_{0}\right|_{p}=M_{1}\left|\u_{0}\right|_{p}. %\ls  M_{1}K_{0}(t),\\
\end{equation}
For any function $f\in L^{\frac{p}{2}}$, there holds
 \begin{equation}
\left|S(t) f\right|_{p} \ls  M_{2} t^{-\frac{N}{2 p}}\left|f\right|_{\frac{p}{2}}.\\
\end{equation}
\underline{\textbf{Estimates for the bilinear term:}}\\
For zero-average function $f$, as mentioned in Buckmaster \cite{Buckmaster2015anomalous}, there holds $\triangle^{-\frac{1}{2}} f=\odiv^{-1}f$, $\nabla\cdot \triangle^{-\frac{1}{2}} f= f$.
By \eqref{time decay of fractional derivative of semigroup}, $S(t-s)\triangle^{\frac{1}{2}}=\triangle^{\frac{1}{2}}S(t-s)$, in the periodic domain we estimate
\begin{align}\label{esti for S(t-s)B(u u)}
& \left|\int_{0}^{t} S(t-s)B\left(\u_{(n-1)},\u_{(n-1)}\right)\d s\right|_{p}
=  \left|\int_{0}^{t} S(t-s) \left(-\nabla\cdot \left(\u_{(n-1)}\otimes \u_{(n-1)}\right)(s)\right)\d s\right|_{p} \notag\\
=&  \left|\int_{0}^{t} S(t-s) \triangle^{\frac{1}{2}}\triangle^{-\frac{1}{2}}\left(-\nabla\cdot \left(\u_{(n-1)}\otimes \u_{(n-1)}\right)(s)\right)\d s\right|_{p} \notag\\
= &  \left|\int_{0}^{t} \triangle^{\frac{1}{2}}S(t-s) \triangle^{-\frac{1}{2}}\left(-\nabla\cdot \left(\u_{(n-1)}\otimes \u_{(n-1)}\right)(s)\right)\d s\right|_{p} \notag\\
\ls &  \int_{0}^{t} \left|\triangle^{\frac{1}{2}} S(t-s) \u_{(n-1)}\otimes \u_{(n-1)}(s)\right|_{p} \d s \\
\ls &  \int_{0}^{t} \frac{M_{1}}{(t-s)^{\frac{1}{2}}}\left| S(t-s) \u_{(n-1)}\otimes \u_{(n-1)}(s)\right|_{p}\d s \notag\\
\ls &  \int_{0}^{t} \frac{M_{1}M_{2}}{(t-s)^{\frac{1}{2}+\frac{N}{2p}}} \left| \u_{(n-1)}\otimes \u_{(n-1)}(s)\right|_{\frac{p}{2}}\d s
\ls   \left(\int_{0}^{t} \frac{M_{1}M_{2}}{(t-s)^{\frac{1}{2}+\frac{N}{2p}}} \d s\right) \sup\limits_{s\in [0,t]} \left| \u_{(n-1)}\otimes \u_{(n-1)}(s)\right|_{\frac{p}{2}} \notag\\
= & t^{\frac{1}{2}\left(1-\frac{N}{p}\right)}M_{1}M_{2}    \sup\limits_{s\in [0,t]} \left|
\u_{(n-1)}\otimes \u_{(n-1)}(s)\right|_{\frac{p}{2}}
\ls  t^{\frac{1}{2}\left(1-\frac{N}{p}\right)}M \sup\limits_{s\in [0,t]}\left| \u_{(n-1)}(s)\right|_{p}^{2},\notag
\end{align}
where the condition $p>N$ guarantees the time integral $\int_{0}^{t} \frac{M}{(t-s)^{\frac{1}{2}+\frac{N}{2p}}} \d s$ bounded for any $t\ls T$, with $M=\max\left\{M_{1},M_{2},M_{1}M_{2},2\right\} $.

\noindent \underline{\textbf{Estimates for the pressure term:}}\\
By acting a divergence on the momentum equation in \eqref{linearized system}, combined with the incompressible condition, the pressure term satisfies
\begin{align}\label{equation of Q}
\left(\nabla \u_{(n-1)}: \nabla \u_{(n-1)}^{T}-\frac{\mu\nabla \rho_{(n)}}{\rho_{(n)}^{2}}\triangle \v_{(n-1)}+\triangle Q\left(a_{(n)},\u_{(n-1)}\right)\right)\d t = 0. %\nabla\cdot \Phi \d W(t),
\end{align}
This gives that
\begin{align}\label{equation of nabla Q}
 \nabla Q\left(a_{(n)},\u_{(n-1)}\right)\d t
 = -\nabla \triangle^{-1}\left(\nabla \u_{(n-1)}: \nabla \u_{(n-1)}^{T}-\frac{\mu\nabla \rho_{(n)}}{\rho_{(n)}^{2}}\triangle \v_{(n-1)}\right)\d t. %+ \nabla\triangle^{-1}\nabla\cdot \Phi \d W(t).
\end{align}
The pressure term is estimated as follows
\begin{align}\label{pressure term is estimated by}
    & \left| \int_{0}^{t} S(t-s) \nabla Q\left(a_{(n)},\u_{(n-1)}\right) \d s \right|_{p} \notag \\
\ls &  \left|  \int_{0}^{t}S(t-s)\left(-\nabla \triangle^{-1}\left( \nabla \cdot \left( \u_{(n-1)} \cdot \nabla \u_{(n-1)} \right)\d s -\frac{\mu\nabla \rho_{(n)}}{\rho_{(n)}^{2}}\triangle \v_{(n-1)}\right)\d s\right) \right|_{p} \\
%&+ \left|  \int_{0}^{t}S(t-s)\left(\nabla\triangle^{-1}\nabla\cdot \Phi \d W(s)\right)\right|_{p} \\
\ls &  \left|\int_{0}^{t} S(t-s)\nabla \triangle^{-1} \nabla \cdot \left( \u_{(n-1)} \cdot \nabla \u_{(n-1)} \right)\d s\right|_{p} +\left|\int_{0}^{t}  S(t-s)\nabla  \triangle^{-1} \nabla\cdot\left(\frac{\mu}{\rho_{(n)}}\triangle \v_{(n-1)}\right)\d s \right|_{p}.\notag % + \left|\int_{0}^{t} S(t-s) \nabla\triangle^{-1}\nabla\cdot \Phi \d W(s)\right|_{p}, \notag
\end{align}
For $\int_{0}^{t} \left| S(t-s)\nabla \triangle^{-1} \nabla \cdot \left( \u_{(n-1)} \cdot \nabla \u_{(n-1)} \right)\right|_{p}\d s$, since  $\u_{(n-1)} \cdot \nabla \u_{(n-1)} =\nabla \cdot \left(\u_{(n-1)}\otimes \u_{(n-1)}\right)$, similarly to \eqref{esti for S(t-s)B(u u)}, we have
\begin{align}
&\int_{0}^{t} \left| S(t-s)\nabla \triangle^{-1} \nabla \cdot \left( \u_{(n-1)} \cdot \nabla \u_{(n-1)} \right)\right|_{p}\d s \notag \\
=&\int_{0}^{t}\left| S(t-s)\nabla \left(\u_{(n-1)}\otimes \u_{(n-1)}\right)\right|_{p}\d s \\%= \left|\int_{0}^{t} S(t-s)B\left(\u_{(n-1)},\u_{(n-1)}\right)\d s\right|_{p} \\
\ls & t^{\frac{1}{2}\left(1-\frac{N}{p}\right)}M \left(\sup\limits_{s\in [0,t]}\left(\left| \v_{(n-1)}(s)\right|_{p}+\left| \z_{(n-1)}(s)\right|_{p}\right)\right)^{2}.\notag
\end{align}

 For $\left|\int_{0}^{t}  S(t-s) \nabla  \triangle^{-1} \nabla\cdot\left(\frac{\mu}{\rho_{(n)}}\triangle \v_{(n-1)}\right)\d s \right|_{p}=\left|\int_{0}^{t}  S(t-s)\nabla  \triangle^{-\frac{1}{2}} \left(\frac{\mu}{\rho_{(n)}}\triangle \v_{(n-1)}\right)\d s \right|_{p}$, due to $\frac{\mu}{\rho_{(n)}}\triangle \v_{(n-1)}=\nabla \cdot \left(\frac{\mu}{\rho_{(n)}}\nabla\v_{(n-1)}\right)-\nabla \frac{\mu}{\rho_{(n)}}\nabla\v_{(n-1)}$, there holds
\begin{align}\label{nabla cdot (mu rho triangle u) in Q esti}
&\left|\int_{0}^{t}  S(t-s) \nabla  \triangle^{-1} \nabla\cdot\left(\frac{\mu}{\rho_{(n)}}\triangle \v_{(n-1)}\right)\d s \right|_{p} \notag \\
\ls & \int_{0}^{t} \left| S(t-s) \nabla \left(\frac{\mu}{\rho_{(n)}}\sum_{i}\sum_{j}\left(\nabla \v_{(n-1)}\right)_{i,j}\right)\right|_{p}\d s\\
 &+ \int_{0}^{t} \left| S(t-s) \nabla \triangle^{-\frac{1}{2}} \left(\nabla \frac{\mu}{\rho_{(n)}}\nabla\v_{(n-1)}\right)\right|_{p}\d s.\notag
\end{align}
The first term in \eqref{nabla cdot (mu rho triangle u) in Q esti} is estimated as
\begin{align}
& \int_{0}^{t} \left| S(t-s)\nabla \left(\frac{\mu}{\rho_{(n)}}\sum_{i}\sum_{j}\left(\nabla \v_{(n-1)}\right)_{i,j}\right)\right|_{p}\d s\notag\\
\ls & \int_{0}^{t} \frac{M_{1}}{(t-s)^{\frac{1}{2}}} \left|S(t-s)\left(\frac{\mu}{\rho_{(n)}}\sum_{i}\sum_{j}\left(\nabla \v_{(n-1)}\right)_{i,j}\right)\right|_{p}\d s \\
\ls & \int_{0}^{t} \frac{M}{(t-s)^{\frac{1}{2}+\frac{N}{2p}}} \left|\left(\frac{\mu}{\rho_{(n)}}\sum_{i}\sum_{j}\left(\nabla \v_{(n-1)}\right)_{i,j}\right)\right|_{\frac{p}{2}}\d s\notag\\
\ls & \frac{4\mu M}{\bar{\rho}}t^{\frac{1}{2}\left(1-\frac{N}{p}\right)}\sup\limits_{s\in[0,t]}\left|a_{(n)}\right|_{p}\sup\limits_{s\in[0,t]}\left|\sum_{i}\sum_{j}\left(\nabla \v_{(n-1)}\right)_{i,j}\right|_{p}.\notag
%\ls & t^{\frac{1}{2}\left(1-\frac{N}{p}\right)}\frac{8\mu M^{2}}{\bar{\rho}}K_{0}(t)^{2}. \notag
\end{align}
In the following statement, we will use $\left|\nabla \v_{(n-1)}\right|_{p}$ to denote $\left|\sum_{i}\sum_{j}\left(\nabla \v_{(n-1)}\right)_{i,j}\right|_{p}$.
%\begin{align}
% \int_{0}^{t} \left\| S(t-s)\nabla \frac{\mu}{\rho_{(n)}}\nabla\v_{(n-1)}\right\|_{p}\d s =  \int_{0}^{t} \left\| S(t-s)\triangle^{\frac{1}{2}}\triangle^{-\frac{1}{2}}\left( \nabla \frac{\mu}{\rho_{(n)}}\nabla\v_{(n-1)}\right)\right\|_{p}\d s,
%\end{align}
To estimate the second term in \eqref{nabla cdot (mu rho triangle u) in Q esti},
we shall consider $\left|\triangle^{-\frac{1}{2}} \left(\nabla\frac{\mu}{\rho_{(n)}}\cdot \nabla \v_{(n-1)}(s,x)\right)\right| $. Since $\nabla \frac{\mu}{\rho_{(n)}} \in C\left([0, T];  C^{0,\ell}\left(\mathbb{T}^{N}\right)\right)$, for any constant $\delta$ such that for $x_{1},x_{2}\in \mathbb{T}^{N}$, $\left|x_{1}-x_{2}\right| \ls \delta^{\frac{1}{\ell}}$, we have
\begin{align}
\sup\limits_{s\in[0,t]}\frac{\left|\nabla a_{(n)}\left(t,x_{1}\right)-\nabla a_{(n)}\left(t,x_{2}\right)\right|}{\left|x_{1}-x_{2}\right|^{\ell}} \ls  C .
\end{align}
Moreover, the anti-derivative holds
\begin{align}
 & \left|\int_{x_{0}}^{x_{0}+\delta}\left(\nabla\frac{\mu}{\rho_{(n)}}\nabla \v_{(n-1)}\right)(s)\d x\right| \ls \int_{x_{0}}^{x_{0}+\delta}\left|\nabla\frac{\mu}{\rho_{(n)}}\right|(s) \left|\nabla \v_{(n-1)}\right|(s)\d x.
\end{align}
By the mean value theorem, since $\left|\nabla\frac{\mu}{\rho_{(n)}}\right|$ is continuous with respect to $x$, we get
\begin{align}
& \int_{x_{0}}^{x_{0}+\delta}\left|\nabla\frac{\mu}{\rho_{(n)}}\right| (s) \left|\nabla \v_{(n-1)}\right|(s)\d x\notag \\
= &  \left|\nabla\frac{\mu}{\rho_{(n)}}\left(s,\xi\right)\right|\int_{x_{0}}^{x_{0}+\delta}\left|\nabla \v_{(n-1)}\right|\d x\notag \\
= &  \left|\nabla\frac{\mu}{\rho_{(n)}}\left(s,\xi\right)\right|\left( \left( \v_{(n-1)}\left(s, x_{0}+\delta\right)-\v_{(n-1)}\left(s,x_{0}\right)\right)I_{\nabla \v_{(n-1)}(s)\geqslant 0}\right)\\
&-\left|\nabla\frac{\mu}{\rho_{(n)}}\left(s,\xi\right)\right|\left( \left( \v_{(n-1)}\left(s, x_{0}+\delta\right)-\v_{(n-1)}\left(s,x_{0}\right)\right)I_{\nabla \v_{(n-1)}(s)\ls 0}\right) \notag\\
=& \left|\nabla\frac{\mu}{\rho_{(n)}}\left(s,\xi\right)\right|\left( \v_{(n-1)}\left(s, x_{0}+\delta\right)I_{\nabla \v_{(n-1)}(s)\geqslant 0}- \v_{(n-1)}\left(s, x_{0}+\delta\right)I_{\nabla \v_{(n-1)}(s)\ls 0}\right)\notag\\
 &-\left|\nabla\frac{\mu}{\rho_{(n)}}\left(s,\xi\right)\right|\left( \v_{(n-1)}\left(s,x_{0}\right)I_{\nabla \v_{(n-1)}(s)\geqslant 0}-\v_{(n-1)}\left(s,x_{0}\right)I_{\nabla \v_{(n-1)}(s)\ls 0}\right) \notag\\
\triangleq &G\left(s,x_{0}+\delta\right)-G\left(s,x_{0}\right) \notag,
\end{align}
where $\xi\in [x_{0}, x_{0}+\delta]$.
On the other hand, it holds
\begin{align}\label{div inverse analysis1}
 &G\left(s,x_{0}\right)\notag\\
\ls & \left|\nabla\frac{\mu}{\rho_{(n)}}\left(s,\xi\right)-\nabla\frac{\mu}{\rho_{(n)}}\left(s, x_{0}\right)\right|\notag\\
&\qquad\qquad  \left( \v_{(n-1)}\left(s, x_{0}\right)I_{\nabla \v_{(n-1)}(s)\geqslant 0}- \v_{(n-1)}\left(s, x_{0}\right)I_{\nabla \v_{(n-1)}(s)\ls 0}\right)\\
  &+\left|\nabla\frac{\mu}{\rho_{(n)}}\left(s, x_{0}\right)\right|\left(\v_{(n-1)}\left(s, x_{0}\right) I_{\nabla \v_{(n-1)}(s)\geqslant 0}- \v_{(n-1)}\left(s, x_{0}\right)I_{\nabla \v_{(n-1)}(s)\ls 0}\right)\notag\\
  \ls & C\delta  \left|2\v_{(n-1)}\left(s, x_{0}\right)\right| + \left|\nabla\frac{\mu}{\rho_{(n)}}\left(s, x_{0}\right)\right| \left|2\v_{(n-1)}\left(s, x_{0}\right)\right|. \notag
\end{align}
Hence, by taking $\delta =\sup\limits_{s \in [0,t]} \frac{ \left|\nabla\frac{\mu}{\rho_{(n)}}\left(s, x_{0}\right)\right|}{C}$, we get
\begin{align}\label{div inverse analysis2}
 G\left(s,x_{0}\right)\ls 2 \sup\limits_{s \in [0,t]}\left|\nabla\frac{\mu}{\rho_{(n)}}\left(s, x_{0}\right)\right| \left|2\v_{(n-1)}\left(s, x_{0}\right)\right|.
\end{align}
Thus, we have
\begin{align}\label{div inverse analysis}
\left|\triangle^{-\frac{1}{2}} \left(\nabla\frac{\mu}{\rho_{(n)}}\cdot \nabla \v_{(n-1)}(s)\right)\right| \ls 4\sup\limits_{s \in [0,t]}\left|\nabla\frac{\mu}{\rho_{(n)}}\left(s, x \right)\right| \left| \v_{(n-1)}\left(s, x \right)\right|.
\end{align}
Therefore, we obtain the estimate for the second term in \eqref{nabla cdot (mu rho triangle u) in Q esti}
\begin{align}
  & \int_{0}^{t} \left| S(t-s)\nabla \triangle^{-\frac{1}{2}}\left(\nabla \frac{\mu}{\rho_{(n)}}\nabla\v_{(n-1)}\right)\right|_{p}\d s  \notag\\
\ls & \int_{0}^{t} \frac{M_{1}}{(t-s)^{\frac{1}{2}}} \left| S(t-s) \triangle^{-\frac{1}{2}}\left(\nabla \frac{\mu}{\rho_{(n)}}\nabla\v_{(n-1)}\right)\right|_{p}\d s   \\
\ls &  \int_{0}^{t} \frac{M}{(t-s)^{\frac{1}{2}+\frac{N}{2p}}}\left|\triangle^{-\frac{1}{2}}\left(\nabla \frac{\mu}{\rho_{(n)}}\nabla\v_{(n-1)}\right)\right|_{\frac{p}{2}}\d s \notag\\
%\ls & t^{\frac{1}{2}\left(1-\frac{N}{p}\right)}4M\sup\limits_{s \in [0,t]}\left|\left(\nabla\frac{\mu}{\rho_{(n)}}\left(s, x \right)\right) \left| \v_{(n-1)}\left(s, x \right)\right|\right|_{\frac{p}{2}} \notag\\
\ls & t^{\frac{1}{2}\left(1-\frac{N}{p}\right)}\frac{8\mu M}{\bar{\rho}}\sup\limits_{s \in [0,t]}\left|\left(\nabla a_{(n)}\right)\right|_{p} \left| \v_{(n-1)}\left(s, x \right)\right|_{p}. \notag %\ls t^{\frac{1}{2}\left(1-\frac{N}{p}\right)}\frac{8\mu M^{2}}{\bar{\rho}}K_{0}(t)^{2}.\notag
\end{align}
In summary, we have
\begin{align}
  & \int_{0}^{t} \left| S(t-s)\nabla \triangle^{-\frac{1}{2}}\left( \frac{\mu}{\rho_{(n)}}\triangle \v_{(n-1)}\right)\right|_{p}\d s \\
 \ls &t^{\frac{1}{2}\left(1-\frac{N}{p}\right)}\frac{4\mu M}{\bar{\rho}}\sup\limits_{s \in [0,t]}\left|\left( a_{(n)}\right)\right|_{p} \left| \nabla \v_{(n-1)}\right|_{p}
 +t^{\frac{1}{2}\left(1-\frac{N}{p}\right)}\frac{8\mu M}{\bar{\rho}}\sup\limits_{s \in [0,t]}\left|\left(\nabla a_{(n)}\right)\right|_{ p} \left| \v_{(n-1)}\right|_{p}. \notag
\end{align}

\noindent\underline{\textbf{Combined with the estimates of bilinear term \eqref{esti for S(t-s)B(u u)} and the pressure \eqref{pressure term is estimated by}}}, \\
it holds
\begin{align}
&\left| \left(\v_{(n)}\right)(t)\right|_{p}\notag\\
\ls &M_{1}\left|\u_{0}\right|_{p}+ t^{\frac{1}{2}\left(1-\frac{N}{p}\right)} 2M \left(\sup\limits_{s\in [0,t]}\left(\left| \v_{(n-1)}(s)\right|_{p}+\left| \z_{(n-1)}(s)\right|_{p}\right)\right)^{2}\\
&+t^{\frac{1}{2}\left(1-\frac{N}{p}\right)}\frac{4\mu M}{\bar{\rho}}\sup\limits_{s \in [0,t]}\left|\left( a_{(n)}\right)\right|_{p} \left| \nabla \v_{(n-1)}\right|_{p}
+t^{\frac{1}{2}\left(1-\frac{N}{p}\right)}\frac{8\mu M}{\bar{\rho}}\sup\limits_{s \in [0,t]}\left|\nabla \left( a_{(n)}\right)\right|_{p} \left| \v_{(n-1)}\right|_{p}.\notag
 \end{align}
Similarly, it holds the higher order estimates
\begin{align}
& \left|\left(\nabla \v_{(n)}\right)(t)\right|_{p} \notag \\
\ls & M_{1}\left|\nabla \u_{0}\right|_{p} \notag \\
 &+ t^{\frac{1}{2}\left(1-\frac{N}{p}\right)} 2 M \left(2\sup\limits_{s\in [0,t]}\left|\nabla \u_{(n-1)}(s)\right|_{p}^{2} + 2\sup\limits_{s\in [0,t]} \left|\nabla \cdot \u_{(n-1)}(s)\right|_{p}\sup\limits_{s\in [0,t]} \left|\u_{(n-1)}(s)\right|_{p}\right) \\
&+t^{\frac{1}{2}\left(1-\frac{N}{p}\right)}\frac{4\mu M}{\bar{\rho}}\sup\limits_{s \in [0,t]}\left|\left( a_{(n)}\right)\right|_{p} \left| \nabla \cdot \nabla \v_{(n-1)}\right|_{p} \notag\\
&+t^{\frac{1}{2}\left(1-\frac{N}{p}\right)}\frac{12\mu M}{\bar{\rho}}\sup\limits_{s \in [0,t]}\left|\nabla\left( a_{(n)}\right)\right|_{p} \left|\nabla \v_{(n-1)}\right|_{p}
+t^{\frac{1}{2}\left(1-\frac{N}{p}\right)}\frac{8\mu M}{\bar{\rho}}\sup\limits_{s \in [0,t]}\left|\nabla \cdot \nabla a_{(n)}\right|_{p} \left| \v_{(n-1)}\right|_{p}.  \notag
\end{align}

For the $W^{2,p}$-norm,
from the expression of the mild solutions of \eqref{mild solutions u n}, there holds
\begin{align}
\left(\nabla \cdot \nabla \u_{(n)}\right)=&\int_{0}^{t} \nabla \cdot \nabla  S(t-s) B\left(\u_{(n-1)},\u_{(n-1)}\right)\d s  \\
&- \int_{0}^{t} \nabla \cdot \nabla  S(t-s)  \nabla Q\left(a_{(n)},\u_{(n-1)}\right) \d s + \nabla \cdot \nabla  S(t-s)\u_{0}.\notag
\end{align}
 Similarly we have
 \begin{equation}
 \left|S(t-s)\nabla \cdot \nabla  \u_{0}\right|_{p} \ls M_{1}\left|\nabla \cdot \nabla  \u_{0}\right|_{p} \ls  M K_{0}.\\
\end{equation}
By \eqref{time decay of fractional derivative of semigroup}, we estimate the bilinear term as
\begin{align}\label{2 order esti for S(t-s)B(u u)}
& \left|\int_{0}^{t} S(t-s)\nabla \cdot \nabla  B\left(\u_{(n-1)} ,\u_{(n-1)} \right)\d s\right|_{p} \notag\\
= &  \left|\int_{0}^{t} S(t-s) \nabla \cdot \nabla  \left(-\nabla \cdot \left(\u_{(n-1)}\otimes \u_{(n-1)}\right) \right)\d s\right|_{p} \notag\\
= &  \left|\int_{0}^{t} \triangle^{\frac{1}{2}}S(t-s) \triangle^{-\frac{1}{2}}\left(-\nabla \cdot \nabla \nabla \cdot \left(\u_{(n-1)}\otimes \u_{(n-1)}\right)\right)\d s\right|_{p} \notag\\
\ls &  \int_{0}^{t} \left|\triangle^{\frac{1}{2}} S(t-s) \nabla \nabla \cdot \left(\u_{(n-1)}\otimes \u_{(n-1)}\right)\right|_{p} \d s \notag\\
\ls &  \int_{0}^{t} \frac{M_{1}}{(t-s)^{\frac{1}{2}}}\left| S(t-s)\nabla  \nabla \cdot \left(\u_{(n-1)}\otimes \u_{(n-1)}\right)\right|_{p}\d s \\
\ls &  \int_{0}^{t} \frac{M_{1}M_{2}}{(t-s)^{\frac{1}{2}+\frac{N}{2p}}} \left|\nabla \nabla \cdot \left(\u_{(n-1)}\otimes \u_{(n-1)}\right)\right|_{\frac{p}{2}}\d s \notag\\
%\ls &  \left(\int_{0}^{t} \frac{M_{1}M_{2}}{(t-s)^{\frac{1}{2}+\frac{N}{2p}}} \d s\right) \sup\limits_{s\in [0,t]} \left|\nabla \nabla \cdot \left(\u_{(n-1)}\otimes \u_{(n-1)}\right)\right|_{\frac{p}{2}} \notag\\
\ls & t^{\frac{1}{2}\left(1-\frac{N}{p}\right)}M_{1}M_{2} \sup\limits_{s\in [0,t]}  \left|\nabla \nabla \cdot \left(\u_{(n-1)}\otimes \u_{(n-1)}\right)\right|_{\frac{p}{2}} \notag\\
\ls & t^{\frac{1}{2}\left(1-\frac{N}{p}\right)}M \sup\limits_{s\in [0,t]}\left(2\left|\nabla \nabla\cdot  \u_{(n-1)}(s)\right|_{p}\left|\u_{(n-1)}(s)\right|_{p}+2\left|\u_{(n-1)}(s)\right|_{1,p}^{2}\right),\notag
\end{align}
where the condition $p>N$ and $M=\max\left\{M_{1},M_{2},M_{1}M_{2},2\right\} $.
%\begin{align}
%&\sup\limits_{s\in [0,t]}\left| \u_{(n-1)}(s)\right|_{p}\\
%\ls &\sup\limits_{s\in [0,t]}\left(\left| \v_{(n-1)}(s)\right|_{p}+\left| \z_{(n-1)}(s)\right|_{p}\right)
%%\ls & 2^{2r-1} \mathbb{E}\left[\sup\limits_{s\in (0,t)}\left(\left| \v_{(n-1)}(s)\right|_{p}^{2r}+\left| \z_{(n-1)}(s)\right|_{p}^{2r}\right)\right]\notag\\
%\ls 2M K_{0}(t)+K_{0}(t). \notag
%\end{align}
The pressure term satisfying \eqref{equation of nabla Q}, is estimated by
\begin{align}
    & \left| \int_{0}^{t} S(t-s) \nabla \cdot \nabla\nabla Q\left(a_{(n)},\u_{(n-1)}\right) \d s \right|_{p} \notag \\
\ls &  \left|  \int_{0}^{t}S(t-s)\left(-\nabla \cdot \nabla \nabla \triangle^{-1}\left( \nabla \cdot \left( \u_{(n-1)} \cdot \nabla \u_{(n-1)} \right)\d s -\left(\frac{\mu\nabla \rho_{(n)}}{\rho_{(n)}^{2}}\triangle \u_{(n-1)}\right)\right)\d s\right) \right|_{p}\notag \\
\ls &  \left|\int_{0}^{t} S(t-s)\nabla \cdot \nabla \nabla \triangle^{-1} \nabla \cdot \left( \u_{(n-1)} \cdot \nabla \u_{(n-1)} \right)\d s\right|_{p} \\
    &+\left|\int_{0}^{t}  S(t-s)\nabla \cdot \nabla  \nabla  \triangle^{-1} \nabla\cdot\left(\frac{\mu}{\rho_{(n)}}\triangle \v_{(n-1)}\right)\d s \right|_{p} , \notag
\end{align}
For $\int_{0}^{t} \left| S(t-s)\nabla \cdot \nabla \nabla \triangle^{-1} \nabla \cdot \left( \u_{(n-1)} \cdot \nabla \u_{(n-1)} \right)\right|_{p}\d s$, % since  $\left( \u_{(n-1)} \cdot \nabla \u_{(n-1)} \right)=\nabla \left|\u_{(n-1)}\right|^{2}$,
from \eqref{2 order esti for S(t-s)B(u u)} we have
\begin{align}
&\int_{0}^{t} \left| S(t-s)\nabla \cdot \nabla \nabla \triangle^{-1} \nabla \cdot \left( \u_{(n-1)} \cdot \nabla \u_{(n-1)} \right)\right|_{p}\d s \notag \\
%=&\int_{0}^{t}\left| S(t-s)\nabla \cdot \nabla \nabla \left|\u_{(n-1)}\right|^{2}\right|_{p}\d s
=& \left|\int_{0}^{t} S(t-s)\nabla \cdot \nabla B\left(\u_{(n-1)},\u_{(n-1)}\right)\d s\right|_{p}.% \\
%\ls & t^{\frac{1}{2}\left(1-\frac{N}{p}\right)}M \left(2M K_{0}\right)^{2}. \notag
\end{align}
For
\begin{align}
&\left|\int_{0}^{t} S(t-s) \nabla \cdot \nabla  \nabla  \triangle^{-1} \nabla\cdot\left(\frac{\mu}{\rho_{(n)}}\triangle \v_{(n-1)}\right)\d s \right|_{p}\\
=&\left|\int_{0}^{t}  S(t-s)\nabla \cdot \nabla \nabla  \triangle^{-\frac{1}{2}} \left(\frac{\mu}{\rho_{(n)}}\triangle \v_{(n-1)}\right)\d s \right|_{p},\notag
\end{align}
 due to $\frac{\mu}{\rho_{(n)}}\triangle \v_{(n-1)}=\nabla \cdot \left(\frac{\mu}{\rho_{(n)}}\nabla\v_{(n-1)}\right)-\nabla \frac{\mu}{\rho_{(n)}}\nabla\v_{(n-1)}$, there holds
\begin{align}\label{2-order nabla cdot (mu rho triangle u)}
&\left|\int_{0}^{t}  S(t-s) \nabla \cdot \nabla  \nabla  \triangle^{-1} \nabla\cdot\left(\frac{\mu}{\rho_{(n)}}\triangle \v_{(n-1)}\right)\d s \right|_{p} \notag \\
\ls & \int_{0}^{t} \left| S(t-s) \nabla \cdot \nabla  \nabla  \left(\frac{\mu}{\rho_{(n)}}\sum_{i}\sum_{j}\left(\nabla \v_{(n-1)}\right)_{i,j}\right)\right|_{p}\d s\\
 &+ \int_{0}^{t} \left| S(t-s) \nabla \cdot \nabla \nabla \triangle^{-\frac{1}{2}} \left(\nabla \frac{\mu}{\rho_{(n)}}\nabla\v_{(n-1)}\right)\right|_{p}\d s.\notag
\end{align}
The first term in \eqref{2-order nabla cdot (mu rho triangle u)} is estimated as
\begin{align}
& \int_{0}^{t} \left| S(t-s)\nabla \cdot \nabla \nabla \left(\frac{\mu}{\rho_{(n)}}\nabla\v_{(n-1)}\right)\right|_{p}\d s\notag\\
\ls & \int_{0}^{t} \frac{M_{1}}{(t-s)^{\frac{1}{2}}} \left|S(t-s)\nabla \cdot \nabla \left(\frac{\mu}{\rho_{(n)}}\sum_{i}\sum_{j}\left(\nabla \v_{(n-1)}\right)_{i,j}\right)\right|_{p}\d s \\
\ls & \int_{0}^{t} \frac{M}{(t-s)^{\frac{1}{2}+\frac{N}{2p}}} \left|\nabla \cdot \nabla \left(\frac{\mu}{\rho_{(n)}}\sum_{i}\sum_{j}\left(\nabla \v_{(n-1)}\right)_{i,j}\right)\right|_{\frac{p}{2}}\d s \notag\\
\ls & \frac{4\mu M}{\bar{\rho}} t^{\frac{1}{2}\left(1-\frac{N}{p}\right)}
\left(\sup\limits_{s\in[0,t]}\left|\nabla \cdot \nabla a_{(n)}\right|_{p}\sup\limits_{s\in[0,t]}\left|\nabla\v_{(n-1)}\right|_{p}
+\sup\limits_{s\in[0,t]}\left|\nabla a_{(n)}\right|_{p}\sup\limits_{s\in[0,t]}\left|\nabla\cdot \nabla \v_{(n-1)}\right|_{2,p}\right).\notag%\\
%\ls & t^{\frac{1}{2}\left(1-\frac{N}{p}\right)}\frac{8\mu M^{2}}{\bar{\rho}}K_{0}^{2}, \notag
\end{align}
%due to  $\odiv \v_{(n-1)} =0$. % $\left|a_{(n)}\right|_{2,p}\ls \left|a_{(n)}(0)\right|_{2,p}+t\left|\v_{(n-1)}\right|_{2,p}$ and}
%\begin{align}
% \int_{0}^{t} \left\| S(t-s)\nabla \frac{\mu}{\rho_{(n)}}\nabla\v_{(n-1)}\right\|_{p}\d s =  \int_{0}^{t} \left\| S(t-s)\triangle^{\frac{1}{2}}\triangle^{-\frac{1}{2}}\left( \nabla \frac{\mu}{\rho_{(n)}}\nabla\v_{(n-1)}\right)\right\|_{p}\d s,
%\end{align}
To estimate the second term in \eqref{2-order nabla cdot (mu rho triangle u)},
 we have
\begin{align}
  & \int_{0}^{t} \left| S(t-s)\nabla \cdot \nabla\nabla \triangle^{-\frac{1}{2}}\left(\nabla \frac{\mu}{\rho_{(n)}}\nabla\v_{(n-1)}\right)\right|_{p}\d s  \notag\\
\ls & \int_{0}^{t} \frac{M_{1}}{(t-s)^{\frac{1}{2}}} \left| S(t-s) \nabla \cdot \nabla\triangle^{-\frac{1}{2}}\left(\nabla \frac{\mu}{\rho_{(n)}}\nabla\v_{(n-1)}\right)\right|_{p}\d s  \\
\ls &  \int_{0}^{t} \frac{M}{(t-s)^{\frac{1}{2}+\frac{N}{2p}}}\left|\nabla \cdot \nabla\triangle^{-\frac{1}{2}}\left(\nabla \frac{\mu}{\rho_{(n)}}\nabla\v_{(n-1)}\right)\right|_{\frac{p}{2}}\d s \notag \\
\ls & \frac{4\mu M}{\bar{\rho}}t^{\frac{1}{2}\left(1-\frac{N}{p}\right)}
\left(\sup\limits_{s\in[0,t]}\left|\nabla\cdot\nabla a_{(n)}\right|_{p}\sup\limits_{s\in[0,t]}\left|\nabla \v_{(n-1)}\right|_{p}
+\sup\limits_{s\in[0,t]}\left|\nabla a_{(n)}\right|_{p}\sup\limits_{s\in[0,t]}\left|\nabla\cdot\nabla \v_{(n-1)}\right|_{p}\right). \notag %\\
% \ls &t^{\frac{1}{2}\left(1-\frac{N}{p}\right)}\frac{8\mu M^{2}}{\bar{\rho}}K_{0}^{2}. \notag
\end{align}
In summary, it holds that
\begin{align}
&\left|\left(\nabla\cdot \nabla \v_{(n)}\right)(t)\right|_{ p}\notag\\
\ls & M_{1}\left|\nabla\cdot \nabla \u_{0}\right|_{p}+  t^{\frac{1}{2}\left(1-\frac{N}{p}\right)}M \sup\limits_{s\in [0,t]}\left(2\left|\nabla \nabla \cdot  \u_{(n-1)}(s)\right|_{p}\left|\u_{(n-1)}(s)\right|_{p}+2\left|\nabla \u_{(n-1)}(s)\right|_{p}^{2}\right) \\
& +\frac{8\mu M}{\bar{\rho}} t^{\frac{1}{2}\left(1-\frac{N}{p}\right)}
\left(\sup\limits_{s\in[0,t]}\left|\nabla \cdot \nabla a_{(n)}\right|_{p}\sup\limits_{s\in[0,t]}\left|\nabla \v_{(n-1)}\right|_{p}
+\sup\limits_{s\in[0,t]}\left|\nabla a_{(n)}\right|_{p}\sup\limits_{s\in[0,t]}\left|\nabla \cdot \nabla \v_{(n-1)}\right|_{p}\right). \notag%\\
%\ls &M K_{0}\left(1+t^{\frac{1}{2}\left(1-\frac{N}{p}\right)}\left(8M^{3} K_{0}^{2} + \frac{12\mu M^{2}}{\bar{\rho}}K_{0}^{2}\right)\right). \notag
 \end{align}

We combine the above estimates, and we have
\begin{align}
&\left|\left( \v_{(n)}\right)(t)\right|_{2,p}\notag\\
\ls & M_{1}\left|\u_{0}\right|_{2,p} + t^{\frac{1}{2}\left(1-\frac{N}{p}\right)}2M \left(\sup\limits_{s\in [0,t]}\left| \u_{(j-1)}(s)\right|_{p}\right)^{2}\notag \\
&+ t^{\frac{1}{2}\left(1-\frac{N}{p}\right)} 2 M \left(2\sup\limits_{s\in [0,t]}\left|\nabla \u_{(n-1)}(s)\right|_{p}^{2} + 2\sup\limits_{s\in [0,t]} \left|\nabla \cdot \u_{(n-1)}(s)\right|_{p}\sup\limits_{s\in [0,t]} \left|\u_{(n-1)}(s)\right|_{p}\right) \notag \\
&+  t^{\frac{1}{2}\left(1-\frac{N}{p}\right)}2M \sup\limits_{s\in [0,t]}\left(2\left|\nabla\nabla\cdot \u_{(n-1)}(s)\right|_{p}\left|\u_{(n-1)}(s)\right|_{p}+2\left|\nabla \u_{(n-1)}(s)\right|_{p}^{2}\right) \notag\\
&+t^{\frac{1}{2}\left(1-\frac{N}{p}\right)}\frac{4\mu M}{\bar{\rho}}\sup\limits_{s \in [0,t]}\left| a_{(n)}\right|_{p} \left| \nabla \v_{(n-1)}\right|_{p}
+t^{\frac{1}{2}\left(1-\frac{N}{p}\right)}\frac{8\mu M}{\bar{\rho}}\sup\limits_{s \in [0,t]}\left|\nabla a_{(n)}\right|_{p} \left| \v_{(n-1)}\right|_{p} \\ &+t^{\frac{1}{2}\left(1-\frac{N}{p}\right)}\frac{4\mu M}{\bar{\rho}}\sup\limits_{s \in [0,t]}\left| a_{(n)}\right|_{p} \left| \nabla \cdot \nabla \v_{(n-1)}\right|_{p} \notag\\
&+t^{\frac{1}{2}\left(1-\frac{N}{p}\right)}\frac{12\mu M}{\bar{\rho}}\sup\limits_{s \in [0,t]}\left|\nabla a_{(n)}\right|_{p} \left|\nabla \v_{(n-1)}\right|_{p}
+t^{\frac{1}{2}\left(1-\frac{N}{p}\right)}\frac{8\mu M}{\bar{\rho}}\sup\limits_{s \in [0,t]}\left|\nabla \cdot \nabla a_{(n)}\right|_{p} \left| \v_{(n-1)}\right|_{p}\notag\\
&+ t^{\frac{1}{2}\left(1-\frac{N}{p}\right)} \frac{16\mu M}{\bar{\rho}}
\left(\sup\limits_{s\in[0,t]}\left|\nabla\cdot \nabla a_{(n)}\right|_{p}\sup\limits_{s\in[0,t]}\left|\nabla \v_{(n-1)}\right|_{p}
+\sup\limits_{s\in[0,t]}\left|\nabla a_{(n)}\right|_{p}\sup\limits_{s\in[0,t]}\left|\nabla \cdot \nabla \v_{(n-1)}\right|_{p}\right)\notag\\
\ls &M_{1}\left|\u_{0}\right|_{2,p}+  t^{\frac{1}{2}\left(1-\frac{N}{p}\right)}2M \left(4\sup\limits_{s\in [0,t]}\left| \u_{(j-1)}(s)\right|_{2,p}\right)^{2} \notag\\
% &+t^{\frac{1}{2}\left(1-\frac{N}{p}\right)}\frac{4\mu M}{\bar{\rho}}\sup\limits_{s \in [0,t]}\left| a_{(n)}\right|_{p} \left|\v_{(n-1)}\right|_{2,p}+ +t^{\frac{1}{2}\left(1-\frac{N}{p}\right)}\frac{8\mu M}{\bar{\rho}}\sup\limits_{s \in [0,t]}\left| a_{(n)}\right|_{2,p} \left| \v_{(n-1)}\right|_{p}\notag\\
&+ t^{\frac{1}{2}\left(1-\frac{N}{p}\right)} \frac{16\mu M}{\bar{\rho}}
\left(\sup\limits_{s\in[0,t]}\left| a_{(n)}\right|_{2, p}\sup\limits_{s\in[0,t]}\left| \v_{(n-1)}\right|_{1,p}
+\sup\limits_{s\in[0,t]}\left| a_{(n)}\right|_{1,p}\sup\limits_{s\in[0,t]}\left| \v_{(n-1)}\right|_{2,p}\right). \notag %\\
%\ls &M K_{0}\left(1+t^{\frac{1}{2}\left(1-\frac{N}{p}\right)}\left(8M^{3} K_{0}^{2} + \frac{12\mu M^{2}}{\bar{\rho}}K_{0}^{2}\right)\right). \notag
 \end{align}
%\begin{align}
%&\mathbb{E}\left[\left|S(t) \left(\v_{(j)}\right)(t)\right|_{p}^{r} \right]\notag\\
%\ls &  \left(\frac{3}{2}M K_{0}(t) + M \left(2M K_{0}(t)+K_{0}(t)\right)^{2} t^{\frac{1}{2}\left(1-\frac{N}{p}\right)}+ \frac{16\mu M^{2}}{\bar{\rho}}K_{0}(t)^{2} t^{\frac{1}{2}\left(1-\frac{N}{p}\right)}\right)^{r} \\
%\ls & \left( M K_{0}(t)\left(\frac{3}{2}+t^{\frac{1}{2}\left(1-\frac{N}{p}\right)}\left(\left(2M+1\right)^{2} K_{0}(t)+\frac{16\mu M}{\bar{\rho}}K_{0}(t)\right)\right)\right)^{r}.\notag
% \end{align}
For $N<p\ls 6$, by Sobolev's embedding, there holds
\begin{align}
\left|\z_{(n-1)}(s)\right|_{2,p} \ls \left\|\z_{(n-1)}(s)\right\|_{3}.
 \end{align}
Hence, by \eqref{onto map of a sto}, we have
\begin{align}
&\left|\left(\v_{(n)}\right)(t)\right|_{2, p}\notag\\
\ls &  M_{1}K_{0}+  t^{\frac{1}{2}\left(1-\frac{N}{p}\right)} 8 M \sup\limits_{s\in [0,t]}\left(\left| \v_{(n-1)}(s)\right|_{2,p}+\left\|\z_{(n-1)}(s)\right\|_{3}\right)^{2} \\
& +\frac{16\mu M}{\bar{\rho}} t^{\frac{1}{2}\left(1-\frac{N}{p}\right)}
\sup\limits_{s\in[0,t]}\left(\left| a_{0}(x)\right|_{2,p} + t \left(\left|\v_{(n-1)}(s,x)\right|_{2,p} + \left\|\z_{(n-1)}\left(s,x\right)\right\|_{3}\right)\right)\left|\v_{(n-1)}\right|_{1,p}\notag\\
&+\frac{16\mu M}{\bar{\rho}} t^{\frac{1}{2}\left(1-\frac{N}{p}\right)}\sup\limits_{s\in[0,t]}\left(\left| a_{0}(x)\right|_{1,p} + t \left(\left|\v_{(n-1)}(s,x)\right|_{1,p} +\left\|\z_{(n-1)}\left(s,x\right)\right\|_{3}\right)\right)\left|\v_{(n-1)}\right|_{2,p}.\notag%\\
%\ls &M K_{0}\left(1+t^{\frac{1}{2}\left(1-\frac{N}{p}\right)}\left(8M^{3} K_{0}^{2} + \frac{12\mu M^{2}}{\bar{\rho}}K_{0}^{2}\right)\right). \notag
\end{align}
 By Theorem \ref{Lp for z}, if $\left| \v_{(j-1)}(s)\right|_{2,p}\ls 2M K_{0}$, then it holds that
\begin{align}
&\left(\left| \v_{(n-1)}(s)\right|_{2,p}+\left\|\z_{(n-1)}(s)\right\|_{3}\right)
%\ls \sup\limits_{s\in [0,t]}\left(\left| \v_{(j-1)}(s)\right|_{p}+\left| \z_{(j-1)}(s)\right|_{p}\right)
%\ls & 2^{2r-1} \mathbb{E}\left[\sup\limits_{s\in (0,t)}\left(\left| \v_{(j-1)}(s)\right|_{p}^{2r}+\left| \z_{(j-1)}(s)\right|_{p}^{2r}\right)\right]\notag\\
\ls 2M K_{0}+K_{0}t^{\frac{1}{2}}.
\end{align}
Therefore, % if
%\begin{equation}
%\left|\v_{(n-1)}\right|_{2, p}\ls 2M K_{0}, %\quad \bar{\mathbb{P}}\quad  {\rm a.s.},
%\end{equation}
 it holds that
\begin{align}\label{time continuity of v}
&\left|\left(\v_{(n)}\right)(t)\right|_{2, p}\notag\\
\ls & M_{1}\left|\u_{0}\right|_{2,p} + t^{\frac{1}{2}\left(1-\frac{N}{p}\right)}8 M \sup\limits_{s\in [0,t]}\left(\left| \v_{(n-1)}(s)\right|_{2,p}+K_{0}t^{\frac{1}{2}}\right)^{2}\notag\\
% &+t^{\frac{1}{2}\left(1-\frac{N}{p}\right)}4M\left(\sup\limits_{s\in [0,t]}\left|\v_{(n-1)}(s)\right|_{1,p}+C_{6}t^{\frac{1}{2}}\right)^{2}\\
& +\frac{ 16 \mu M}{\bar{\rho}} t^{\frac{1}{2}\left(1-\frac{N}{p}\right)}
\sup\limits_{s\in[0,t]}\left(\left| a_{0}(x) \right|_{2,p} + t \left(\left|\v_{(n-1)}(s,x)\right|_{2,p} +K_{0}t^{\frac{1}{2}}\right)\right)\left|\v_{(n-1)}\right|_{1,p} \\
&+\frac{ 16 \mu M}{\bar{\rho}} t^{\frac{1}{2}\left(1-\frac{N}{p}\right)} \sup\limits_{s\in[0,t]} \left(\left| a_{0}(x)\right|_{1,p} + t \left(\left|\v_{(n-1)}(s,x)\right|_{1,p} +K_{0}t^{\frac{1}{2}}\right)\right)\left|\v_{(n-1)}\right|_{2,p} \notag \\
\ls &M K_{0}\left(1+t^{\frac{1}{2}\left(1-\frac{N}{p}\right)}\left(32M^{2} K_{0}+\frac{64\mu M K_{0}}{\bar{\rho}}+8t + \frac{128\mu M K_{0}}{\bar{\rho}}t
+32K_{0}t^{\frac{1}{2}} + \frac{64\mu M}{\bar{\rho}}t^{\frac{3}{2}}\right)\right). \notag
\end{align}
Due to the fact that $K_{0}(t)$ is increasing with respect to $t$, there exists $T_{1}$, $T_{1}\ls T_{0}$ such that
 \begin{align}
 T_{1}^{\frac{1}{2}\left(1-\frac{N}{p}\right)}\left(32M^{2} K_{0}+\frac{64\mu M K_{0}}{\bar{\rho}}+8 T_{1} + \frac{128\mu M K_{0}}{\bar{\rho}} T_{1}
+32K_{0}T_{1}^{\frac{1}{2}} + \frac{64\mu M}{\bar{\rho}}T_{1}^{\frac{3}{2}}\right)\ls 1,
% T_{1}^{\frac{1}{2}\left(1-\frac{N}{p}\right)}\left(\left(2M+1\right)^{2} K_{0}(T_{1})+\frac{16\mu M}{\bar{\rho}}K_{0}(T_{1})\right) \ls \frac{1}{2}, \\
%\text{ i.e., }T_{1}\ls \left(\left(\left(2M+1\right)^{2} K_{0}(T_{1})+\left(\frac{1}{2}+9K_{0}(T_{1})\right)\right)\right)^{-\frac{2p}{p-N}},
 \end{align}
then we know that
\begin{equation}
\left|\v_{(n)}\right|_{p}\ls 2M K_{0}. %\quad \bar{\mathbb{P}}\quad  {\rm a.s.},
\end{equation}
\eqref{time continuity of v} also implies the continuity of solutions in time.
\hfill $\square$
\end{proof}
%Reviewing the proof in this subsection, we can also draw a conclusion that: \\
% if $\v_{(n-1)} \in L^{\infty}\left(0, T;  W^{2,p}\left(\mathbb{T}^{N}\right)\right)$, then so does $\v_{(n)}$.

%\begin{figure}[h]%h require the picture in current position
%  \centering

%  % Requires \usepackage{graphicx}
%  \includegraphics[height=12cm,width=16.5cm]{iteration.eps}\\
%  \caption{iteration procedure}\label{iteration procedure}
%\end{figure}

%TeXCAD Picture [Iteration procedure.pic]. Options:
%\grade{\on}
%\emlines{\off}
%\epic{\off}
%\beziermacro{\on}
%\reduce{\on}
%\snapping{\off}
%\quality{8.00}
%\graddiff{0.01}
%\snapasp{1}
%\zoom{4.0000}
In summary, the iteration procedure is also illustrated in the following figure.\\
\unitlength 0.5mm % = 2.85pt
\linethickness{0.4pt}
\ifx\plotpoint\undefined\newsavebox{\plotpoint}\fi % GNUPLOT compatibility
\begin{picture}(315.25,171)(61.25,148.5)\label{iteration procedure}
\tiny{%\put(93.75,319.5){$s=\left[\frac{N}{2}\right]+2,$}
\put(68,299){\framebox(60,9.25)[cc]
{$\mathbf{v}_{(0)}=\mathbf{u}_{(0)}\in W^{2,p}(\mathbb{T}^N)$}}
\put(83,285){\framebox(30,13.25)[cc]
{$ \mathbf{z}_{(0)}=0$}}
\put(131,305){\vector(1,0){38}}
\put(130,310.25){$a_{(0)}\in W^{2,p}(\mathbb{T}^N)$}
\put(173,300.5){\framebox(100,9.25)[cc]
{$a_{(1)}\in C\left([0, T]; W^{2,p}(\mathbb{T}^N)\cap C^{1,\ell}(\mathbb{T}^N)\right)$}}
\put(171.25,275.5){\framebox(77.75,9.25)[cc]
{$\mathbf{z}_{(1)}\in C\left([0, T]; W^{2,p}(\mathbb{T}^N)\right)$}}
\put(270,289){\framebox(90,9.25)[cc]
{$\mathbf{v}_{(1)}\in C\left([0, T]; W^{2,p}(\mathbb{T}^N)\right)$}}
\put(208,298.25){\vector(0,-1){11.5}}
\put(265,261.75){\framebox(105,9.25)[cc]
{$a_{(2)}\in C\left([0, T]; W^{2,p}(\mathbb{T}^N)\cap C^{1,\ell}(\mathbb{T}^{N})\right)$}}
\put(322,288){\vector(0,-1){14.75}}
\put(313,274.25){\vector(3,-1){.07}}\multiput(239.25,301)(.093001261,-.0337326608){793}{\line(1,0){.093001261}}
\put(264.25,288.75){\vector(2,1){.07}}\multiput(250,282)(.07089552,.03358209){201}{\line(1,0){.07089552}}
\put(322,288){\vector(0,-1){14.75}}
\put(313,274.25){\vector(3,-1){.07}}\multiput(239.25,301)(.093001261,-.0337326608){793}{\line(1,0){.093001261}}
\put(299,273.75){\vector(1,0){.07}}\multiput(252,277.25)(.45192308,-.03365385){104}{\line(1,0){.45192308}}
\put(264.25,288.75){\vector(2,1){.07}}\multiput(250,282)(.07089552,.03358209){201}{\line(1,0){.07089552}}
\put(285,302){\vector(4,-1){.07}}\multiput(273,304.25)(.1923077,-.0336538){52}{\line(1,0){.1923077}}
\put(89.75,249.25){$\cdots$,}
\put(60,231.25){\framebox(78,8.25)[cc]
{$\mathbf{z}_{(n-1)}\in C\left([0, T];W^{2,p}(\mathbb{T}^N)\right)$}}
\put(60,213.5){\framebox(78,8.25)[cc]
{$\mathbf{v}_{(n-1)}\in C\left([0, T];W^{2,p}(\mathbb{T}^N)\right)$}}
\put(267.5,231.75){\framebox(101,9.25)[cc]
{$a_{(n)}\in C\left([0, T];W^{2,p}(\mathbb{T}^N)\cap C^{1,\ell}(\mathbb{T}^N)\right)$}}
\put(289,207.25){\framebox(80,9.25)[cc]
{$\mathbf{z}_{(n)}\in C\left([0, T];W^{2,p}(\mathbb{T}^N)\right)$}}
\put(183.25,182.75){\framebox(90,9.25)[cc]
{$\mathbf{v}_{(n)}\in C\left([0, T]; W^{2,p} (\mathbb{T}^N)\right)$}}
\multiput(136.5,236)(.65979381,-.03350515){97}{\line(1,0){.65979381}}
\multiput(160,232.75)(-.046134663,-.033665835){401}{\line(-1,0){.046134663}}
\put(267.75,233){\vector(1,0){.07}}\multiput(200.75,232.5)(5.8,.033333){11}{\line(1,0){5.8}}
\put(324.25,230){\vector(0,-1){11.25}}
\put(259,194.75){\vector(-2,-1){.07}}\multiput(287.25,210.5)(-.060492505,-.03372591){467}{\line(-1,0){.060492505}}
\put(254.25,195.75){\vector(-1,-1){.07}}\multiput(288.5,229.75)(-.0339781746,-.0337301587){1008}{\line(-1,0){.0339781746}}
\put(155,238.5){$a_{(n-1)}\in C\left([0, T];W^{2,p}(\mathbb{T}^N)\cap C^{1,\ell}(\mathbb{T}^N)\right)$}}
\large{\put(160,170){Figure 1. Iteration procedure. }},
\end{picture}

\subsection{The contractions}

We use Banach's fixed point theorem to get the fixed point $\z, \v, a$.

First, we estimate $\left| a_{(n)}-a_{(n-1)}\right|_{p} $.
The linearized mass conservation equation is
\begin{equation}
 \partial_{t} a_{(n)}+\u_{(n-1)}\cdot \nabla a_{(n)}= 0.
\end{equation}
Hence, there holds the new transport equation,
\begin{align}\label{the new transport equation}
 &\partial_{t}\left( a_{(n)}-a_{(n-1)}\right)+\u_{(n-1)}\cdot \nabla\left( a_{(n)}-a_{(n-1)}\right)
 =-\left(\u_{(n-1)}-\u_{(n-2)}\right)\cdot \nabla a_{(n-1)}.
\end{align}
By the property of transport equations, denoting $\overline{\mathrm{D}}$ as material derivative, it holds that
\begin{align}
\overline{\mathrm{D} }\left( a_{(n)}-a_{(n-1)}\right) =-\left(\u_{(n-1)}-\u_{(n-2)}\right)\cdot \nabla a_{(n-1)}.
\end{align}
Let $\mathcal{C}$ be the characteristic line. By Minkowski's integral inequality and \eqref{the new transport equation}, we have
\begin{align}
&\left| a_{(n)}-a_{(n-1)}(s,x)\right|_{p}\notag\\
\ls & \left| \left(a_{(n)}-a_{(n-1)}\right)\left(0,x-\int_{0}^{s}\u_{(n-1)}(s')\d s'\right)\right|_{p}
 +\left| \int_{\mathcal{C}}\left(\u_{(n-1)}-\u_{(n-2)}\right)\cdot \nabla a_{(n-1)} \d \mathcal{C}\right|_{p} \notag \\
\ls & 0+ \int_{0}^{s}\left|\left(\u_{(n-1)}-\u_{(n-2)}\right)\cdot \nabla a_{(n-1)}\right|_{p}\left|\u_{(n-1)}\right|_{\infty} \d s' \\
\ls & \int_{0}^{s}\left|\u_{(n-1)}-\u_{(n-2)}\right|_{p} \left| \nabla a_{(n-1)}\right|_{\infty}\left|\u_{(n-1)}\right|_{\infty} \d s' \notag \\
\ls &  2MK_{0}^{2} s \sup\limits_{s' \in [0,s]}\left|\left(\u_{(n-1)}-\u_{(n-2)}\right)(s',x)\right|_{p}. \notag
\end{align}
Hence, we obtain
\begin{align}\label{contraction of a}
&\sup\limits_{s \in [0,t]}\left| a_{(n)}-a_{(n-1)}(s,x)\right|_{p}
\ls   2MK_{0}^{2}t\sup\limits_{s \in [0,t]}\left|\left(\u_{(n-1)}-\u_{(n-2)}\right)(s,x)\right|_{p}.
\end{align}

Now we prove the contraction of $\z_{(n)}$ in $C\left([0, T];  L^{p}\left(\mathbb{T}^{N}\right)\right)$ with the regularity of $\z$ established. We estimate $\left|\z_{(n)}-\z_{(n-1)}\right|_{p}$. We recall the iteration equations:
\begin{equation}
\d \z_{(n)}(t)- \frac{\mu}{\bar{\rho}\left(1+a_{(n)}\right)}\triangle\z_{(n)}(t)\d t=\Phi \d W(t),\\
\end{equation}
and
\begin{equation}
\d \z_{(n-1)}(t)-\frac{\mu}{\bar{\rho}\left(1+a_{(n-1)}\right)}\triangle\z_{(n-1)}(t)\d t=\Phi \d W(t).\\
\end{equation}
The subtraction of the second equation from the first yields the following:
\begin{equation}\label{equation of z for contraction}
\d \left( \z_{(n)}(t) -\z_{(n-1)}(t)\right)-\frac{\mu}{\bar{\rho}\left(1+a_{(n)}\right)}\triangle\z_{(n)}(t)\d t
+\frac{\mu}{\bar{\rho}\left(1+a_{(n-1)}\right)}\triangle\z_{(n-1)}(t)\d t=0.
\end{equation}
We rewrite the contraction equation as follows:
\begin{align}
& \d \left( \z_{(n)}(t) -\z_{(n-1)}(t)\right)-\frac{\mu}{\bar{\rho}\left(1+a_{(n)}\right)}\triangle\left(\z_{(n)}(t) - \z_{(n-1)}(t)\right)\d t\notag\\
&-\left(\frac{\mu}{\bar{\rho}\left(1+a_{(n)}\right)}-\frac{\mu}{\bar{\rho}\left(1+a_{(n-1)}\right)}\right)\triangle \z_{(n-1)}(t)\d t =0.
\end{align}
Referring to \eqref{equation of z for contraction}, by Minkowski's integral inequality, we have
\begin{align}
   &\left| \z_{(n)}(t) -\z_{(n-1)}(t)\right|_{p}\notag\\
\ls &\left| \int_{0}^{t} S_{(n)}(t-s)\frac{\mu\left(a_{(n)}-a_{(n-1)}\right)}{\bar{\rho}\left(1+a_{(n)}\right)\left(1+a_{(n-1)}\right)}\triangle\z_{(n-1)}(s)\d s \right|_{p}\notag \\
\ls & \int_{0}^{t}\left| S_{(n)}(t-s) \frac{\mu\left(a_{(n)}-a_{(n-1)}\right)}{\bar{\rho}\left(1+a_{(n)}\right)\left(1+a_{(n-1)}\right)}\triangle\z_{(n-1)}(s)\right|_{p}\d s \notag \\
\ls & \int_{0}^{t}\left| S_{(n)}(t-s) \frac{4\mu\left(a_{(n)}-a_{(n-1)}\right)}{\bar{\rho}}\triangle\z_{(n-1)}(s)\right|_{p}\d s  \\
= & \frac{4\mu}{\bar{\rho}}\int_{0}^{t}\left| S_{(n)}(t-s) \left(\nabla\cdot\left(\left(a_{(n)}-a_{(n-1)}\right) \nabla \z_{(n-1)}(s)\right)- \nabla\left(a_{(n)}-a_{(n-1)}\right)\nabla \z_{(n-1)}(t) \right)\right|_{p}\d s . \notag
\end{align}
For convenience, we denote
\begin{align}
\delta_{a,\z}=\left(\nabla\cdot\left(\left(a_{(n)}-a_{(n-1)}\right) \nabla \z_{(n-1)}(t)\right)- \nabla\left(a_{(n)}-a_{(n-1)}\right)\cdot \nabla \z_{(n-1)}(s) \right).
\end{align}
Then, we estimate
\begin{align}
   &\left| \z_{(n)}(t) -\z_{(n-1)}(t)\right|_{p}\notag\\
   = & \frac{4\mu}{\bar{\rho}}\int_{0}^{t}\left| S_{(n)}(t-s) \nabla\cdot \triangle^{-\frac{1}{2}} \delta_{a,\z}\right|_{p}\d s
\ls  \frac{4\mu}{\bar{\rho}}\int_{0}^{t}\frac{1}{(t-s)^{\frac{1}{2}}}\left| S_{(n)}(t-s)  \triangle^{-\frac{1}{2}} \delta_{a,\z}\right|_{p}\d s \notag\\
\ls &\frac{4\mu}{\bar{\rho}}\int_{0}^{t}\frac{M_{1}}{(t-s)^{\frac{1}{2}}}\left| S_{(n)}(t-s)\left(a_{(n)}-a_{(n-1)}\right) \nabla \z_{(n-1)}(s)\right|_{p}\d s \\
     &+ \frac{4\mu}{\bar{\rho}}\int_{0}^{t}\frac{M_{1}}{(t-s)^{\frac{1}{2}}}\left| S_{(n)}(t-s) \triangle^{-\frac{1}{2}} \left(\nabla\left(a_{(n)}-a_{(n-1)}\right)(s)\cdot \nabla \z_{(n-1)}\right)(s) \right|_{p}\d s. \notag
\end{align}
The first term in right-hand-side is estimated by
\begin{align}
 &\int_{0}^{t}\frac{1}{(t-s)^{\frac{1}{2}}}\left| S_{(n)}(t-s)\left(a_{(n)}-a_{(n-1)}\right) (s)\nabla \z_{(n-1)}(s)\right|_{p}\d s \notag\\
\ls &\int_{0}^{t}\frac{1}{(t-s)^{\frac{1}{2}+\frac{N}{2p}}}\left|\left(a_{(n)}-a_{(n-1)}\right)(s) \nabla \z_{(n-1)}(s)\right|_{\frac{p}{2}}\d s \\
\ls &t^{\frac{1}{2}\left(1-\frac{N}{p}\right)}\sup_{s \in [0,t]}\left|\nabla \z_{(n-1)}(t)\right|_{p}\sup_{s \in [0,t]}\left|\left(a_{(n)}-a_{(n-1)}\right)(s,x)\right|_{p} \notag \\
\ls & t^{\frac{1}{2}\left(1-\frac{N}{p}\right)} K_{0}\sup_{s \in [0,t]} \left|\left(a_{(n)}-a_{(n-1)}\right)\left(s, x \right)\right|_{p}. \notag
\end{align}
The second term holds
\begin{align}
&\left|\triangle^{-\frac{1}{2}} \left(\nabla\left(a_{(n)}-a_{(n-1)}\right)\cdot \nabla \z_{(n-1)}(s)\right)\right|\\
 \ls &4\sup\limits_{s \in [0,t]}\left|\left(a_{(n)}-a_{(n-1)}\right)\left(s, x \right)\right| \left|\nabla \z_{(n-1)}\left(s, x \right)\right|. \notag
\end{align}
%\textcolor[rgb]{1.00,0.00,0.00}{Denote
%$$f_{1}=\triangle^{-\frac{1}{2}} \left(\nabla\left(a_{(n)}-a_{(n-1)}\right)(t,x)\nabla \z_{(n-1)}(t,x)\right),$$ $$f_{2}=\nabla\left(a_{(n)}-a_{(n-1)}\right)\left(t,x\right)\z_{(n-1)}(t,x) ,$$
% and extend the domain $\mathbb{T}^{N}$ periodically to the whole space $\mathds{R}^{d}$. We do Fourier transform on $f_{1}$ and $f_{2}$ respectively.
%\begin{align}
%\mathcal{F}\left(f_{1}\right)= \frac{1}{iy}\left(iy\mathcal{F}\left(a_{(n)}-a_{(n-1)}\right)(t,y)\right)\ast \left(iy\mathcal{F}\left(\z_{(n-1)}\right)(t,y)\right),
%\end{align}
%\begin{align}
%\mathcal{F}\left(f_{2}\right)=\left(iy\mathcal{F}\left(a_{(n)}-a_{(n-1)}\right)(t,y)\right)\ast \mathcal{F}\left(\z_{(n-1)}\right)(t,y)
%\end{align}}
Hence, we have
\begin{align}
&\int_{0}^{t}\frac{M_{1}}{(t-s)^{\frac{1}{2}}}\left| S_{(n)}(t-s) \triangle^{-\frac{1}{2}} \left(\nabla\left(a_{(n)}-a_{(n-1)}\right)\cdot\nabla \z_{(n-1)}(s)\right) \right|_{p}\d s\notag\\
\ls &\int_{0}^{t}\frac{M}{(t-s)^{\frac{1}{2}+\frac{N}{2p}}}\left| \triangle^{-\frac{1}{2}} \left(\nabla\left(a_{(n)}-a_{(n-1)}\right)\cdot\nabla \z_{(n-1)}(s)\right) \right|_{\frac{p}{2}}\d s \notag\\
\ls &\int_{0}^{t}\frac{4M}{(t-s)^{\frac{1}{2}+\frac{N}{2p}}}\left(\int_{\mathbb{T}^{N}}\left(\sup\limits_{s \in [0,t]}\left|\left(a_{(n)}-a_{(n-1)}\right)\left(s,x \right)\right| \left|\nabla \z_{(n-1)}(s)\right|\right)^{\frac{p}{2}}\d x\right)^{\frac{2}{p}} \d s \\
\ls &\int_{0}^{t}\frac{4M}{(t-s)^{\frac{1}{2}+\frac{N}{2p}}} \left|\sup\limits_{s \in [0,t]} \left(a_{(n)}-a_{(n-1)}\right)\left(s,x \right)\right|_{p}\sup\limits_{s \in [0,t]}\left| \z_{(n-1)}(s,x)\right|_{1,p}\d s \notag\\
\ls & t^{\frac{1}{2}\left(1-\frac{N}{p}\right)} 4 M K_{0}\sup\limits_{s \in [0,t]} \left|\left(a_{(n)}-a_{(n-1)}\right)\left(s, x\right)\right|_{p}.\notag
\end{align}
Therefore, we conclude that
\begin{align}
\left| \z_{(n)}(t) -\z_{(n-1)}(t)\right|_{p}
\ls t^{\frac{1}{2}\left(1-\frac{N}{p}\right)}\frac{20\mu}{\bar{\rho}}M K_{0} \sup\limits_{s \in [0,t]}\left|\left(a_{(n)}-a_{(n-1)}\right)\left(x \right)(s,x)\right|_{p}.
\end{align}
 From \eqref{contraction of a} and the estimates of $\left|a_{(n)}-a_{(n-1)}\right|_{p}$,
 it holds that
\begin{align}\label{beta involved contraction of z}
&\left| \z_{(n)}(t) -\z_{(n-1)}(t)\right|_{p}\\
\ls &t^{1+\frac{1}{2}\left(1-\frac{N}{p}\right)}\frac{40\mu}{\bar{\rho}}M^{2}K_{0}^{3}\sup\limits_{s \in [0,t]}\left(\left|\left(\v_{(n-1)}-\v_{(n-2)}\right)(s,x)\right|_{p}+\left|\left(\z_{(n-1)}-\z_{(n-2)}\right)(s,x)\right|_{p}\right). \notag
\end{align}
Subsequently, we need to estimate $\left|\v_{(n-1)}-\v_{(n-2)}\right|_{p}$. Noting that
\begin{align}
& \partial_{t}\v_{(n)}(t) - \frac{\mu}{\bar{\rho}\left(1+a_{(n)}\right)}\triangle\v_{(n)}(t)\\
=&  -\left(\v_{(n-1)}(t)+\z_{(n-1)}(t)\right)\cdot\nabla\left(\v_{(n-1)}(t)+\z_{(n-1)}(t)\right)- \nabla Q\left(a_{(n)},\u_{(n-1)}\right),\notag
\end{align}
we have
\begin{align}\label{equation for bar w -w}
&\left(\v_{(n)}-\v_{(n-1)}\right)_{t}- \frac{\mu}{\bar{\rho}\left(1+a_{(n)}\right)}\triangle\left(\v_{(n)}-\v_{(n-1)}\right)\notag\\
=& B\left(\v_{(n-1)}+\z_{(n-1)}\right)-B\left(\v_{(n-2)}+\z_{(n-2)}\right) \\
&-\left(\nabla Q\left(a_{(n)},\u_{(n-1)}\right)-\nabla Q\left(a_{(n-1)},\u_{(n-2)}\right)\right) + \left(\frac{\mu}{\bar{\rho}\left(1+a_{(n)}\right)}-\frac{\mu}{\bar{\rho}\left(1+a_{(n-1)}\right)}\right) \triangle \v_{(n-1)}.\notag
\end{align}
The mild solutions of \eqref{equation for bar w -w} are written in the form of
\begin{align}\label{equation for bar v -v}
&\left(\v_{(n)}-\v_{(n-1)}\right) \notag\\
=&S(t-s) \left(\left(\v_{(n)}-\v_{(n-1)}\right)(0)\right)\notag\\
&+\int_{0}^{t} S(t-s) \left(B\left(\v_{(n-1)}+\z_{(n-1)}\right)(s)-B\left(\v_{(n-2)}+\z_{(n-2)}\right)(s)\right)\d s\\
&- \int_{0}^{t} S(t-s)\left(\nabla Q\left(a_{(n)},\u_{(n-1)}\right)-\nabla Q\left(a_{(n-1)},\u_{(n-2)}\right)\right) \d s \notag \\
&+ \int_{0}^{t} S(t-s) \left(\frac{\mu}{\bar{\rho}\left(1+a_{(n)}\right)}-\frac{\mu}{\bar{\rho}\left(1+a_{(n-1)}\right)}\right) \triangle \v_{(n-1)} \d s.\notag
\end{align}
Since
\begin{equation}
\left|S(t) f\right|_{q} \ls C t^{-\frac{N}{2\eth}}\left|f\right|_{p},\\
\end{equation}
 where $\frac{1}{\eth}=\frac{1}{p}-\frac{1}{q}$, we have
 \begin{equation}
\left|S(t) \left(\v_{(n)}-\v_{(n-1)}\right)(0)\right|_{p}\ls C \left|\left(\v_{(n)}-\v_{(n-1)}\right)(0)\right|_{p}=0.\\
\end{equation}
Specially, it holds that
 \begin{equation}
\left|S(t) f\right|_{p} \ls M_{2} t^{-\frac{N}{2p}}\left|f\right|_{\frac{p}{2}}. \\
\end{equation}
For $p>N$, there holds
\begin{align}
&\left|\int_{0}^{t} S(t-s)\left(B\left(\v_{(n-1)}+\z_{(n-1)}\right)(s)-B\left(\v_{(n-2)}+\z_{(n-2)}\right)(s)\right)\d s\right|_{p}\notag\\
\ls &\int_{0}^{t} \left| S(t-s) \left(-\nabla\cdot\left(\u_{(n-1)}\otimes \u_{(n-1)} \right)+\nabla\cdot\left(\u_{(n-2)}\otimes\u_{(n-2)} \right)\right)\right|_{p}\d s\\
\ls &\int_{0}^{t}\left|- S(t-s) \nabla\cdot \left(\left(\u_{(n-1)}  \right)\otimes\left(\left(\u_{(n-1)}-\u_{(n-2)} \right) \right)\right.\right.\notag\\
&\left. \qquad -S(t-s) \nabla\cdot\left(\u_{(n-2)}\otimes\left(\u_{(n-1)}-\u_{(n-2)} \right) \right)\right|_{p}\d s.\notag
\end{align}
The first term in the right of the above inequality, satisfies
\begin{align}
 &  \int_{0}^{t} \left| S(t-s)\nabla\cdot \left(\u_{(n-1)} \otimes\left(\u_{(n-1)}-\u_{(n-2)}\right)\right)\right|_{p}\d s \notag\\
\ls & \int_{0}^{t} \frac{M_{1}}{(t-s)^{\frac{1}{2}}}\left| S(t-s) \left(\u_{(n-1)} \otimes\left(\u_{(n-1)}-\u_{(n-2)}\right)\right)\right|_{p}\d s \notag\\
\ls &  \int_{0}^{t} \frac{M}{(t-s)^{\frac{1}{2}+\frac{N}{2p}}} \left|\left(\u_{(n-1)} \otimes\left(\u_{(n-1)}-\u_{(n-2)}\right)\right)\right|_{\frac{p}{2}}\d s  \\
\ls &   \left(\int_{0}^{t} \frac{M}{(t-s)^{\frac{1}{2}+\frac{N}{2p}}} \d s\right) \left|\sup\limits_{s\in [0,t]} \left(\left(\u_{(n-1)} \right)\otimes\left(\u_{(n-1)}-\u_{(n-2)}\right)\right)\right|_{\frac{p}{2}} \notag\\
\ls &   t^{\frac{1}{2}\left(1-\frac{N}{p}\right)}M \sup\limits_{s\in [0,t]}\left| \u_{(n-1)}\right|_{p}   \sup\limits_{s\in [0,t]}\left| \left(\u_{(n-2)} -\u_{(n-1)} \right)\right|_{p} .\notag
\end{align}
Similarly, we estimate
\begin{align}
 &\int_{0}^{t} \left| S(t-s) \nabla\cdot \left(\left(\u_{(n-2)} \right)\otimes\left( \u_{(n-2)} -\u_{(n-1)}  \right)\right)\right|_{p}\d s\\
 \ls &  Mt^{\frac{1}{2}\left(1-\frac{N}{p}\right)} \sup\limits_{s\in [0,t]}\left| \u_{(n-2)}  \right|_{p} \sup\limits_{s\in [0,t]}\left| \left( \u_{(n-2)} -\u_{(n-1)}  \right)\right|_{p}.\notag
\end{align}
%By Theorem \ref{Lp for z}, we obtain
%\begin{align}
%  &\sup\limits_{s\in [0,t]}\left| \v_{(n-1)}(s)+\v_{(n-2)}(s) +\z_{(n-2)} +\z_{(n-1)}\right|_{p}\notag\\
%\ls & \sup\limits_{s\in [0,t]}\left(\left| \v_{(n-1)}(s)\right|_{p}+\left| \v_{(n-2)}(s)\right|_{p}+\left| \z_{(n-2)} \right|_{p}+ \left| \z_{(n-1)} \right|_{p} \right)\\
%%\ls & 3^{\frac{2r-1}{2}} \left(\mathbb{E}\left[\sup\limits_{s\in (0,t)}\left(\left(\int_{\mathbb{T}^{d}}| \left|\v_{(n-1)}(s)\right|^{p}\d x \right)^{\frac{2r}{p}}+\left(\int_{\mathbb{T}^{d}} \left|\v_{(n-2)}(s)\right|^{p}\d x \right)^{\frac{2r}{p}}+ 2\left(\int_{\mathbb{T}^{d}} \left|\z(s)\right|^{p}\d x \right)^{\frac{2r}{p}}\right)\right]\right)^{\frac{1}{2}}\notag\\
%%\ls & 3^{\frac{2r-1}{2}} \left(\mathbb{E}\left[\sup\limits_{s\in (0,t)}\left(\left(\int_{\mathbb{T}^{d}} \left|\v_{(n-1)}(s)\right|^{2r}\d x \right)+\left(\int_{\mathbb{T}^{d}} \left|\v_{(n-2)}(s)\right|^{2r}\d x \right)+ 2\left(\int_{\mathbb{T}^{d}} \left|\z(s)\right|^{2r}\d x \right)\right)\right]\right)^{\frac{1}{2}}\notag\\
%\ls &(4M+2) K_{0}(t). \notag
%\end{align}
In summary, we estimate
\begin{align}\label{contraction of B(u)}
&\left|\int_{0}^{t} S(t-s)\left(B\left(\v_{(n-1)}+\z_{(n-1)}\right)(s)-B\left(\v_{(n-2)}+\z_{(n-2)}\right)(s)\right)\d s\right|_{p}\\
\ls &M(4M+2)K_{0} t^{\frac{1}{2}\left(1-\frac{N}{p}\right)}\left( \sup\limits_{s\in [0,t]}\left| \left( \v_{(n-2)} -\v_{(n-1)} \right)\right|_{p}+ \sup\limits_{s\in [0,t]}\left|\left( \z_{(n-2)} - \z_{(n-1)}\right)\right|_{p}\right). \notag
\end{align}

For the pressure term, it holds
\begin{align}\label{nabla Qn- nabla Qn-1}
 & \nabla Q\left(a_{(n)},\u_{(n-1)}\right)\d t-\nabla Q\left(a_{(n-1)},\u_{(n-2)}\right)\d t\notag\\
=& \left(\nabla \triangle^{-1}\left(\nabla \u_{(n-2)}: \nabla \u_{(n-2)}^{T}\right)-\nabla \triangle^{-1}\left(\nabla \u_{(n-1)}: \nabla \u_{(n-1)}^{T}\right)\right)\d t \notag\\
& + \nabla \triangle^{-1}\left(\nabla \cdot \left(\frac{\mu}{\rho_{(n-1)}}\triangle \v_{(n-2)}\right)-\nabla \cdot \left(\frac{\mu}{\rho_{(n)}}\triangle \v_{(n-1)}\right)\right)\d t.\notag
\end{align}
The first term in the right hand side of \eqref{nabla Qn- nabla Qn-1} holds
\begin{align}
&\left(\nabla \triangle^{-1}\left(\nabla \u_{(n-2)}: \nabla \u_{(n-2)}^{T}\right)-\nabla \triangle^{-1}\left(\nabla \u_{(n-1)}: \nabla \u_{(n-1)}^{T}\right)\right)\d t\notag\\
=&\left(\nabla \left(\u_{(n-2)}\otimes \u_{(n-2)}\right)-\nabla \left(\u_{(n-1)}\otimes \u_{(n-1)}\right)\right)\d t,
\end{align}
which can be estimated similarly to the bilinear term \eqref{contraction of B(u)}.
 The second term in the right hand side of \eqref{nabla Qn- nabla Qn-1} is estimated as
\begin{align}
&\left|\int_{0}^{t} S(t-s) \nabla \triangle^{-1}\left(\nabla \cdot \left(\frac{\mu}{\rho_{(n-1)}}\triangle \v_{(n-2)}\right)-\nabla \cdot \left(\frac{\mu}{\rho_{(n)}}\triangle \v_{(n-1)}\right)\right) \d s \right|_{p} \notag \\
\ls &\int_{0}^{t}\left| S(t-s) \nabla \triangle^{-\frac{1}{2}}\left(\left(\frac{\mu}{\rho_{(n-1)}}\triangle \v_{(n-2)}\right)- \left(\frac{\mu}{\rho_{(n)}}\triangle \v_{(n-1)}\right)\right)\right|_{p}\d s \notag \\
= &\int_{0}^{t}\left| S(t-s) \nabla \triangle^{-\frac{1}{2}}\nabla\cdot\left(\frac{\mu}{\rho_{(n-1)}}\nabla \v_{(n-2)} - \frac{\mu}{\rho_{(n)}}\nabla \v_{(n-1)}\right)\right|_{p}\d s \\
&- \int_{0}^{t}\left| S(t-s) \nabla \triangle^{-\frac{1}{2}}\left(\nabla \left(\frac{\mu}{\rho_{(n-1)}}\right)\nabla \v_{(n-2)} - \nabla \left( \frac{\mu}{\rho_{(n)}}\right)\nabla \v_{(n-1)}\right)\right|_{p}\d s \notag \\
= & \int_{0}^{t}\left| S(t-s) \nabla \left(\frac{\mu}{\rho_{(n-1)}}\nabla \v_{(n-2)} - \frac{\mu}{\rho_{(n)}}\nabla \v_{(n-1)}\right)\right|_{p}\d s \notag\\
&- \int_{0}^{t}\left| S(t-s) \nabla \triangle^{-\frac{1}{2}}\left(\nabla \left(\frac{\mu}{\rho_{(n-1)}}\right)\nabla \v_{(n-2)} - \nabla \left( \frac{\mu}{\rho_{(n)}}\right)\nabla \v_{(n-1)}\right)\right|_{p}\d s. \notag
\end{align}
We estimate
\begin{align}
&\int_{0}^{t}\left| S(t-s) \nabla \left(\frac{\mu}{\rho_{(n-1)}}\nabla \v_{(n-2)} - \frac{\mu}{\rho_{(n)}}\nabla \v_{(n-1)}\right)\right|_{p}\d s \notag\\
\ls & \int_{0}^{t}\frac{\mu M}{\bar{\rho}(t-s)^{\frac{1}{2}}}\left| S(t-s) \left(\frac{\mu}{\rho_{(n-1)}}\nabla \v_{(n-2)} - \frac{\mu}{\rho_{(n)}}\nabla \v_{(n-1)}\right)\right|_{p}\d s\notag\\
\ls & \int_{0}^{t}\frac{\mu M}{\bar{\rho}(t-s)^{\frac{1}{2}+\frac{N}{2p}}}\left|\left(\frac{\mu}{\rho_{(n-1)}}\nabla \v_{(n-2)} - \frac{\mu}{\rho_{(n)}}\nabla \v_{(n-1)}\right)\right|_{\frac{p}{2}}\d s \notag\\
\ls & \frac{4\mu M}{\bar{\rho}}  \int_{0}^{t} \frac{1}{(t-s)^{\frac{1}{2}+\frac{N}{2p}}}\left|\left(a_{(n-1)}-a_{(n)}\right)\nabla\v_{(n-1)}(s)\right|_{\frac{p}{2}}\d s \\
&+\frac{4\mu M}{\bar{\rho}}  \int_{0}^{t} \frac{1}{(t-s)^{\frac{1}{2}+\frac{N}{2p}}}\left|a_{(n-1)}\nabla\left(\v_{(n-1)}-\v_{(n-2)}\right)(s)\right|_{\frac{p}{2}}\d s\notag\\
\ls & \frac{4\mu M}{\bar{\rho}} \sup\limits_{s \in [0,t]}\left|\left(a_{(n-1)}-a_{(n)}\right)\right|_{p}\sup\limits_{s \in [0,t]}\left|\v_{(n-1)}(s)\right|_{1,p} \int_{0}^{t} \frac{1}{(t-s)^{\frac{1}{2}+\frac{N}{2p}}}\d s \notag\\
&+\frac{4\mu M}{\bar{\rho}} \sup\limits_{s \in [0,t]}\left|\left(\v_{(n-1)}-\v_{(n-2)}\right)\right|_{p}\sup\limits_{s \in [0,t]}\left|a_{(n-1)}(s)\right|_{1,p} \int_{0}^{t} \frac{1}{(t-s)^{\frac{1}{2}+\frac{N}{2p}}}\d s \notag\\
\ls &\frac{4\mu M}{\bar{\rho}}\sup\limits_{s \in [0,t]}\left|\left(a_{(n-1)}-a_{(n)}\right)\right|_{p}2MK_{0} t^{\frac{1}{2}\left(1-\frac{N}{p}\right)} +\frac{4\mu M}{\bar{\rho}}\sup\limits_{s \in [0,t]}\left|\left(\v_{(n-1)}-\v_{(n-2)}\right)\right|_{p}2MK_{0} t^{\frac{1}{2}\left(1-\frac{N}{p}\right)}  \notag\\
= & \frac{8\mu}{\bar{\rho}}M^{2}K_{0} t^{\frac{1}{2}\left(1-\frac{N}{p}\right)} \sup\limits_{s \in [0,t]}\left|\left(a_{(n-1)}-a_{(n)}\right)\right|_{p}+\frac{8\mu}{\bar{\rho}}M^{2}K_{0} t^{\frac{1}{2}\left(1-\frac{N}{p}\right)} \sup\limits_{s \in [0,t]}\left|\left(\v_{(n-1)}-\v_{(n-2)}\right)\right|_{p}. \notag
\end{align}
As in \eqref{div inverse analysis}, we have
\begin{align}
&  \triangle^{-\frac{1}{2}}\left(\nabla \left(\frac{\mu}{\rho_{(n-1)}}\right)\nabla \v_{(n-2)} - \nabla \left( \frac{\mu}{\rho_{(n)}}\right)\nabla \v_{(n-1)}\right)\notag\\
=& \triangle^{-\frac{1}{2}}\left(\nabla \left(\frac{\mu}{\rho_{(n-1)}}\right)\nabla \left(\v_{(n-2)}-\v_{(n-1)}\right)+ \nabla \left(\frac{\mu}{\rho_{(n-1)}}-\frac{\mu}{\rho_{(n)}}\right)\nabla \v_{(n-1)}\right)\\
\ls & 4\sup\limits_{s \in [0,t]}\left|\nabla\frac{\mu}{\bar{\rho}a_{(n-1)}}\left(s, x \right)\right| \left| \left(\v_{(n-2)}-\v_{(n-1)}\right) \left(s, x \right)\right|
 + 4\frac{\mu}{\bar{\rho}}\sup\limits_{s \in [0,t]}\left|a_{(n-1)}-a_{(n)}\right| \left|\nabla \v_{(n-1)}\right|. \notag
\end{align}
Similarly we have
\begin{align}
& \int_{0}^{t}\left| S(t-s) \nabla \triangle^{-\frac{1}{2}}\left(\nabla \left(\frac{\mu}{\rho_{(n-1)}}\right)\nabla \v_{(n-2)} - \nabla \left( \frac{\mu}{\rho_{(n)}}\right)\nabla \v_{(n-1)}\right)\right|_{p}\d s\notag\\
\ls & \frac{8\mu}{\bar{\rho}}M^{2}K_{0} t^{\frac{1}{2}\left(1-\frac{N}{p}\right)} \sup\limits_{s \in [0,t]}\left|\left(a_{(n-1)}-a_{(n)}\right)\right|_{p}+\frac{4\mu M }{\bar{\rho}} K_{0} t^{\frac{1}{2}\left(1-\frac{N}{p}\right)} \sup\limits_{s \in [0,t]}\left|\left(\v_{(n-1)}-\v_{(n-2)}\right)\right|_{p}.\notag
\end{align}
In summary, for the pressure term, there holds
\begin{align}
 & \left|\int_{0}^{t} S(t-s)\left( \nabla Q\left(a_{(n)},\u_{(n-1)}\right)-\nabla Q\left(a_{(n-1)},\u_{(n-2)}\right)\right) \d s \right|_{p} \notag \\
\ls &  M(4M+2)K_{0} t^{\frac{1}{2}\left(1-\frac{N}{p}\right)} \sup\limits_{s\in [0,t]}\left| \left( \u_{(n-2)} -\u_{(n-1)}  \right)\right|_{p}\notag\\
 &+\left(\frac{16\mu}{\bar{\rho}}M^{2}K_{0} t^{1+\frac{1}{2}\left(1-\frac{N}{p}\right)}
\right) \cdot \left(\sup\limits_{s \in [0,t]}\left|\left(\u_{(n-1)}-\u_{(n-2)}\right)(s,x)\right|_{p}\right)\\
&+\left(\frac{8\mu}{\bar{\rho}} M^{2}K_{0}+\frac{4\mu M }{\bar{\rho}} K_{0}\right) t^{\frac{1}{2}\left(1-\frac{N}{p}\right)}\sup\limits_{s \in [0,t]}\left|\left(\v_{(n-1)}-\v_{(n-2)}\right)(s,x)\right|_{p} \notag\\
\ls & \left( M(4M+2)K_{0} t^{\frac{1}{2}\left(1-\frac{N}{p}\right)}+\frac{16\mu}{\bar{\rho}}M^{2}K_{0} t^{1+\frac{1}{2}\left(1-\frac{N}{p}\right)}\right) \sup\limits_{s\in [0,t]} \left| \left(\z_{(n-1)} -\z_{(n-2)} \right)\right|_{p}\notag\\
&+ \left( \left( M(4M+2)K_{0} +\frac{8\mu}{\bar{\rho}} M^{2}K_{0}+\frac{4\mu M }{\bar{\rho}} K_{0}\right)t^{\frac{1}{2}\left(1-\frac{N}{p}\right)}+\frac{16\mu}{\bar{\rho}}M^{2}K_{0} t^{1+\frac{1}{2}\left(1-\frac{N}{p}\right)}\right)\notag\\
&\qquad \sup\limits_{s \in [0,t]}\left|\left(\v_{(n-1)}-\v_{(n-2)}\right)(s,x)\right|_{p} .\notag
\end{align}

The last integral in \eqref{equation for bar v -v} is estimated similarly as follows:
\begin{align}\label{last integral in bar v -v}
& \left|\int_{0}^{t} \bar{S}(t-s) \left(\frac{\mu}{\bar{\rho}\left(1+a_{(n)}\right)}-\frac{\mu}{\bar{\rho}\left(1+a_{(n-1)}\right)}\right) \triangle \v_{(n-1)} \d s\right|_{p} \notag\\
\ls & \int_{0}^{t} \left|\bar{S}(t-s)\left(\frac{\mu}{\bar{\rho}\left(1+a_{(n)}\right)}-\frac{\mu}{\bar{\rho}\left(1+a_{(n-1)}\right)}\right) \triangle \v_{(n-1)} \right|_{p}\d s \notag \\
= &  \int_{0}^{t} \left|\bar{S}(t-s)\nabla\cdot  \triangle^{-\frac{1}{2}}\left(\left(\frac{\mu}{\bar{\rho}\left(1+a_{(n)}\right)}-\frac{\mu}{\bar{\rho}\left(1+a_{(n-1)}\right)}\right) \triangle \v_{(n-1)} \right)\right|_{p}\d s \notag \\
\ls &  \int_{0}^{t} \frac{M_{1}}{(t-s)^{\frac{1}{2}}}\left|\bar{S}(t-s)\triangle^{-\frac{1}{2}}\left(\left(\frac{\mu}{\bar{\rho}\left(1+a_{(n)}\right)}-\frac{\mu}{\bar{\rho}\left(1+a_{(n-1)}\right)}\right) \triangle \v_{(n-1)} \right)\right|_{p}\d s \notag \\
\ls &  \int_{0}^{t} \frac{M}{(t-s)^{\frac{1}{2}+\frac{N}{2p}}}\left|\triangle^{-\frac{1}{2}}\left(\left(\frac{\mu}{\bar{\rho}\left(1+a_{(n)}\right)}-\frac{\mu}{\bar{\rho}\left(1+a_{(n-1)}\right)}\right) \triangle \v_{(n-1)} \right)\right|_{\frac{p}{2}}\d s  \\
=& \int_{0}^{t} \frac{\mu M}{\bar{\rho}(t-s)^{\frac{1}{2}+\frac{N}{2p}}}\left|\triangle^{-\frac{1}{2}}\nabla\cdot\left(\left(\frac{1}{\left(1+a_{(n)}\right)}
          -\frac{1}{\left(1+a_{(n-1)}\right)}\right)\nabla\v_{(n-1)}(s)\right)\right.\notag\\
&\quad \left.-\triangle^{-\frac{1}{2}} \left(\nabla\v_{(n-1)}(s) \cdot\nabla \left(\frac{1}{\left(1+a_{(n)}\right)}-\frac{1}{\left(1+a_{(n-1)}\right)}\right)\right)\right|_{\frac{p}{2}}\d s \notag \\
%=& \int_{0}^{t} \frac{\mu}{\bar{\rho}(t-s)^{\frac{1}{2}+\frac{N}{2p}}}\left|\left(\frac{1}{\left(1+a_{(n)}\right)}-\frac{1}{\left(1+a_{(n-1)}\right)}\right)\nabla\v_{(n-1)}(s)\right. \notag \\
%&\quad \left.-\triangle^{-\frac{1}{2}} \left(\nabla\v_{(n-1)}(s) \nabla \left(\frac{1}{\left(1+a_{(n)}\right)}-\frac{1{\left(1+a_{(n-1)}\right)}\right)\right)\right)\right|_{\frac{p}{2}}\d s \notag\\
\ls & \int_{0}^{t} \frac{\mu M}{\bar{\rho}(t-s)^{\frac{1}{2}+\frac{N}{2p}}}\left|\left(\frac{1}{\left(1+a_{(n)}\right)}-\frac{1}{\left(1+a_{(n-1)}\right)}\right)\nabla\v_{(n-1)}(s)\right|_{\frac{p}{2}}\d s \notag \\
   & + \int_{0}^{t} \frac{\mu M}{\bar{\rho}(t-s)^{\frac{1}{2}+\frac{N}{2p}}}\left|\triangle^{-\frac{1}{2}} \left(\nabla\v_{(n-1)}(s)\cdot \nabla \left(\frac{1}{\left(1+a_{(n)}\right)}-\frac{1}{\left(1+a_{(n-1)}\right)}\right)\right)\right|_{\frac{p}{2}}\d s. \notag
\end{align}
 The first term in the right hand-side of \eqref{last integral in bar v -v} is estimated by
\begin{align}
  &\frac{1}{\bar{\rho}} \int_{0}^{t} \frac{\mu M}{(t-s)^{\frac{1}{2}+\frac{N}{2p}}}\left|\left(\frac{1}{\left(1+a_{(n)}\right)}-\frac{1}{\left(1+a_{(n-1)}\right)}\right)\nabla\v_{(n-1)}(s)\right|_{\frac{p}{2}}\d s \notag \\
\ls & \frac{4\mu M}{\bar{\rho}}  \int_{0}^{t} \frac{1}{(t-s)^{\frac{1}{2}+\frac{N}{2p}}}\left|\left(a_{(n-1)}-a_{(n)}\right)\nabla\v_{(n-1)}(s)\right|_{\frac{p}{2}}\d s \\
\ls & \frac{4\mu M}{\bar{\rho}} \sup\limits_{s \in [0,t]}\left|\left(a_{(n-1)}-a_{(n)}\right)\right|_{p}\sup\limits_{s \in [0,t]}\left|\v_{(n-1)}(s)\right|_{1,p} \int_{0}^{t} \frac{1}{(t-s)^{\frac{1}{2}+\frac{N}{2p}}}\d s \notag\\
\ls &\frac{8\mu M^{2}K_{0}}{\bar{\rho}}\sup\limits_{s \in [0,t]}\left|\left(a_{(n-1)}-a_{(n)}\right)\right|_{p} t^{\frac{1}{2}\left(1-\frac{N}{p}\right)}
=  \frac{8\mu M^{2}K_{0}}{\bar{\rho}} t^{\frac{1}{2}\left(1-\frac{N}{p}\right)} \sup\limits_{s \in [0,t]}\left|\left(a_{(n-1)}-a_{(n)}\right)\right|_{p}. \notag
\end{align}
The second term in \eqref{last integral in bar v -v} is estimated by
\begin{align}
    & \int_{0}^{t} \frac{\mu M}{\bar{\rho}(t-s)^{\frac{1}{2}+\frac{N}{2p}}}\left|\triangle^{-\frac{1}{2}} \left(\nabla\v_{(n-1)}(s) \cdot \nabla \left(\frac{1}{1+a_{(n)}}-\frac{1 }{1+a_{(n-1)}}\right)\right)\right|_{\frac{p}{2}}\d s \notag \\
\ls &  \int_{0}^{t} \frac{4\mu M}{\bar{\rho}(t-s)^{\frac{1}{2}+\frac{N}{2p}}}\sup\limits_{s \in [0,t]}\left|\left(\v_{(n-1)}(s) \nabla\left(\frac{1}{1+a_{(n)}}-\frac{1 }{1+a_{(n-1)}}\right)\right)\right|_{\frac{p}{2}}\d s \notag \\
\ls & \int_{0}^{t}\frac{4\mu M}{\bar{\rho}(t-s)^{\frac{1}{2}+\frac{N}{2p}}}\sup\limits_{s \in [0,t]}\left|\left(\frac{1}{1+a_{(n)}}-\frac{1 }{1+a_{(n-1)}}\right)\right|_{p}\sup\limits_{s \in [0,t]}\left| \nabla \v_{(n-1)}(s)\right|_{p}\d s \\
=& \int_{0}^{t}\frac{16\mu M}{\bar{\rho}(t-s)^{\frac{1}{2}+\frac{N}{2p}}}\sup\limits_{s \in [0,t]}\left|\left( a_{(n)}- a_{(n-1)}\right)\right|_{p}\sup\limits_{s \in [0,t]}\left| \v_{(n-1)}(s)\right|_{1,p}\d s, \notag\\
\ls & \frac{32\mu M^{2}K_{0}}{\bar{\rho}} t^{\frac{1}{2}\left(1-\frac{N}{p}\right)}\sup\limits_{s \in [0,t]}\left|\left(a_{(n-1)}-a_{(n)}\right)\right|_{p}.\notag
\end{align}
With \eqref{contraction of a}, it holds
\begin{align}
& \left|\int_{0}^{t} \bar{S}(t-s) \left(\frac{\mu}{\bar{\rho}\left(1+a_{(n)}\right)}-\frac{\mu}{\bar{\rho}\left(1+a_{(n-1)}\right)}\right) \triangle \v_{(n-1)} \d s\right|_{p} \notag\\
\ls & \frac{40\mu M^{2}K_{0}}{\bar{\rho}} t^{\frac{1}{2}\left(1-\frac{N}{p}\right)}
 \left(\sup\limits_{s \in [0,t]}\left|a_{(n-1)}-a_{(n-2)}\right|_{p}\right)\\
\ls &  \frac{80\mu M^{3}K_{0}^{2}}{\bar{\rho}} t^{\frac{1}{2}\left(1-\frac{N}{p}\right)+1} \left(\sup\limits_{s \in [0,t]}\left|\v_{(n-1)}-\v_{(n-2)}\right|_{p}+\sup\limits_{s \in [0,t]}\left|\z_{(n-1)}-\z_{(n-2)}\right|_{p}\right).\notag
\end{align}

In conclusion, combining the above estimates, we have
\begin{align}
& \left|\left(\v_{(n)}-\v_{(n-1)}\right)(t)\right|_{p} \notag \\
\ls &  \left|\int_{0}^{t} S(t-s)\left( B\left(\v_{(n-1)}+\z_{(n-1)}\right)-B\left(\v_{(n-2)}+\z_{(n-2)}\right)\right)\d s\right|_{p} \notag\\
    & +  \left|\int_{0}^{t} S(t-s)\left( \nabla Q\left(a_{(n)},\u_{(n-1)}\right)-\nabla Q\left(a_{(n-1)},\u_{(n-2)}\right)\right) \d s \right|_{p} \notag \\
    &+ \left|\int_{0}^{t} \bar{S}(t-s) \left(\frac{\mu}{\bar{\rho}\left(1+a_{(n)}\right)}-\frac{\mu}{\bar{\rho}\left(1+a_{(n-1)}\right)}\right) \triangle \v_{(n-1)} \d s\right|_{p} \\
\ls &M(4M+2)K_{0} t^{\frac{1}{2}\left(1-\frac{N}{p}\right)}\left( \sup\limits_{s\in [0,t]}\left| \left( \v_{(n-2)} -\v_{(n-1)} \right)\right|_{p}+ \sup\limits_{s\in [0,t]}\left|\left( \z_{(n-2)} - \z_{(n-1)}\right)\right|_{p}\right) \notag\\
  &\left( M(4M+2)K_{0} t^{\frac{1}{2}\left(1-\frac{N}{p}\right)}+\frac{16\mu}{\bar{\rho}}M^{2}K_{0} t^{1+\frac{1}{2}\left(1-\frac{N}{p}\right)}\right) \sup\limits_{s\in [0,t]} \left| \left(\z_{(n-1)} -\z_{(n-2)} \right)\right|_{p}\notag\\
&+ \left( \left( M(4M+2)K_{0} +\frac{8\mu}{\bar{\rho}} M^{2}K_{0}+\frac{4\mu M }{\bar{\rho}} K_{0}\right)t^{\frac{1}{2}\left(1-\frac{N}{p}\right)}+\frac{16\mu}{\bar{\rho}}M^{2}K_{0} t^{1+\frac{1}{2}\left(1-\frac{N}{p}\right)}\right)\notag\\
&\qquad \sup\limits_{s \in [0,t]}\left|\left(\v_{(n-1)}-\v_{(n-2)}\right)(s,x)\right|_{p}\notag\\
&+  \frac{80\mu M^{3}K_{0}^{2}}{\bar{\rho}} t^{\frac{1}{2}\left(1-\frac{N}{p}\right)+1} \left(\sup\limits_{s \in [0,t]}\left|\v_{(n-1)}-\v_{(n-2)}\right|_{p}+\sup\limits_{s \in [0,t]}\left|\z_{(n-1)}-\z_{(n-2)}\right|_{p}\right) \notag\\
\triangleq &  \alpha(t)\left(\sup\limits_{s \in [0,t]}\left|\z_{(n-1)}-\z_{(n-2)}\right|_{p}+\sup\limits_{s \in [0,t]}\left|\v_{(n-1)}-\v_{(n-2)}\right|_{p}\right) .\notag
\end{align}
where
\begin{align}
\alpha(t)=&\left( \left( 2M(4M+2)K_{0} +\frac{8\mu}{\bar{\rho}} M^{2}K_{0}+\frac{4\mu M }{\bar{\rho}} K_{0}\right)t^{\frac{1}{2}\left(1-\frac{N}{p}\right)}\right)\\
&+\frac{16\mu}{\bar{\rho}}M^{2}K_{0} t^{1+\frac{1}{2}\left(1-\frac{N}{p}\right)}+ \frac{80\mu M^{3}K_{0}^{2}}{\bar{\rho}} t^{\frac{1}{2}\left(1-\frac{N}{p}\right)+1}.\notag
\end{align}
Taking $\bar{T}$ such that
\begin{align}
\alpha (\bar{T}) <\frac{1}{2} ,\quad \bar{T}\ls T_{1},
\end{align}
we have
\begin{align}\label{alpha involved induction of v}
    \sup\limits_{s \in [0,t]} \left|\v_{(n)}-\v_{(n-1)}\right|_{p}\ls \alpha \left(\sup\limits_{s \in [0,t]}\left|\v_{(n-1)}-\v_{(n-2)}\right|_{p}+\sup\limits_{s \in [0,t]}\left|\z_{(n-1)}-\z_{(n-2)}\right|_{p}\right).
\end{align}
%then the operator $F$ is a contraction map. \textbf{By fixed point theorem, there exists a local mild solution to the linearized system.} \\
Moreover, recalling that \eqref{beta involved contraction of z},
\begin{align}
 \beta \triangleq \bar{T}^{1+\frac{1}{2}\left(1-\frac{N}{p}\right)}\frac{40\mu}{\bar{\rho}} M^{2} K_{0}^{3},
\end{align}
$\beta\ls \alpha $, by combining \eqref{alpha involved induction of v} and \eqref{beta involved contraction of z}, estimates for $ \sup\limits_{s \in [0,t]} \left| \z_{(n)}(t) -\z_{(n-1)}(t)\right|_{p}$, we obtain
\begin{align}
&  \sup\limits_{s \in [0,t]}\left| \z_{(n)}(t) -\z_{(n-1)}(t)\right|_{p}+  \sup\limits_{s \in [0,t]}\left|\v_{(n)}-\v_{(n-1)}\right|_{p} \\
\ls &2\alpha  \left( \sup\limits_{s \in [0,t]}\left|\v_{(n-1)}-\v_{(n-2)}\right|_{p} + \sup\limits_{s \in [0,t]}\left|\z_{(n-1)}-\z_{(n-2)}\right|_{p}\right),\notag
\end{align}
where $\alpha < \frac{1}{2}$. Therefore, $\left\{\z_{(n)}\right\}$ and $\left\{\v_{(n)}\right\}$ are the Cauchy sequences in $ \in C\left([0, \bar{T}]; L^{p}\left(\mathbb{T}^{N}\right)\right)$.
Since
 \begin{align}
  \sup\limits_{s \in [0,t]} \left| a_{(n)}-a_{(n-1)}\right|_{p}
\ls  t 2MK_{0}^{2}\left( \sup\limits_{s \in [0,t]}\left|\v_{(n-1)}-\v_{(n-2)}\right|_{p} + \sup\limits_{s \in [0,t]}\left|\z_{(n-1)}-\z_{(n-2)}\right|_{p}\right),
\end{align}
 $\left\{a_{(n)}\right\}$ is a Cauchy sequence as well. Therefore, the local existence and uniqueness of the mild solution to \eqref{sto inhomo incom NS} is obtained, with $a \in C\left([0, \bar{T}]; L^{p}\left(\mathbb{T}^{N}\right)\right)$, $\u \in  C\left([0, \bar{T}]; L^{p}\left(\mathbb{T}^{N}\right)\right)$.

\smallskip
\smallskip

\section{The global existence}

For any $T$, given the initial condition $\rho_{0}\in H^{3}\left(\mathbb{T}^{N}\right)$, $\u_{0}\in H^{3}\left(\mathbb{T}^{N}\right)$ $\mathbb{P}$ {\rm a.s.}, there exist the energy solution
$\rho \in C\left([0, T];  H^{3}\left(\mathbb{T}^{N}\right) \cap C^{1,\ell}\left(\mathbb{T}^{N}\right)\right)$ and $\u\in C\left([0, T];  H^{3}\left(\mathbb{T}^{N}\right)\right)$, $\mathbb{P}$ {\rm a.s.} Namely, the energy solution will not blow up until any time $T$. By Zorn's lemma, the energy estimate holds globally in time,  $\rho \in C\left([0, \infty); H^{3}\left(\mathbb{T}^{N}\right)\cap C^{1,\ell}\left(\mathbb{T}^{N}\right)\right)$, $\u\in  C\left([0, \infty); H^{s}\left(\mathbb{T}^{N}\right)\right)$. Since a local $L^{p}$ mild solution already exists in $\left[0, \bar{T}\right]$, by the uniform estimate Lemma \ref{priori estimate of v and z} in the following subsection and Sobolev's embedding, there holds
\begin{align}
\left|\u\left(t\right)\right|_{2,p}\ls \left\|\u\left(t\right)\right\|_{3} \ls \hat{C}(T) \left\|\u_{0}\right\|_{3},~ p\ls 6, ~t \in \left[0, \bar{T}\right].
\end{align}
 With a modifying of the definition of $K_{0}$, if we define
\begin{align}
K_{0}=\max \left\{2 \hat{C}(T)\left\|\u_{0}\right\|_{3}, 2 \left(\left| a_{0}(x)\right|_{2,p} + T \hat{C}(T) \left\|\u_{0}\right\|_{3}\right), C_{5}, 1 \right\},
\end{align}
then for the new initial data $\u\left(\bar{T}\right)$, there holds
\begin{align}
\left|\u\left(\bar{T}\right)\right|_{2,p} \ls \left\|\u\left(\bar{T}\right)\right\|_{3} \ls \hat{C}(T) \left\|\u_{0}\right\|_{3},~ \forall k \in \mathds{N}^{+}, ~t\in [\bar{T}, \bar{T}+\tilde{T}].
\end{align}
Moreover, it holds that
\begin{align}
\left|a\left(\bar{T},x\right)\right|_{2,p} \ls \left| a_{0}(x)\right|_{2,p} + \bar{T} \left|\u\left(\bar{T},x\right)\right|_{2,p} \ls  \left| a_{0}(x)\right|_{2,p} + T \hat{C}(T) \left\|\u_{0}\right\|_{3}.
\end{align}
 Thus, there exists a unique solution in a new time interval $[\bar{T}, \bar{T}+\tilde{T}]$.
Repeating this procedure, the initial data in the time interval $[\bar{T}+(k-1)\tilde{T}, \bar{T}+k\tilde{T}]$, satisfies
\begin{align}
\left|\u\left(\bar{T}+(k-1)\tilde{T}\right)\right|_{2,p} \ls \left\|\u\left(\bar{T}+(k-1)\tilde{T}\right)\right\|_{3} \ls \hat{C}(T) \left\|\u_{0}\right\|_{3},~ \forall k \in \mathds{N}^{+},
\end{align}
and
\begin{align}
\left|a\left(\bar{T}+(k-1)\tilde{T},x\right)\right|_{2,p} \ls & \left| a_{0}(x)\right|_{2,p} + \left(\bar{T}+(k-1)\tilde{T}\right) \left|\u\left(\bar{T}+(k-1)\tilde{T},x\right)\right|_{2,p}\\
 \ls & \left| a_{0}(x)\right|_{2,p} + T \hat{C}(T) \left\|\u_{0}\right\|_{3}. \notag
\end{align}
Hence, we extend the mild solution to any $T$ globally.

 \subsection{The a priori energy estimates of $\u$}\label{priori estimate of v}

We need to do the {\it a priori} energy estimate for $\u$ in this section since we will utilize the incompressible condition.
\begin{lemma}\label{priori estimate of v and z}
%Given $\z\in L^{r}\left(\Omega; C\left(0,T; H^{3}\left(\mathbb{T}^N\right)\right)\right)$,
Let $a\in C\left([0,T]; H^{3}\left(\mathbb{T}^{N}\right)\right)$, $\u_{0}\in H^{3}\left(\mathbb{T}^{N}\right)$. Then, there holds
\begin{align}
\u \in  C\left([0,T]; H^{3}\left(\mathbb{T}^{N}\right)\right)\cap L^{2}\left(0,T; H^{4}\left(\mathbb{T}^{N}\right)\right), ~\mathbb{P} ~{\rm a.s.}
\end{align}
\end{lemma}
\begin{proof}
From
\begin{equation}\label{equation of v for energy estimate}
\d\u(t)-\frac{ \mu }{\bar{\rho}\left(1+a\right)}\triangle\u(t)\d t=(\u(t))\cdot\nabla(\u(t))\d t- \nabla Q\left(a,\u\right)\d t+ \Phi \d W,
\end{equation}
by It\^o's formula, it holds that
\begin{align}\label{Ito foumula}
&\d \frac{|\u|^{2}}{2}=\u\cdot \d \u +\frac{1}{2} \langle \Phi\d W, \Phi \d W \rangle \notag\\
=&\frac{ \mu }{\bar{\rho}\left(1+a\right)}\triangle \u(t)\cdot \u\d t + B(\u,\u)\cdot \u\d t - \nabla Q\left(a,\u\right)\cdot\u\d t\\
&+\Phi \cdot\u \d W+ \frac{1}{2}\left|\Phi\right|^{2}\d t,\notag
\end{align}
where we denotes $B(\u,\u):=(\u(t))\cdot\nabla(\u(t))$.
%We integrate \eqref{Ito foumula} with respect to $x$, and we have
%\begin{align}
%&\d \int_{\mathbb{T}^{N}}\frac{|\u|^{2}}{2}\d x-\int_{\mathbb{T}^{N}}\frac{\mu}{\bar{\rho}\left(1+a\right)}\triangle\u\cdot \u \d x \d t\\
%=&\int_{\mathbb{T}^{N}}B(\u,\u)\cdot \u\d x \d t- \int_{\mathbb{T}^{N}}\nabla Q\left(a,\u\right)\cdot \u \d x\d t+\int_{\mathbb{T}^{N}}\Phi\cdot\u \d W \d x+ \int_{\mathbb{T}^{N}} \frac{1}{2}\left|\Phi\right|^{2}\d x\d t. \notag
%\end{align}
For the sake of energy estimate, we multiply the above formula by density. Then, we have
\begin{align}\label{Ito foumula for u times rho}
 \rho \d \frac{|\u|^{2}}{2}=&\mu\triangle \u\cdot \u\d t + B(\rho\u,\u)\cdot \u\d t - \nabla P\cdot\u\d t\\
&+\rho\Psi\cdot\u \d W+ \frac{1}{2}\rho\left|\Psi\right|^{2}\d t. \notag
\end{align}
We integrate ~\eqref{Ito foumula for u times rho} with respect to ~$x$. Noticing that ~$\left(1+a\right)\bar{\rho} \geqslant \frac{1}{2}\bar{\rho}$, we have
\begin{align}
\int_{\mathbb{T}^{N}}\rho \d \frac{|\u|^{2}}{2} \d x \geqslant c_{0}\int_{\mathbb{T}^{N}}  \d \frac{|\u|^{2}}{2} \d x.
\end{align}
That is,
\begin{align}
\int_{0}^{t}\int_{\mathbb{T}^{N}}  \d \frac{|\u|^{2}}{2} \d x \d s \ls  \frac{1}{c_{0}}\int_{0}^{t}\int_{\mathbb{T}^{N}}\rho \d \frac{|\u|^{2}}{2} \d x \d s,
\end{align}
which can be controlled by the integral of the right hand side of \eqref{Ito foumula for u times rho}. We consider the right hand side of \eqref{Ito foumula for u times rho} term by term.
By manipulating integration by parts, we have
\begin{align}
-\mu \int_{0}^{t}\int_{\mathbb{T}^{N}}\triangle\u\cdot \u  \d x \d s=\mu\int_{0}^{t}\int_{\mathbb{T}^{N}} \left|\nabla \u\right|^{2} \d x \d s.
\end{align}
Using integration by parts, we get
\begin{align}
\langle B(\mathbf{a}, \mathbf{b}), \mathbf{c}\rangle=-\langle B(\mathbf{a},\mathbf{c}), \mathbf{b}\rangle, \quad\langle B(\mathbf{a}, \mathbf{b}), \mathbf{b}\rangle=0.
\end{align}
Thus, we obtain
\begin{align}
 &\int_{\mathbb{T}^{N}}B\left(\u,\u\right)\cdot\u\d x=0.
\end{align}
Again, integration by parts gives
\begin{align}\label{no pressure in energy estimate}
 -\int_{\mathbb{T}^{N}} \nabla Q\left(a,\u\right) \cdot \u \d x
=\int_{\mathbb{T}^{N}} Q\left(a,\u\right) \operatorname{div} \u \d x
=0.
\end{align}
By stochastic Fubini's theorem and B\"urkh\"older-Davis-Gundy's inequality, for any ~$r\geqslant 2$, we have
\begin{align}
&\mathbb{E}\left[\left|\int_{0}^{t} \int_{\mathbb{T}^{N}} \rho\Psi \cdot \u  \d W\d x \right|^{r}\right]
 = \mathbb{E}\left[\left|\int_{0}^{t} \int_{\mathbb{T}^{N}} \rho\Psi \cdot \u \d x \d W \right|^{r}\right]\notag \\
 \ls & \mathbb{E}\left[\left|\int_{0}^{t} \left|\int_{\mathbb{T}^{N}} \rho\Psi \cdot \u \d x \right|^{2}\d s \right|^{\frac{r}{2}}\right]
 \ls \mathbb{E}\left[\left|\int_{0}^{t} \left|\int_{\mathbb{T}^{N}} C \rho|\u|\d x \right|^{2} \d s \right|^{\frac{r}{2}}\right] \notag\\
 \ls & \mathbb{E}\left[\left|\int_{0}^{t} \left|\int_{\mathbb{T}^{N}} C\left(\rho^{2}+ |\u|^{2}\right)\d x \right|^{2} \d s \right|^{\frac{r}{2}}\right] \\
 \ls & \mathbb{E}\left[\left|\int_{0}^{t} \left|\int_{\mathbb{T}^{N}} C\left(\rho_{0}^{2}+ |\u|^{2}\right)\d x \right|^{2} \d s \right|^{\frac{r}{2}}\right] \notag \\
 \ls & C_{r}\mathbb{E}\left[ \left| t\int_{\mathbb{T}^{N}} \rho_{0}^{2}\d x\right|^{r}  \right]+ \mathbb{E}\left[  \int_{0}^{t} \left| C \int_{\mathbb{T}^{N}} |\u|^{2} \d x\right|^{r} \d s \right], \notag
\end{align}
where ~$C_{r}$ is the $r$-th power of some constant. 
It is straightforward to estimate
\begin{align}
\mathbb{E}\left[\left|\int_{0}^{t} \int_{\mathbb{T}^{N}} \frac{1}{2}\rho\left|\Psi\right|^{2}\d x\d s \right|^{r}\right] \ls C_{r}\mathbb{E}\left[ \left| t\int_{\mathbb{T}^{N}} \rho_{0}^{2}\d x\right|^{r} \right].
\end{align} 
In summary, the following inequality holds
\begin{align}
&\mathbb{E}\left[\left|\int_{\mathbb{T}^{N}}\frac{1}{2}\left|\u\right|^{2}\d x \right|^{r}\right]+\mathbb{E}\left[\left|\int_{0}^{t}\mu \int_{\mathbb{T}^{N}}\left|\nabla \u\right|^{2} \d x \right|^{r}\right]\\
\ls &\mathbb{E}\left[\left|\int_{\mathbb{T}^{N}}\frac{1}{2}\left|\u_{0}\right|^{2}\d x \right|^{r}\right]+C_{r}\mathbb{E}\left[ \left| t\int_{\mathbb{T}^{N}} \rho_{0}^{2}\d x\right|^{r} \right]+ C\int_{0}^{t}\mathbb{E}\left[\left| \int_{\mathbb{T}^{N}} |\u|^{2} \d x\right|^{r}\right] \d s, \notag
\end{align}
where ~$C_{r}$ is the $r$-th power of some constant. 
By Gr\"onwall's inequality, the following estimates hold
\begin{align}
\mathbb{E}\left[\left|\int_{\mathbb{T}^{N}}\frac{1}{2}\left|\u\right|^{2}\d x \right|^{r}\right] \ls & C(r, t,\rho_{0})\mathbb{E}\left[\left|\int_{\mathbb{T}^{N}}\frac{1}{2}\left|\u_{0}\right|^{2}\d x \right|^{r}\right] \notag\\
\ls &C(r, T,\rho_{0})\mathbb{E}\left[\left|\int_{\mathbb{T}^{N}}\frac{1}{2}\left|\u_{0}\right|^{2}\d x \right|^{r}\right],\\
\mathbb{E}\left[\left|\int_{0}^{t}\int_{\mathbb{T}^{N}}\left|\nabla \u\right|^{2} \d x \right|^{r}\right]  %\ls & C(t)\mathbb{E}\left[\left|\int_{\mathbb{T}^{N}}\frac{1}{2}\left|\u_{0}\right|^{2}\d x \right|^{r}\right]\notag\\
 \ls &C(r, T,\rho_{0})\mathbb{E}\left[\left|\int_{\mathbb{T}^{N}}\frac{1}{2}\left|\u_{0}\right|^{2}\d x \right|^{r}\right],
\end{align}
where $C(r, T,\rho_{0})$ is the $r$-th power of some constant dependent of $T$ and $\rho_{0}$.
Taking the first-order derivative to \eqref{equation of v for energy estimate}, applying ~It\^o's formula for ~$\left|\nabla \u\right|^{2}$, integrating with respect to $x$ and $t$, similarly we can obtain the high order estimates:
\begin{align}
&\mathbb{E}\left[\left|\int_{\mathbb{T}^{N}}\frac{1}{2}\left|\nabla^{k} \u\right|^{2}\d x \right|^{r}\right] \ls C(r, T,\rho_{0})\mathbb{E}\left[\left|\int_{\mathbb{T}^{N}}\frac{1}{2}\left|\nabla^{k} \u_{0}\right|^{2}\d x \right|^{r}\right] ,\\
&\mathbb{E}\left[\left|\int_{0}^{T}\int_{\mathbb{T}^{N}}\left|\nabla^{k+1} \u\right|^{2} \d x \d s \right|^{r}\right]  \ls C(r, T,\rho_{0})\mathbb{E}\left[\left|\int_{\mathbb{T}^{N}}\frac{1}{2}\left|\nabla^{k} \u_{0}\right|^{2}\d x \right|^{r}\right] ,
\end{align}
~$k=1, ~2, ~3$.

By Lemma \ref{arbitrariness of r}, we have the following estimates:
\begin{align}
&\int_{\mathbb{T}^{N}}\frac{1}{2}\left|\nabla^{k} \u\right|^{2}\d x  \ls C(T)\int_{\mathbb{T}^{N}}\frac{1}{2}\left|\nabla^{k} \u_{0}\right|^{2}\d x ,\\
& \int_{0}^{t}\int_{\mathbb{T}^{N}}\left|\nabla^{k+1} \u\right|^{2} \d x  \ls C(T)\int_{\mathbb{T}^{N}}\frac{1}{2}\left|\nabla^{k} \u_{0}\right|^{2}\d x,
\end{align}
$\mathbb{P}$ a.s., $k=0,1, 2, 3$, where $C(T)$ depends on $C_{\Phi}$ and $r$.
\hfill $\square$
\end{proof}

\begin{remark}
If $a_{0}\in C\left([0,T]; H^{3}\left(\mathbb{T}^{N}\right)\right)$, $\u\in C\left([0,T]; H^{3}\left(\mathbb{T}^{N}\right)\right)$, $\mathbb{P}$ {\rm a.s.}, then there holds\\
$a \in  C\left(0,T; H^{3}\left(\mathbb{T}^{N}\right)\right)$, $\mathbb{P}$ {\rm a.s.}
%Since we used a generalized Gr\"onwall's inequality and the norm of solutions are bounded by the constants dependent on time, we can not estimate the decay rate as time goes to infinity.
\end{remark}
\begin{remark}
Since we used a generalized Gr\"onwall's inequality, and that the norm of solutions are bounded by the constants dependent of time, we can not estimate the decay rate as time goes to infinity.
\end{remark}
\begin{remark}
If the viscosity coefficient $\mu\left(\rho\right)$ is density-dependent, and $\mu\left(\rho\right)\geqslant \underline{\mu}>0 $, then there is such a global mild solution of the system \eqref{sto inhomo incom NS} as well, by repeating the above arguments. However, our method could not apply to the case that $\mu\left(\rho\right)$ degenerates at the vacuum.
\end{remark}
\subsection{Remark on the deterministic problem}
For the deterministic case, there is no need to estimates the stochastic integral. Without separating the systems, by similar estimate for the deterministic terms, and applying Banach's fixed point theorem,  one can obtain the following conclusion.
\begin{corollary}\label{deterministic theorem}
For $1\ls N\ls 3$, $N<p\ls 6$, given $\rho_{0}\in H^{3}\left(\mathbb{T}^{N}\right)$, $0<c_{0} \leq \rho_0 \leq C_{0}$, $\u_{0}\in H^{3}\left(\mathbb{T}^{N}\right)$, if there exists a constant $C>0$, such that
\begin{align}
\left\|\rho_{0}\right\|_{H^{3}},\ls C,~ \left\|\u_{0}\right\|_{H^{3}} \ls C, %\quad \sum\limits_{k}\sum\limits_{j}\lambda_{j}^{N+1}\sup_{t\in [0,T]}\left|\left(\Phi_{k},\tilde{e}_{\mathbf{j}}\right)\right|^{2}\ls C,
\end{align}
 then there exists a unique global mild solution to IINS: \par
 \qquad $\rho\in C\left([0, T];  L^{p}\left(\mathbb{T}^{N}\right)\right)$, $\u\in  C\left([0, T];  L^{p}\left(\mathbb{T}^{N}\right)\right)$.\\
   Besides, there exists a global energy solution to IINS: \par
\qquad $\rho\in  C\left([0, T];  H^{3}\left(\mathbb{T}^{N}\right)\right)$, $\u\in  C\left([0,T]; H^{3}\left(\mathbb{T}^{N}\right)\right)\cap L^{2}\left(0,T; H^{4}\left(\mathbb{T}^{N}\right)\right)$.
\end{corollary}

 \smallskip
 \smallskip
\section{Appendix}

%\begin{theorem}\label{Trace theorem}(Trace theorem)
%Assume $U$ is bounded and $\partial U$ is $C^{1}$. Then there exists a bounded linear operator
%\begin{equation}
%T: W^{1,p}(U)\ra L^{p}(\partial U)
%\end{equation}
%such that
%\begin{enumerate}
%  \item $Tu =u\mid_{\partial U}$ if $u\in  W^{1,p}(U)\cap C\left(\bar{U}\right)$,
%  \item $\left\|Tu\right\|_{L^{p}(\partial U)}\ls C \left\|u\right\|_{W^{1,p}(U)}$,
%\end{enumerate}
%for each $u\in W^{1,p}(U)$, with the constant $C$ depending only on $p$ and $U$.
%\end{theorem}

We provide an overview of the fundamental theory concerning stochastic analysis.
Let $E$ be a separable Banach space, and $\mathscr{B}(E)$ be the $\sigma$-field of its Borel subsets, respectively. Let $\left(\Omega, \mathcal{F}, \mathbb{P}\right)$ be a stochastic basis. A filtration $\mathcal{F}=\left(\mathcal{F}_{t}\right)_{t \in \mathbf{T}}$ is a family of $\sigma$-algebras on $\Omega$ indexed by $\mathbf{T}$ such that $\mathcal{F}_{s} \subseteq$ $\mathcal{F}_{t} \subseteq \mathcal{F}$, $s \leq t$, $s, t \in \mathbf{T}$. $\left(\Omega, \mathcal{F}, \mathbb{P}\right)$ is also called a filtered space.
 We first list some definitions.
\begin{enumerate}
  \item [1.] {\bf $E$-valued random variables. \cite{Da-Prato-Zabczyk2014}}
 For $(\Omega, \mathscr{F})$ and $(E, \mathscr{E})$ being two measurable spaces, a mapping $X$ from $\Omega$ into $E$, such that the set $\{\omega \in \Omega: X(\omega) \in A\}=\{X \in A\}$ belongs to $\mathscr{F}$, for arbitrary $A \in \mathscr{E}$, is called a measurable mapping or a random variable from $(\Omega, \mathscr{F})$ into $(E, \mathscr{E})$ or an $E$-valued random variable.

  \item [2.]{\bf Strongly measurable operator valued random variables. \cite{Da-Prato-Zabczyk2014}} Let $\mathcal{U}$ and $\mathcal{H}$ be two separable Hilbert spaces which can be infinite-dimensional, and denote by $L(\mathcal{U}, \mathcal{H})$ the set of all linear bounded operators from $\mathcal{U}$ into $\mathcal{H}$. A functional operator $\Psi(\cdot)$ from $\Omega$ into $L(\mathcal{U}, \mathcal{H})$ is said to be strongly measurable, if for arbitrary $X \in \mathcal{U}$ the function $\Psi(\cdot) X$ is measurable, as a mapping from $(\Omega, \mathscr{F})$ into $(\mathcal{H}, \mathscr{B}(\mathcal{H}))$. Let $\mathscr{L}$ be the smallest $\sigma$-field of subsets of $L(\mathcal{U}, \mathcal{H})$ containing all sets of the form
\begin{align}
\{\Psi \in L(\mathcal{U}, \mathcal{H}): \Psi X \in A\}, \quad X \in \mathcal{U}, ~A \in \mathscr{B}(\mathcal{H}).
\end{align}
Then, $\Psi: \Omega \rightarrow L(\mathcal{U}, \mathcal{H})$ is a strongly measurable mapping from $(\Omega, \mathscr{F})$ into $(L(\mathcal{U}, \mathcal{H}), \mathscr{L})$.
\item  [3.]{\bf Law of a random variable.} For an $E$-valued random variable
 $X:(\Omega, \mathcal{F}) \rightarrow (E, \mathscr{E})$, % We denote by $\sigma(\mathbf{U})$ the smallest $\sigma$ -field with respect to which U is measurable. More precisely,
%\begin{equation}
%\sigma(\mathbf{U}):=\{\{\omega \in \Omega ; \mathbf{U}(\omega) \in A\} ; A \in \mathcal{A}\}
%\end{equation}
%and $\sigma(\mathbf{U}) \subset \mathcal{F}$. In addition,
we denote by $\mathcal{L}[X]$ the law of $X$ on $E$, that is, $\mathcal{L}[X]$ is the probability measure on $(E, \mathscr{E})$ given by
\begin{equation}
\mathcal{L}[X](A)=\mathbb{P}[X \in A], \quad A \in \mathscr{E}.\\
\end{equation} %%P22 in Feireisl's book.
%In measure theory, a pushforward measure %(also push-forward, push-forward, or image measure)
%is obtained by using a measurable function, transferring a measure from one measurable space to another space.
%
%\item [13.] {\bf Dirac measure}. Let $(E, \mathscr{B}(E))$ be a measurable space. Given $x \in E$, the Dirac measure $\delta_x$ at $x$ is the measure defined by
%\begin{equation}
%\delta_x(A):= \begin{cases}1, & x \in A \\ 0, & x \notin A\end{cases}
%\end{equation}
%for each measurable set $A \subseteq E$. In this paper, there holds
%$$\delta_{\bar{\rho}}=\mathcal{L}[\bar{\rho}](A) = \mathbb{P}\left[\left\{\omega\in \Omega|\bar{\rho}(x)\in A \right\}\right]=1.$$

 \item [4.]%\begin{definition}
{\bf Stochastic process. \cite{Da-Prato-Zabczyk2014}}
A stochastic process $X$ is defined as an arbitrary family $X = \{X_t\}_{t\in \mathbf{T}}$ of $E$-valued random variables $X_t$, $t \in \mathbf{T}$. $X$ is also regarded as a mapping from $\Omega$ into a Banach space like $C([0, T] ; E)$ or $L^p=L^p(0, T ; E), 1 \leq p<+\infty$, by associating $\omega \in \Omega$ with the trajectory $X(\cdot, \omega)$.
\item [5.]{\bf Cylindrical Wiener Process valued in Hilbert space.} \cite{Da-Prato-Zabczyk2014} % An $\mathbb{R}^{m}$ -valued stochastic process $W$ is called an $\mathcal{F}$-adapted Wiener process, provided:
%\begin{itemize}
%  \item  $W(t)$ is $\mathcal{F}_{t}$-adapted;\quad  $W(0)=0, \quad  \mathbb{P}$ a.s.;\quad
%  \item $W$ has continuous trajectories: $t \mapsto W(t)$ is continuous $\mathbb{P}$ a.s.; \quad
%  \item $W$ has independent increments: $W(t)-W(s)$ is independent of $\mathcal{F}_{s}$ for all $0 \leq s \leq$ $t<\infty$.
%\end{itemize}
%Let $Q$ be a trace class nonnegative operator on a Hilbert space $\mathcal{U}$.
 A $\mathcal{U}$-valued stochastic process $W(t), t \geqslant 0$, is called a cylindrical Wiener process if:
\begin{itemize}
 \item  $W(0)=0$;
 \item $W$ has continuous trajectories;
 \item $W$ has independent increments;
 \item The distribution of $(W(t)-W(s))$ is $\mathscr{N}(0,(t-s)), \quad 0\ls s < t$.
\end{itemize}
%In the setting of this paper, $Q=I$ is the identity operator.

  \item [6.]%\begin{definition}
{\bf Adapted stochastic process.} A stochastic process $X$ is $\mathcal{F}$-adapted if $X_{t}$ is $\mathcal{F}_{t}$-measurable for every $t \in \mathbf{T}$; %Assume that stochastic process $X(t)$ is indexed by $\mathbf{T}$. $X$ is
% $\mathcal{F}$-adapted or simply adapted if $X(t)$ is $\mathcal{F}_{t}$-measurable for every $t \in \mathbf{T}$;
  %\end{definition}
%\begin{definition}

\item [7.]{\bf Martingale.} The process $X$ is called integrable provided $\mathbb{E}\left[\|X_t\|\right]<+\infty$ for every $t \in \mathbf{T}$. An integrable and adapted $E$-valued process $X_t, t \in \mathbf{T}$, is a martingale, if
    \begin{itemize}
      \item  $X$ is adapted;
      \item $X_{s}=\mathbb{E}\left[X_{t} \mid \mathcal{F}_{s}\right]$, for arbitrary $t, s \in \mathbf{T},~ 0\ls s \ls  t$.
    \end{itemize}

\item [8.]\label{stopping time def}{\bf Stopping time.} On $\left(\Omega, \mathcal{F}, \mathbb{P}\right)$, a random time is a measurable mapping $\tau: \Omega \rightarrow \mathbf{T} \cup \infty$. A random time is a stopping time, if $\{\tau \leq t\} \in \mathcal{F}_{t}$ for every $t \in \mathbf{T}$.
For a process $X$ and a subset $V$ of the state space, we define the hitting time of $X$ in $V$ as
\begin{align}
\tau_{V}(\omega)=\inf \left\{\left.t \in \mathbf{T}\right| X_{t}(\omega) \in V\right\}.
\end{align}
If $X$ is a continuous adapted process and $V$ is closed, then $\tau_{V}$ is a stopping time.

 \item [9.]%\begin{definition}
{\bf Modification.}
 A stochastic process $Y$ is called a modification or a version of $X$ if
\begin{align}
\mathbb{P}(\omega \in \Omega: X(t, \omega) \neq Y(t, \omega))=0 \quad \text { for all } t \in \mathbf{T}.
\end{align}

\item [10.]%\begin{definition}
{\bf Progressive measurability.} In $\left(\Omega,\mathcal{F},\mathbb{P} \right)$, stochastic process $X$ is progressively measurable or simply progressively measurable, if, for $\omega\in \Omega$, $(\omega, s) \mapsto X(s,\omega),~ s \ls t$ is $\mathcal{F}_{t} \otimes \mathscr{B}(\mathbf{T} \cap[0, t])$-measurable for every $t \in \mathbf{T}$.

\item [11.]%\begin{definition}
{\bf Progressive measurability of continuous functions.} Let $X(t), t \in[0, T]$, be a stochastically continuous and adapted process with values in a separable Banach space $E$. Then, $X$ has a progressively measurable modification.

  \item [12.]{\bf Cross quadratic variation.} Fixing a number $T>0$, we denote by $\mathcal{M}_T^2(E)$ the space of all $E$-valued continuous, square integrable martingales $M$, such that $M(0)=0$.
If $M \in \mathcal{M}_T^2\left(\mathbb{R}^1\right)$ then there exists a unique increasing predictable process $\langle M(\cdot)\rangle$, starting from $0$, such that the process
\begin{align}
M^2(t)-\langle M(\cdot)\rangle, \quad t \in[0, T]
\end{align}
is a continuous martingale. The process $\langle M(\cdot)\rangle$ is called the quadratic variation of $M$. If $M_1, M_2 \in \mathcal{M}_T^2\left(\mathbb{R}^1\right)$ then the process
\begin{align}
\left\langle M_1(t), M_2(t)\right\rangle=\frac{1}{4}\left[\left\langle\left(M_1+M_2\right)(t)\right\rangle-\left\langle\left(M_1-M_2\right)(t)\right\rangle\right]
\end{align}
is called the cross quadratic variation of $M_1, M_2$. It is the unique, predictable process with trajectories of bounded variation, starting from 0 such that
\begin{align}
M_1(t) M_2(t)-\left\langle M_1(t), M_2(t)\right\rangle, \quad t \in[0, T]
\end{align}
is a continuous martingale.\\
For $M\in\mathcal{M}_T^2(\mathcal{H})$, where $\mathcal{H}$ is Hilbert space,
the quadratic variation is defined by
\begin{align}
\langle M(t)\rangle=\sum_{i, j=1}^{\infty}\left\langle M_i(t), M_j(t)\right \rangle  e_i \otimes e_j , \quad t \in[0, T],
\end{align}
as an integrable adapted process, where $M_i(t)$ and $M_j(t)$ are in $\mathcal{M}_T^2\left(\mathbb{R}^1\right)$. If $a \in \mathcal{H}_1, b \in \mathcal{H}_2$, then $a \otimes b$ denotes a linear operator from $\mathcal{H}_2$ into $\mathcal{H}_1$ given by the formula
\begin{align}
(a \otimes b) x=a\langle b, x\rangle_{\mathcal{H}_2}, ~x \in \mathcal{H}_2.
\end{align}
%Moreover,
%\begin{align}
%\langle M(t), a\rangle\langle M(t), b\rangle-\langle\langle\langle M(t)\rangle\rangle a, b\rangle
%\end{align}
%is a continuous $\mathbb{P}$-a.s. a nonnegative martingale.
We define a cross quadratic variation for $M^1 \in \mathcal{M}_T^2\left(\mathcal{H}_1\right)$, $M^2 \in \mathcal{M}_T^2\left(\mathcal{H}_2\right)$ where $\mathcal{H}_1$ and $\mathcal{H}_2$ are two Hilbert spaces. Namely we define
\begin{align}
\left\langle M^1(t), M^2(t)\right\rangle=\sum_{i, j=1}^{\infty}\left\langle M_i^1(t), M_j^2(t)\right\rangle e_i^1 \otimes e_j^2, \quad t \in[0, T],
\end{align}
where $\left\{e_i^1\right\}$ and $\left\{e_j^2\right\}$ are complete orthonormal bases in $\mathcal{H}_1$ and $\mathcal{H}_2$ respectively.

\item [13.]{\bf Stochastic integral.} Let $W$ be the Wiener process. Let $\Psi(t), t \in [0, T ]$, be a measurable Hilbert--Schmidt operators in $L(\mathcal{U}, \mathcal{H})$, which is set in the space $\mathcal{L}_{2}$ such that
    \begin{align}
    \mathbb{E}\left[\int_{0}^{t}\|\Psi(s)\|_{\mathcal{L}_{2}}^2  \d s\right]:=\mathbb{E} \int_0^t \langle\Psi(s),\Psi^{\star}(s)\rangle_{\mathcal{H}}\d s <+\infty,
    \end{align}
where $\langle\cdot,\cdot\rangle_{\mathcal{H}}$ means the inner product in $\mathcal{H}$.
 For the stochastic integral $\int_{0}^{t}\Psi \d W$, there holds
\begin{equation}
\mathbb{E}\left[\left(\int_{0}^{t}\Psi \d W\right)^{2}\right]=\mathbb{E}\left[\int_{0}^{t}\|\Psi(s)\|_{\mathcal{L}_{2}}^2 \d s\right].
\end{equation}
Furthermore, the following properties hold
\begin{itemize}
                                             \item Linearity: $\int(a \Psi_{1}+b \Psi_{2}) \d W=a \int \Psi_{1} \d W+b \int \Psi_{2} \d W$ for constants $a$ and $b$;
                                             \item Stopping property: $\int 1_{\{\cdot \leq \tau\}} \Psi \d W=\int \Psi \d M^{\tau}=\int_{0}^{\cdot \wedge \tau} \Psi \d W$;
                                             \item It\^o-isometry: for every $t$,
\begin{equation}
\mathbb{E}\left[\left(\int_{0}^{t} \Psi \d W\right)^{2}\right]=\mathbb{E}\left[\int_{0}^{t} \|\Psi(s)\|_{\mathcal{L}_{2}}^2 \d s\right].
\end{equation}
\end{itemize}

\end{enumerate}
\smallskip

 We list some important theorems in stochastic analysis.

\begin{enumerate}
\item [1.]{\bf It\^o's formula.} \cite{Ito1944,Da-Prato-Zabczyk2014} Assume that $\Psi$ is an $\mathcal{L}_{2}$-valued process stochastically integrable in $[0, T], \varphi$ being a $\mathcal{H}$-valued predictable process Bochner integrable on $[0, T], \mathbb{P}$-a.s., and $X(0)$ being a $\mathscr{F}_{0}$-measurable $\mathcal{\mathcal{H}}$-valued random variable. Then the following process
\begin{align}\label{form of X}
X(t)=X(0)+\int_0^t \varphi(s) d s+\int_0^t \Psi(s) \d W(s), \quad t \in[0, T]
\end{align}
is well defined. Assume that a function $F:[0, T] \times \mathcal{H} \rightarrow \mathbb{R}^1$ and its partial derivatives $F_t, F_x, F_{x x}$, are uniformly continuous on bounded subsets of $[0, T] \times \mathcal{H}$. Under the above conditions, $\mathbb{P}$-a.s., for all $t \in[0, T]$,
\begin{align}
~\qquad  F(t, X(t))= & F(0, X(0))+\int_0^t \left\langle F_x(s, X(s)), \Psi(s) \d W(s)\right\rangle_{\mathcal{H}} \\
& +\int_0^t\left\{F_t(s, X(s))+\left\langle F_x(s, X(s)), \varphi(s)\right\rangle_{\mathcal{H}}
 +\frac{1}{2} F_{x x}(s, X(s))\|\Psi(s)\|_{\mathcal{L}_{2}}^2 \right\} \d s. \notag
\end{align}
Applying the above formula for $F=\langle x, x \rangle_{\mathcal{H}}$, we have It\^o's formula for $\langle X, X \rangle_{\mathcal{H}}$. Then by
\begin{align}
\langle X, Y \rangle_{\mathcal{H}} =\frac{\langle X+Y, X+Y \rangle_{\mathcal{H}} -\langle X-Y, X-Y \rangle_{\mathcal{H}}}{4}
\end{align}
in Hilbert space,
 the following It\^o's formula holds for $X$ and $Y$ in form of \eqref{form of X},
\begin{equation}\label{Ito for XY}
\begin{aligned}
\langle X, Y \rangle_{\mathcal{H}} &= \langle X_{0}, Y_{0}\rangle_{\mathcal{H}} +\int \langle X, \d Y \rangle_{\mathcal{H}}\d s+ \int  \langle Y, \d X \rangle_{\mathcal{H}}\d s+\int \d\left\langle~\langle X,Y \rangle,  \langle X,Y \rangle~\right\rangle_{\mathcal{H}}^{\frac{1}{2}}\\
&=\langle X_{0}, Y_{0}\rangle_{\mathcal{H}} +\int \langle X, \d Y \rangle_{\mathcal{H}}\d s+ \int  \langle Y, \d X \d s \rangle_{\mathcal{H}}+\left\langle~\langle X,Y \rangle,  \langle X,Y \rangle~\right\rangle_{\mathcal{H}}^{\frac{1}{2}},\\
\end{aligned}
\end{equation}
where $\langle X,Y\rangle$ means the cross quadratic  variation of $X$ and $Y$ defined above. In this paper, for the convenience of notations, we still use  $\langle X,Y\rangle$ to denote $\left\langle~\langle X,Y \rangle,  \langle X,Y \rangle~\right\rangle_{\mathcal{H}}^{\frac{1}{2}}$.

\item [2.]{\bf Chebyshev's inequality}. Let $Y$ be a random variable in probability space $\left(\Omega, \mathcal{F}, \mathbb{P}\right)$, $\varepsilon>0$. For every $0<r<\infty$, Chebyshev's inequality reads
\begin{equation}
\mathbb{P}[\left\{|Y| \geq \varepsilon \right\}] \leq \frac{1}{\varepsilon^r}\mathbb{E}\left[|Y|^r\right] .
\end{equation}

\item [3.] {\bf Burkholder-Davis-Gundy's inequality}. \cite{BDG-inequality,Da-Prato-Zabczyk2014} Let $M$ be a continuous local martingale in $\mathcal{H}$. Let $M^{\ast}=\max\limits_{0\ls s\ls t}|M(s)|$, $m \geqslant 1$. $\langle M\rangle_{T}$ denotes the quadratic variation stopped by $T$. Then, there exist constants $K^{m}$ and $K_{m}$ such that
\begin{equation}
 K_{m} \mathbb{E}\left[\left(\langle M\rangle_{T}\right)^{m}\right]\ls \mathbb{E}\left[\left(M^{\ast}_{T}\right)^{2m}\right]\ls K^{m} \mathbb{E}\left[\left(\langle M\rangle_{T}\right)^{m}\right],
\end{equation}
for every stopping time $T$. For $m\geqslant  1$, $K^{m}=\left(\frac{2m}{2m-1}\right)^{\frac{2m(2m-2)}{2}}$, which is equivalent to $e^{m}$ as $m\ra \infty$.
Specifically, for every $m \geqslant 1$, and for every $t \geqslant 0$, there holds
\begin{equation}
\mathbb{E}\left[ \sup _{s \in[0, t]}\left|\int_0^t \Psi(s) \d W(s)\right|^{2m} \right]\leq  K^{m}\left(\mathbb{E}\left[\int_0^{t} \|\Psi(s)\|_{\mathcal{L}_{2}}^2\d s\right]\right)^{m}
\end{equation}

\item [4.] {\bf Stochastic Fubini theorem}. Assume that $(E, \mathscr{E})$ is a measurable space and let
$$
\Psi:(t, \omega, x) \rightarrow \Psi(t, \omega, x)
$$
be a measurable mapping from $\left(\Omega_T \times E, \mathscr{B}(\Omega_T) \times \mathscr{B}(E)\right)$ into $\left(\mathcal{L}^{2}, \mathscr{B}\left(\mathcal{L}^{2}\right)\right)$. Moreover, assume that
\begin{equation}
\int_E \left[\mathbb{E} \int_0^T \langle\Psi(s),\Psi^{\star}(s)\rangle_{\mathcal{H}}\d t\right]^{\frac{1}{2}} \mu(\d x)<+\infty,
\end{equation}
then $\mathbb{P}$-a.s. there holds
\begin{equation}
\int_E\left[\int_0^T \Psi(t, x) \d W(t)\right] \mu(\d x)=\int_0^T\left[\int_E \Psi(t, x) \mu(\d x)\right] \d W(t).
\end{equation}

% \item [5.]\label{Centov thm}{\bf Kolmogorov-Centov's continuity theorem.} \cite{Karatzas1988,Da-Prato-Zabczyk2014} Let $(\Omega, \mathcal{F}, \mathbb{P})$ be a probability space and $\bar{X}$ a process on $[0, T]$ with values in a complete metric space $(E, \mathscr{E})$. Suppose that
%\begin{equation}
%\mathbb{E}\left[\left|\bar{X}_{t}-\bar{X}_{s}\right|^{a}\right] \leq C|t-s|^{1+b},
%\end{equation}
%for every $s<t \leq T$ and some strictly positive constants $a, b$ and $C$. Then $\bar{X}$ admits a continuous modification $X$, $\mathbb{P}\left[\left\{X_{t}=\bar{X}_{t}\right\}\right]=1$ for every $t$, and $X$ is locally H\"older continuous for every exponent $0<\gamma<\frac{b }{a},$ namely,
%\begin{equation}
%\mathbb{P}\left[\left\{\omega: \sum_{0<t-s<h(\omega), t, s \leq T} \frac{\left|X_{t}(\omega)-X_{s}(\omega)\right|}{|t-s|^{\gamma}} \leq \delta\right\}\right]=1,
%\end{equation}
%where $h(\omega)$ is an strictly positive random variable a.s., and the constant satisfies $\delta>0$.
\end{enumerate}

\smallskip
\smallskip
  
{\bf Acknowledgements.} The research of Y. Li was supported in part
by National Natural Science Foundation of China under grants 12371221, 12161141004, and 11831011. Y. Li was also grateful to the supports by the Fundamental Research Funds for the Central Universities and Shanghai Frontiers Science Center of Modern Analysis. The research of M. Mei was supported by NSERC of Canada grant RGPIN 2022-03374 and NNSFC of China grant W2431005.

%\bibliographystyle{plain}
%\bibliography{stoincominhomoNSref}

% %It's worth noticing that the noise for this system is a kind of addictive noise, i.e., $\Phi$ is independent of $\u$. Multiplicative noise $\Phi$ could not meet the condition for $\Phi$ because we do not have the priori-estimate of $\u$. \par
%
%
%%\bibliographystyle{unsrt}%按引用顺序排
\bibliographystyle{plain}%按人名字母顺序排

\end{document}